\documentclass[11pt,a4paper]{article}
\usepackage{amssymb,amsmath,amscd,amsthm,amsxtra,latexsym}
\theoremstyle{plain}
\newtheorem{theorem}{Theorem}
\newtheorem{proposition}{Proposition}
\newtheorem{lemma}{Lemma}
\newtheorem{corollary}{Corollary}

\theoremstyle{remark}
\newtheorem{definition}{Definition}
\newtheorem{remark}{Remark}
\newtheorem{example}{Example}
\newtheorem*{question}{Question}

\newcommand{\SO}{\mathrm{SO}}
\newcommand{\rmO}{\mathrm{O}}
\newcommand{\SU}{\mathrm{SU}}

\newcommand{\so}{\mathfrak{so}}
\newcommand{\su}{\mathfrak{su}}

\newcommand{\tr}{\mathfrak{tr}}

\newcommand{\Sym}{\mathrm{Sym}}
\newcommand{\Menge}[2]{\{\,#1\,|\,#2\,\}}
\newcommand{\g}[2]{<#1,#2>}
\newcommand{\R}{\mathrm{I\!R}}
\newcommand{\C}{\mathbb{C}}

\newcommand{\Z}{\mathbb{Z}}
\newcommand{\frakC}{\mathfrak{C}}

\newcommand{\frakg}{\mathfrak{g}}
\newcommand{\frakh}{\mathfrak{h}}
\newcommand{\fraki}{\mathfrak{i}}
\newcommand{\frakk}{\mathfrak{k}}
\newcommand{\frakp}{\mathfrak{p}}

\newcommand{\osc}{\mathcal{O}}
\newcommand{\Id}{\mathrm{Id}}
\newcommand{\Iso}{\mathrm{I}}

\newcommand{\End}{\mathrm{End}}
\newcommand{\trace}{\mathrm{trace}}

\renewcommand{\i}{\mathrm{i}}

\newcommand{\bbH}{\mathbb{H}}
\newcommand{\Kern}{\mathrm{Kern}}
\newcommand{\Spann}[2]{\{#1\big|#2\}_{\scriptstyle\R} }\,

\newcommand{\fetth}{\boldsymbol{h}}

\newcommand{\Hom}{\mathrm{Hom}}

\newcommand{\rank}{\mathrm{rank}}
\newcommand{\rmG}{\mathrm{G}}
\newcommand{\rmH}{\mathrm{H}}
\newcommand{\rmS}{\mathrm{S}}
\newcommand{\rmU}{\mathrm{U}}

\newcommand{\rmP}{\mathrm{P}}
\newcommand{\bbK}{\mathbb{K}}
\newcommand{\e}{\mathrm{e}}

\newcommand{\scrR}{\mathcal{R}}
\newcommand{\scrU}{\mathcal{U}}
\newcommand{\fraku}{\mathfrak{u}}

\title{Parallel submanifolds of the real 2-Grassmannian}
\author{Tillmann Jentsch}
\begin{document}
\sloppy

\maketitle
\begin{abstract}
A submanifold of a Riemannian symmetric space 
is called parallel if its second fundamental form is a parallel section of the appropriate tensor bundle.
We classify parallel submanifolds of the Grassmannian $\rmG^+_2(\R^{n+2})$ 
which parameterizes the oriented 2-planes of the Euclidean space $\R^{n+2}$\,. Our main result states that
every complete parallel submanifold of $\rmG^+_2(\R^{n+2})$\,, which is not a
curve, is contained in some totally geodesic submanifold as a symmetric submanifold. This result holds also
if the ambient space is the non-compact dual of $\rmG^+_2(\R^{n+2})$\,. 
\end{abstract}

\section{Introduction}
Let $N$ be a Riemannian symmetric space. A submanifold of $N$ is called {\em parallel} if the second fundamental form is parallel.
D.~Ferus~\cite{Fe1} has shown that every compact parallel submanifold of a Euclidean space is a special orbit of some
s-representation, called a {\em symmetric R-space}. In particular, 
such a submanifold is invariant under the reflections in its affine 
normal spaces which means that it is {\em (extrinsically) symmetric}. More generally, every 
complete parallel submanifold of a space form has this property (see~\cite{BR,Fe2,Str,Ta}).
Note, this should be seen as an extrinsic analog of the following well known fact: every complete and simply
connected Riemannian manifold with parallel curvature tensor is already a symmetric
space.

More generally, symmetric submanifolds of Riemannian symmetric spaces
were studied and classified by H.~Naitoh and others, see~\cite[Ch.~9.3]{BCO}. 
These submanifolds are parallel and intrinsically symmetric (in particular, the induced Riemannian metric is complete), 
but not every complete parallel submanifold is extrinsically symmetric unless
the ambient space is a space form. Nevertheless, in the other simply connected rank-one spaces 
(i.e.\ the projective spaces over the complex numbers or the quaternions, the Cayley plane, and their non-compact duals),
there is still a close correspondence between parallel and symmetric submanifolds. Namely, it turns out 
that every complete parallel submanifold, which is not a curve, is contained in some totally geodesic
submanifold as a symmetric submanifold (see~\cite[Ch.~9.4]{BCO}). Further, recall that a submanifold is called {\em full} if
it is not contained in any proper totally geodesic submanifold. In particular,
in a simply connected rank-one space, the previous result implies that every full complete parallel submanifold, which is not a
curve, is a symmetric submanifold. 

However, in symmetric spaces of higher rank, parallel submanifolds are not well understood yet. 
Note, here the situation becomes more involved, since already the classification 
of the totally geodesic submanifolds is a non-trivial problem. Hence, it is an interesting 
fact that at least for the rank-two symmetric spaces the
totally geodesic submanifolds are well known due to B.-Y.~Chen/T.~Nagano~\cite{chen-naganoI,chen-naganoII}\footnote{However, the claimed
classification of totally geodesic submanifolds of $\rmG^+_2(\R^{n+2})$
from~\cite{chen-naganoI} is incomplete.} 
and S.~Klein~\cite{sebastian1,sebastian2,sebastian3,sebastian4,sebastian5} using different methods.  
Thus, it is natural to ask, more generally, for the classification of parallel submanifolds in
these ambient spaces. 

In this article, we consider parallel submanifolds of the Grassmannian $\rmG^+_2(\R^{n+2})$ -- 
which parameterizes the oriented 2-planes of the Euclidean space
$\R^{n+2}$ -- and its non-compact dual, the symmetric space $\rmG^+_2(\R^{n+2})^*$\,, i.e.\
the Grassmannian of time-like 2-planes in the pseudo Euclidean space $\R^{n,2}$ equipped with
the indefinite metric $d x_1^2+\cdots +d x_n^2-d x_{n+1}^2-d x_{n+2}^2$\,. Note,
these are simply connected symmetric spaces of rank two if $n\geq 2$\,.

\bigskip
\begin{theorem}[Main Theorem]\label{th:main}
If $M$ is a complete parallel submanifold of the Grassmannian
$\rmG^+_2(\R^{n+2})$ with $\dim(M)\geq 2$\,, then there exists a totally geodesic submanifold $\bar M\subset \rmG^+_2(\R^{n+2})$ 
such that $M$ is a symmetric submanifold of $\bar M$\,. 
In particular, every full complete parallel submanifold of
$\rmG_2^+(\R^{n+2})$\,, which is not a curve, is a symmetric submanifold. The analogous result holds for
ambient space $\rmG^+_2(\R^{n+2})^*$\,.
\end{theorem}

We also obtain the classification of higher-dimensional parallel submanifolds in a product of two
Euclidean spheres or two real hyperbolic spaces of equal curvature (see
Corollary~\ref{co:product}). Further, we conclude that every  
higher-dimensional complete parallel submanifold of $\rmG^+_2(\R^{n+2})$ is extrinsically homogeneous 
(see Corollary~\ref{co:homogeneous}). 

Here, we focus our attention on the real Grassmannian $\rmG^+_2(\R^{n+2})$ and
its non-compact dual. But we will also 
develop some general theory on the existence of 
parallel submanifolds in arbitrary Riemannian symmetric spaces. 
Amongst others, we will establish a splitting theorem for parallel submanifolds with curvature isotropic tangent spaces 
of maximal possible dimension in any symmetric space (of compact or non-compact type), see Corollary~\ref{co:circles}.

Hence, one may hope that it is also possible to classify the parallel submanifolds of the other rank-two symmetric spaces
(e.g.\ the Grassmannians of complex or quaternionic 2-planes). However, for the proof of Theorem~\ref{th:main} 
we use a ``case by case'' strategy and  it would be an interesting question whether some analogue of Theorem~\ref{th:main} remains true then.

\subsection{Overview}\label{se:overview}
We give an overview on the results presented in this article, an outline of the proof of Theorem~\ref{th:main} included. 
For a Riemannian symmetric space $N$ and a submanifold\footnote{\label{fn:notion}
We are implicitly dealing with isometric immersions defined from a connected Riemannian manifold $M$ into $N$\,. In particular, a ``submanifold'' needs not necessarily be regularly embedded. For example, it may have self-intersections.}
$M\subset N$\,, let $TM$\,, $\bot M$\,, $h:TM\times TM\to\bot M$ and $S:TM\times\bot M\to TM$
denote the tangent bundle, the normal bundle, 
the second fundamental form and the shape operator of $M$\,, respectively. 
Let $\nabla^M$ and $\nabla^N$ denote the Levi Civita connection of $M$ and $N$\,, respectively, 
and $\nabla^\bot$ be the usual connection on $\bot M$ (obtained by orthogonal projection of $\nabla^N\xi$ 
along $TM$ for every section $\xi$ of $\bot M$). Let $\Sym^2(TM,\bot M)$ 
denote the vector bundle whose sections are $\bot M$-valued symmetric bilinear maps on $TM$\,. 
Then there is a linear connection on $\Sym^2(TM,\bot M)$ induced by
$\nabla^M$ and $\nabla^\bot$ in a natural way, often called {\em Van der Waerden-Bortolotti connection}.

\bigskip
\begin{definition}\label{de:parallel} A submanifold $M\subset N$ is called
 \emph{parallel} if $h$ is a parallel section of $\Sym^2(TM,\bot M)$\,.
\end{definition}

\bigskip
\begin{example}\label{ex:circle}
A unit speed curve $c:J\to N$ is parallel if and only if it satisfies the equation
\begin{equation}
    \label{eq:circles}
\nabla^N_\partial\nabla^N_\partial\dot c=-\kappa^2 \dot c
\end{equation}
for some constant $\kappa\in\R$\,. For $\kappa=0$ these curves are geodesics;
otherwise, due to K.~Nomizu and K.~Yano~\cite{NY}, $c$ is called an {\em extrinsic circle}.
\end{example}
Recall that for every unit vector $x\in T_pN$ and every $\eta\in T_pN$ with $\eta\bot x$ 
there exists a unique unit speed curve $c$ satisfying~\eqref{eq:circles} with $c(0)=p$\,, $\dot c(0)=x$ and $\nabla_\partial \dot c(0)=\eta$\,.

\bigskip
\begin{example}\label{ex:factor}
Let $\bar M$ be a totally geodesic submanifold of $N$ (i.e. $h^{\bar M}=0$). A submanifold  
of $\bar M$ is parallel if and only if it is parallel in $N$\,.
\end{example}

\bigskip
\begin{definition}\label{de:extrinsically_symmetric}
A submanifold $M\subset N$ is called {\em (extrinsically) symmetric}\/ if $M$ is a symmetric space (whose
geodesic symmetries are denoted by $\sigma^M_p$\,, where $p$ ranges over $M$) and for every point $p\in M$ there
exists an involutive isometry $\sigma^\bot_p$ of $N$ such that
\begin{itemize}
\item $\sigma^\bot_p(M)=M$\,,
\item $\sigma^\bot_p|M=\sigma_p^M$\,,
\item the differential $T_p\sigma^\bot_p$ is the reflection
 in the normal space $\bot_pM$\,.
\end{itemize}
\end{definition}

As mentioned already before, every symmetric submanifold is parallel.
However, in the situation of Example~\ref{ex:factor}, we do not necessarily obtain a symmetric submanifold of
$N$ even if $M$ is symmetric in $\bar M$\,.

Let $M$ be a parallel submanifold of the symmetric space $N$ and consider the 
linear space $\bot^1_pM:=\{h(x,y)|x,y\in W\}_\R$ called the {\em first normal space} at $p$\,.

\bigskip
\begin{question}\label{question}
Given a pair of linear spaces $(W,U)$
both contained in $T_pN$ and such that $W\bot U$\,, does there exist some parallel submanifold $M$ through $p$ with 
$W=T_pM$ and $U=\bot^1_pM$? In particular, are there natural obstructions
against the existence of such a submanifold?
\end{question}

Let $R^N$ denote the curvature tensor of $N$ and recall that a
linear subspace $V\subset T_pN$ is called {\em curvature invariant} if $R^N(V\times V\times V)\subset V$ holds. 
It is well known that $T_pM$ is a curvature invariant subspace of $T_pN$ for
every parallel submanifold $M$\,. In Section~\ref{se:definition_of_cip}, we will show that also $\bot^1_pM$ is
curvature invariant. Moreover, the curvature endomorphisms of 
$T_pN$ generated by $T_pM$ leave $\bot^1_pM$ invariant and
vice versa. This means that $(T_pM,\bot^1_pM)$ is an {\em orthogonal curvature
  invariant pair}, see Definition~\ref{de:cip} and Proposition~\ref{p:cip}. As a first illustration of this concept,
we classify the orthogonal curvature invariant pairs $(W,U)$  of the complex
projective space $\C\rmP^n$\,, see Example~\ref{ex:CPm}.
We observe that here the linear space $W\oplus U$ is complex or totally real (in
particular, curvature invariant) unless $\dim(W)=1$\,. Hence, following the
proof of Theorem~\ref{th:main} given below, we obtain the well known result 
that the analogue of Theorem~\ref{th:main} is true for ambient space $\C\rmP^n$\,.

In Section~\ref{se:cip}, we will determine the orthogonal curvature invariant
pairs of $N:=\rmG^+_2(\R^{n+2})$\,. Our result is summarized in Table~\ref{table}.
Note, even if we assume additionally that $\dim(W)\geq 2$\,, there do exist certain
orthogonal curvature invariant pairs $(W,U)$ for which the linear space $W\oplus U$ is not curvature
invariant (in contrast to the situation where the ambient space is $\C\rmP^n$\,, 
see above). Hence, at least at the level of curvature invariant pairs, we can not yet give the proof of Theorem~\ref{th:main}.

Therefore, it still remains to decide whether there actually exists some
parallel submanifold $M\subset N$ such that $(W,U)=(T_pM,\bot^1_pM)$ in which case 
the orthogonal curvature invariant pair $(W,U)$ will be called {\em
  integrable}. In Section~\ref{se:integrability}, by means of a case by case analysis, we
will show that if $(W,U)$ is integrable and $\dim(W)\geq 2$\,, then 
the linear space $W\oplus U$ is curvature invariant. For this, we will
need some more intrinsic properties of the second fundamental 
form of a parallel submanifold of a symmetric space which are derived in Section~\ref{se:general_results}.

Further, note that (orthogonal) curvature invariant pairs of $N$ and $N^*$\,, respectively, are the same.\footnote{However, there is no duality between parallel submanifolds of $N$ and $N^*$\,, respectively.  This is due to the semi-parallelity condition on the second fundamental form (see~\eqref{eq:sp} with $R=R^N$) 
which is not preserved if one changes the sign of $R^N$\,. For example, any 
complex parallel submanifold of the complex hyperbolic space $\C\rmH^n$ is totally geodesic (see~\cite[Theorem~9.4.3]{BCO}) whereas this is not true for ambient space $\C\rmP^n$\,, the  complex projective 
space (see~\cite[Theorem~9.3.5]{BCO}). Since both $\C\rmP^n$ and $\C\rmH^n$ are totally geodesically embedded in $N$ and $N^*$\,, respectively, this gives counter-examples also in our case.}
Moreover, it turns out that all arguments from Section~\ref{se:integrability} remain valid for ambient space $N^*$\,.

\begin{proof}[Proof of Theorem~\ref{th:main}]
We can assume that $n\geq 2$\,. Fix some $p\in M$\,. Then $(T_pM,\bot^1_pM)$ is an integrable curvature invariant pair.
Using the results  mentioned before, we conclude that the {\em second osculating space} 
$\osc_pM:=T_pM\oplus \bot^1_pM$ is a curvature invariant subspace of $T_pN$\,. 
Let $\exp^N:TN\to N$ denote the exponential spray. It follows from a
result of P.~Dombrowski~\cite{D} that $\bar M:=\exp^N(\osc_pM)$ is a totally geodesic submanifold 
of $N$ such that $M\subset\bar M$ (``reduction of the codimension''). 
By construction, $\bot^1_qM=\bot_qM$ for all $q\in M$ where the normal spaces
are taken in $T\bar M$\,, i.e. $M$ is a {\em 1-full} complete parallel submanifold of $\bar M$\,. Thus we conclude
from Corollary~\ref{co:1-full} (see below) that $M$ is even a symmetric submanifold of $\bar
M$\,. The same arguments apply to ambient space $N^*$\,.\end{proof}
\paragraph{}

We consider the Riemannian product $\rmS^k\times\rmS^\ell$ of two Euclidean unit-spheres with $k+\ell=n\geq 2$
and $k\leq \ell$\,. Set $0_k:=(0,\cdots,0)\in\R^k$ for $k\geq 0$\,.
The map $\tau:\rmS^k\times\rmS^\ell\to \rmG^+_2(\R^{n+2}), (p,q)\mapsto
\{(p,0_{\ell+1}),(0_{k+1},q)\}_\R$ defines a 2-fold isometric covering onto a totally 
geodesic submanifold of $\rmG^+_2(\R^{n+2})$\,,
see~\cite{sebastian1},\cite{sebastian4}. 
Hence every parallel submanifold of $\rmS^k\times\rmS^\ell$ is also parallel in $\rmG^+_2(\R^{n+2})$\,.
Further, the totally geodesic  embedding $\iota_{k,\ell}:\rmS^k\to \rmS^k\times\rmS^\ell, p\mapsto (p,p)$ 
is a homothety onto its image by a factor $\sqrt{2}$\,. 

\bigskip
\begin{corollary}[Parallel submanifolds of $\rmS^k\times\rmS^\ell$]\label{co:product}
Every complete parallel submanifold $M\subset \rmS^k\times\rmS^\ell$ 
with $\dim(M)\geq 2$ is a product, $M=M_1\times M_2$\,, 
of two symmetric submanifolds $M_1\subset\rmS^k$ and $M_2\subset\rmS^\ell$\,, 
or is conjugate to a symmetric submanifold of $\iota_{k,\ell}(\rmS^k)$ via some isometry of $\rmS^k\times\rmS^\ell$\,.
In the first case, $M$ is a symmetric submanifold of $\rmS^k\times\rmS^\ell$\,. 
In the second case, $M$ is not symmetric in $\rmS^k\times\rmS^\ell$ unless $k=\ell$ and $M\cong \iota_{k,\ell}(\rmS^k)$\,. 
The analogous result holds for complete parallel submanifolds of $\rmH^k\times\rmH^\ell$\,,
the Riemannian product of two hyperbolic spaces 
of sectional curvature $-1$ (for $2\leq k\leq \ell$), or of $\R\times\rmH^\ell$\,, respectively.
\end{corollary}
\begin{proof}
Let $M$ be a parallel submanifold of $\tilde N:=\rmS^k\times\rmS^\ell$ through $(p,q)$\,. 
Then $\tau(M)$ is parallel in $N:=\rmG^+_2(\R^{n+2})$\,. 
Hence, according to Theorem~\ref{th:main} and its proof,
the second osculating space $V:=T_{(p,q)}M\oplus\bot^1_{(p,q)}M$ is a curvature invariant 
subspace of both $T_{\tau(p,q)}N$ and $T_{(p,q)}\tilde N$ such that $M$ is contained in 
the totally geodesic submanifold $\exp^N(V)$ as a symmetric submanifold. 
Further, using the classification of curvature invariant subspaces of
$T_{(p,q)}N$ (see Theorem~\ref{th:ci} below), we obtain that there are only 
two possibilities:

\begin{itemize}
\item  we have $V=W_1\oplus W_2$ 
where $W_1$ and $W_2$ are $i$- and $j$-dimensional subspaces of $T_{p}\rmS^k$ and $T_{q}\rmS^\ell$\,, respectively (Type $(tr_{i,j})$).
Hence, $M$ is contained in the totally geodesic submanifold $\bar M:=\exp^N(W_1)\times\exp^N(W_2)$ 
-- where, of course, the
factors $\exp^N(W_1)$ and $\exp^N(W_2)$ are Euclidean unit-spheres, too.
If $\bar M$ is the product of two great circles in 
$\rmS^k$ and $\rmS^\ell$\,, respectively, then $\dim(\bar M)=2$ and $M=\bar M$\,. Otherwise, at least one of the factors
of $\bar M$ is a higher-dimensional Euclidean sphere. It follows from a result of Naitoh 
(see Theorem~\ref{th:products} below) that $M=M'\times M''$ where $M'\subset \exp^N(W_1)$ 
and $M''\subset \exp^N(W_2)$ are symmetric submanifolds. 
Anyway, we obtain that $M=M'\times M''$ where $M'\subset\rmS^k$ and $M''\subset \rmS^\ell$ 
are symmetric submanifolds. Therefore, the product $M'\times M''$ is symmetric in $\tilde N$\,.

\item there exists an $i$-dimensional linear space $W_0'\subset T_{p}\rmS^k$ and some linear isometry $I'$ defined from $W_0'$ onto its image $I'(W_0')\subset T_{q}\rmS^\ell$ such that 
$V=\Menge{(v,I'\,v)}{v\in W_0'}$ (Type $(tr_i')$). Then, up to an isometry of $\tilde N$\,, we can assume that 
$M$ is a complete parallel submanifold of the space form $\iota_{k,\ell}(\rmS^k)$\,, i.e.\ a symmetric submanifold. 
It follows from Theorem~\ref{th:products} that $M$ is not symmetric in $\tilde N$ unless $M$ is totally geodesic.
Moreover, a totally geodesic submanifold of $\iota_{k,\ell}(\rmS^k)$ is symmetric in $\tilde N$  
if and only if  the normal spaces of $\iota_{k,\ell}(\rmS^k)$ are curvature invariant (cf.~\cite[Ch.~9.3]{BCO})  
which is given only for $M\cong\iota_{k,\ell}(\rmS^k)$ and $k=\ell$\,.
\end{itemize}
The hyperbolic case is handled in a similar way. Our result follows. 
\end{proof}

\paragraph{}
Recall that a submanifold $M\subset N$ is called {\em extrinsically homogeneous} 
if a suitable subgroup of the isometry group $\Iso(N)$ acts transitively on $M$\,.
In~\cite{J2,J3}, we dealt with the question whether a complete parallel submanifold of a 
symmetric space of compact or non-compact type is automatically extrinsically homogeneous.
It follows a priori from~\cite[Corollary~1.4]{J3} that every complete parallel 
submanifold $M$ of a simply connected compact or non-compact rank-two symmetric 
space $N$ without Euclidean factor (e.g. $N=\rmG^+_2(\R^{n+2})$ or $N=\rmG^+_2(\R^{n+2})^*$) 
is extrinsically homogeneous provided that the
Riemannian space $M$ does not split of (not even locally) a factor of dimension one or
two (e.g. $M$ is locally irreducible and $\dim(M)\geq 3$).
Moreover, then $M$ has even {\em extrinsically homogeneous holonomy bundle}. 
The latter means the following: there exists a subgroup $G\subset \Iso(N)$ such that $g(M)=M$ for every $g\in G$ and $G|_M$ is  
the group which is generated by the {\em transvections} of $M$\,. Using
Theorem~\ref{th:main}, we can now prove a stronger result for $N=\rmG_2^+(\R^{n+2})$\,.

\bigskip
\begin{corollary}[Homogeneity of parallel submanifolds]\label{co:homogeneous}
Every complete parallel submanifold of $\rmG^+_2(\R^{n+2})$\,, which is not a
curve, has extrinsically homogeneous holonomy bundle. In particular, every such
submanifold is extrinsically homogeneous in $\rmG^+_2(\R^{n+2})$\,. This result holds also for ambient space $\rmG^+_2(\R^{n+2})^*$\,.
\end{corollary}  
\begin{proof}
Let $M$ be a complete parallel submanifold of $N:=\rmG_2^+(\R^{n+2})$ with $\dim(M)\geq 2$\,.
Then there exists a totally geodesic submanifold $\bar M\subset N$ such that $M$ is a symmetric submanifold of $\bar M$\,.
In particular, $\bar M$ is intrinsically a symmetric space. Furthermore, since the rank of $N$ is two, 
the rank of $\bar M$ is less than or equal to two. It follows immediately that there are no more 
than the following possibilities:

\begin{itemize}
\item the totally geodesic submanifold $\bar M$ is the 2-dimensional flat torus. 
Then we automatically have $M=\bar M$ (since $\dim(M)\geq 2$). Hence, we
have to show that the totally geodesic flat  $\bar M$ 
has extrinsically homogeneous holonomy bundle: let $\bar\fraki=\bar \frakk\oplus\bar \frakp$ 
and $\fraki=\frakk\oplus\frakp$ denote the Cartan decompositions of the 
Lie algebras of $\Iso(\bar M)$ and $\Iso(N)$\,, respectively. Then $[\bar \frakp,\bar \frakp]=\{0\}$\,, 
since $\bar M$ is flat.
Let $\bar G\subset \Iso(\bar M)$ denote the connected subgroup whose Lie algebra is 
$\bar \frakp$\,. Then $\bar G$ is the transvection group of $\bar M$\,.
Moreover, $\bar\frakp\subset\frakp$\,, because $\bar M$ is totally geodesic. 
Hence, we may take $G$ as the connected subgroup of $\Iso(N)$ whose  Lie algebra is $\bar\frakp$\,.

\item the totally geodesic submanifold $\bar M$ is locally the Riemannian product $\R\times\tilde M$ where $\tilde M$ 
is a locally irreducible symmetric space with $\dim(\tilde M)\geq 2$\,. Since $M\subset \bar M$ is symmetric, 
there exists a distinguished reflection $\sigma^\bot_p$ of $\bar M$ whose restriction to $M$ is the geodesic 
reflection in $p$ for every $p\in M$\,, see Definition~\ref{de:extrinsically_symmetric}. 
Therefore, these reflections generate a subgroup of $\Iso(\bar M)$  whose connected component acts transitively on $M$ and 
gives the full transvection group of $M$\,. 
Thus, it suffices to show that there exists a suitable subgroup of $\Iso(N)$ whose 
restriction to $\bar M$ is the connected component of $\Iso(\bar M)$\,:

let $\bar\fraki=\bar \frakk\oplus\bar \frakp$\,,  $\tilde\fraki=\tilde\frakk\oplus\tilde \frakp$ 
and  $\fraki=\frakk\oplus\frakp$ denote the Cartan decompositions of the 
Lie algebras of $\Iso(\bar M)$\,, $\Iso(\tilde M)$ and $\Iso(N)$\,, respectively.
Then $\bar\frakk=\tilde\frakk=[\tilde\frakp,\tilde\frakp]=[\bar\frakp,\bar\frakp]$\,, where the first and the last equality 
are related to the special product structure of $\bar M$ and the second one uses the fact 
that the Killing form of $\tilde \fraki$ is non-degenerate.
It follows that $\bar\fraki=[\bar\frakp,\bar\frakp]\oplus \bar\frakp$\,.
Moreover, we have $\bar\frakp\subset\frakp$\,, see above. Hence, every Killing 
vector field of $\bar M$ is the restriction of some Killing vector field of $N$\,.

\item the totally geodesic submanifold $\bar M$ is locally irreducible or locally the Riemannian product of two higher dimensional locally irreducible symmetric spaces:
then we have $\bar\fraki=[\bar\frakp,\bar\frakp]\oplus \bar\frakp$ because the Killing form of $\bar \fraki$ is non-degenerate. 
Hence we can use arguments given in the previous case.
\end{itemize}
\end{proof}
Note, in the previous theorems, the condition $\dim(M)\geq 2$ can not be ignored: 

consider the ambient space $\rmG_2^+(\R^4)$ which is isometric to $\rmS^2_{1/\sqrt{2}}\times\rmS^2_{1/\sqrt{2}}$\,. 
Here, a ``generic'' extrinsic circle is full but not extrinsically
homogeneous (e.g.\ not a symmetric submanifold), see~\cite{J2}, Example~1.9.

\section{Parallel submanifolds of symmetric spaces}
\label{se:general_results}
First, we solve the existence problem for parallel submanifolds of symmetric spaces by means of giving 
necessary and sufficient tensorial ``integrability conditions'' on the 2-jet (see Theorem~\ref{th:integrable}).\footnote{Note, such conditions were already claimed in~\cite{JR}.
However, the tensorial conditions stated in~\cite[Theorem~2]{JR} are not very strong.}
From this, we derive the fact (already mentioned before) that $(T_pM,\bot^1_pM)$ is a curvature invariant pair 
for every parallel submanifold $M$\,. Then we establish a necessary condition on the 2-jet of a parallel submanifold
which relates its integrability to the linearized isotropy representation of the ambient space 
(see Theorem~\ref{th:dec} and Corollary~\ref{co:dec}).
Some of the results mentioned so far were already obtained in~\cite{J1,J2}, however, for readers convenience, 
here we will derive them directly from the integrability conditions mentioned before.

Further, we give two results on the reduction of the codimension: 
for certain parallel submanifolds with one dimensional first normal spaces 
(see Proposition~\ref{p:sph}) and for parallel submanifolds with
curvature isotropic tangent spaces of maximal possible dimension 
(see Proposition~\ref{p:cfl} and Corollary~\ref{co:circles}). 
Note, whereas the first of these results is a straightforward 
generalization of a well known result on extrinsic spheres, the second one is apparently new.

We will also state a result of H.~Naitoh on symmetric submanifolds of product spaces (see Theorem~\ref{th:products}). 
This result was already used in the proof of Corollary~\ref{co:product}. Moreover,
we will need it again in order to show that certain curvature invariant pairs 
of Type $(tr_k',tr_k')$ are not integrable (cf.\ the proof of Corollary~\ref{co:kgeq3}).

\subsection{Existence of parallel submanifolds in symmetric spaces}
\label{se:existence}
It was first shown by W.~Str\"ubing~\cite{Str} that a parallel submanifold $M$ of
an arbitrary Riemannian manifold is uniquely
determined by its {\em 2-jet} $(T_pM,h_p)$ at some point $p\in M$\,. 
 Conversely, let a prescribed 2-jet $(W,h)$ at $p$
be given (i.e. $W\subset T_pN$ is a subspace and
$h:W\times W\to W^\bot$ is a symmetric bilinear map).
If there exists a parallel submanifold $M\subset N$ through $p$ such that
$(W,h)$ is the 2-jet of $M$\,, then $(W,h)$ will be called {\em integrable}.
Note, according to~\cite[Theorem 7]{JR}, for every integrable 2-jet, 
the corresponding parallel submanifold can be assumed to be complete.

Let $U$ be the subspace of $W^\bot$ which
is spanned by the image of $h$  and set $V:=W\oplus U$\,, i.e. $U$ and $V$ play the roles of the ``first normal
space'' and the ``second osculating space'', respectively.
Then the orthogonal splitting $V:=W\oplus U$ turns $\so(V)$ into a naturally
$\Z_2$-graded algebra $\so(V)=\so(V)_+\oplus\so(V)_-$
where $A\in\so(V)_+$ or $A\in\so(V)_-$ according to
whether $A$ respects the splitting $V=W\oplus U$ or $A(W)\subset U$ and $A(U)\subset W$\,.
Further, consider the linear map $\fetth:W\to\so(T_pN)$ given by
\begin{equation}\label{eq:fettb1}
\forall x,y\in W,\xi\in W^\bot:\fetth_x(y+\xi)=-S_\xi x+h(x,y)
\end{equation}
(where $S_\xi$ denotes the shape operator associated with $h$ for every
$\xi\in U$ in the usual way). Since $S_\xi=0$ holds for every $\xi\in W^\bot$ which is orthogonal to $U$\,,
we actually have
\begin{equation}\label{eq:fettb2}
\forall x\in W:\fetth_x\in\so(V)_-.
\end{equation}

\bigskip
\begin{definition}\label{de:sp}
Let a curvature like tensor $R$ on $T_pN$ and an $R$-invariant subspace $W$
of $T_pN$  (i.e. $R(W\times W\times W)\subset W$)  be given. A symmetric bilinear map $h:W\times W\to
W^\bot$ will be called {\em $R$-semi-parallel} if
\begin{equation}\label{eq:sp}
\fetth_{R_{x,y}z-[\fetth_x,\fetth_y]\,z}\,v=[R_{x,y}-[\fetth_x,\fetth_y],\fetth_z]\,v\
\end{equation}
holds for all $x,y,z\in W$ and $v\in T_pN$\,. 
Here $R_{u,v}:T_pN\to T_pN$ denotes the {\em curvature endomorphism} 
$R(u,v,\cdot)$ for all $u,v\in T_pN$\,. If $W$ is a
curvature invariant subspace of $T_pN$ and~\eqref{eq:sp} holds for $R=R^N_p$\,,
then $h$ is simply called {\em semi-parallel}.
\end{definition}
In the situation of Definition~\ref{de:sp}, it is easy to see that $h$ is $R$-semi-parallel 
if and only if~\eqref{eq:sp} holds for all $x,y,z\in W$ and $v\in V$\,.

Clearly, each linear map $A$ on $V$ induces an endomorphism $A\cdot$ on $\Lambda^2V$ 
by means of the usual rule of derivation, i.e.
$A\cdot u\wedge v=A\,u\wedge v+u\wedge A\,v$\,. Let $(A\cdot)^k$
denotes the $k$-th power of $A\cdot$ on $\Lambda^2V$\,. Similarly, $[A,\cdot]$ defines an endomorphism on $\so(V)$ whose
$k$-th power will be denoted by $[A,\cdot]^k$\,.
Furthermore, every curvature like tensor $R:T_pN\times T_pN\times T_pN\to T_pN$
can be seen as a linear map $R:\Lambda^2 T_pN\to\so(V)$ characterized by
$R(u\wedge v)= R_{u,v}$\,. The following theorem states the necessary and
sufficient ``integrability conditions'':\footnote{This result was also
  obtained in an unpublished paper by E.~Heintze.}

\bigskip
\begin{theorem}\label{th:integrable}
Let $N$ be a symmetric space. The 2-jet
$(W,h)$ is integrable if and only if the following conditions together hold:
\begin{itemize}
\item $W$ is a curvature invariant subspace of $T_pN$\,,
\item $h$ is semi-parallel,
\item we have
\begin{equation}\label{eq:cond2}
[\fetth_x,\cdot]^kR^N_{y,z}v=R^N((\fetth_x\cdot)^k y\wedge z)v
\end{equation}
for all $x,y,z\in W$\,, $k=1,2,3,4$ and each $v\in V$\,.
\end{itemize}
\end{theorem}
\begin{proof}
In order to apply the main result of~\cite{JR}, consider the space $\frakC$ of all curvature like tensors
on $T_pN$ and the affine subspace $\tilde\frakC\subset \frakC$ which consists, by definition,
of all curvature like tensors $R$ on $T_pN$ such that $W$ is $R$-invariant and $h$ is $R$-semi-parallel. 
Then we define the one-parameter subgroup $R_x(t)$ of curvature like tensor on $T_pN$ characterized by
\begin{equation}\label{eq:one-parameter subgroup}
\exp(t\,\fetth_x) R_x(t)(u,v,w)=R^N(\exp(t\,\fetth_x)u,\exp(t\,\fetth_x)v,\exp(t\,\fetth_x)w)
\end{equation}
for all $u,v,w\in T_pN$ and $x\in W$\,.
According to~\cite[Theorem~1 and Remark~2]{JR}, $(W,h)$ is integrable if and only if $R_x(t)\in\tilde\frakC$ 
for all $x\in W$ and $t\in\R$ (since $R^N$ is a parallel tensor). Moreover, if
$(W,h)$ is integrable, then one
can show that the function $t\mapsto R_x(t)(y,z,v)$ is constant for all
$x,y,z\in W$ and $v\in V$ (see~\cite[Example~3.7~(a) and
Lemma~3.8]{J1}). Conversely, if $R^N_p\in\tilde\frakC$ and $R_x(t)(y,z,v)$
is constant in $t$ for all $x,y,z\in W$ and $v\in V$\,, then $R_x(t)$ in
$\tilde\frakC$ for all $t$ by straightforward arguments. 

Let us assume that $(W,h)$ is integrable. Then the previous implies that  
\begin{equation}\label{eq:invariance}
\exp(t\,\fetth_x)R^N_{y,z}\exp(-t\,\fetth_x)v=R^N_{\exp(t\,\fetth_x)y,\exp(t\,\fetth_x)z}v
\end{equation}
Taking the derivatives of~\eqref{eq:invariance} with respect to $t$\,,
we now see that~\eqref{eq:cond2} holds for all $k\geq 1$\,.

Conversely, suppose that $R^N_p\in\tilde\frakC$ holds.
It suffices to show that~\eqref{eq:cond2} implies that the function $t\mapsto R_x(t)(y,z,v)$
is constant for all $x,y,z\in W$ and $v\in V$\,:

Put $A:=\fetth_x$\,, set $\Sigma:=\sum_{i=0}^3(A\cdot)^i (\Lambda^2W)$ and note that
\begin{align}
&A\cdot y\wedge z=Ay\wedge z+y\wedge Az\;,\\
&(A\cdot)^2 y\wedge z=A^2y\wedge z+2Ay\wedge Az+y\wedge A^2z\;,\\
&(A\cdot)^3 y\wedge z=A^3y\wedge z+3A^2y\wedge Az+3Ay\wedge A^2z+y\wedge A^3z\;,\\
&(A\cdot)^4 y\wedge z=A^4y\wedge z+4A^3y\wedge Az+6A^2y\wedge A^2z+4Ay\wedge A^3z+y\wedge A^4z
\end{align}
for all $y,z\in W$\,. Since $A^2W\subset W$\,, we hence see that 
$(A\cdot)^4 (\Lambda^2W)\subset \Lambda^2W+(A\cdot)^2(\Lambda^2W)$\,.
Therefore, $A\cdot \Sigma\subset \Sigma$ and, furthermore, since~\eqref{eq:cond2} holds for $k=1,2,3,4$\,,
the natural map $\Lambda^2 T_pN\to \so(T_pN), u\wedge v\mapsto R^N_{u,v}$ induces a linear map $\Sigma\to\so(V), \omega \mapsto R^N(\omega)|_V$ 
which is equivariant with respect to the linear actions
of the 1-dimensional Lie algebra $\R$ induced by $A\cdot $ and $[A,\cdot]$ on $\Sigma$ and $\so(V)$\,, respectively.
Switching to the level of one-parameter subgroups, we obtain that $R_x(t)(\omega) v$
is constant in $t$ for all $\omega\in\Sigma$ and $v\in V$\,, in particular $R_x(t)(y,z,v)$ is constant in $t$
for all $x,y,z\in W$\,, $v\in V$\,.
\end{proof}

\bigskip
\begin{remark}
In the situation of Theorem~\ref{th:integrable}, suppose that $(W,h)$ is integrable. Then we have
\begin{equation}\label{eq:cond3}
[\fetth_{x_1},\ldots[\fetth_{x_k},R^N_{y,z}]\ldots]|_V=R^N_{\fetth_{x_1}\cdot\, \cdots\,\cdot \fetth_{x_k}\cdot y\wedge z}|_V
\end{equation}
for all ${x_1},\ldots, x_k,y,z\in W$ with $k=1,2,\ldots$\,.
Note, here $x_i\neq x_j$ is possible.
\end{remark}
\begin{proof}
For Equation~\eqref{eq:cond3} with $k=1,2$ see~\cite[Lemma 3.9]{J1}. 
The proof for $k\geq 3$ is done in a similar fashion.
\end{proof}

\subsection{Curvature invariant pairs}\label{se:definition_of_cip}
Suppose that $(W,h)$ is an integrable 2-jet at $p$\,, set $U:=\Spann{h(x,y)}{x,y\in W}$ and $V:=W\oplus U$\,.
Then $W$ is a curvature invariant subspace of $T_pN$ and 
$h:W\times W\to W^\bot$ is a semi-parallel symmetric bilinear map, hence
\begin{equation}\label{eq:cip1}
R^N(W\times W\times W)\subset W\ \text{and}\ R^N(W\times W\times U)\subset U\;.
\end{equation}
In other words, $R^N_{x,y}(V)\subset V$ and $R^N_{x,y}|_V\in\so(V)_+$ for all $x,y\in W$\,.

Moreover, using~\eqref{eq:cond3} with $k=2$\,,
we obtain that
\begin{equation}\label{eq:ci}
R^N_{h(x,x),h(y,y)}|_V=[\fetth_{x},[\fetth_{y},R^N_{x,y}]]|_V+R^N_{S_{h(x,y)}x,y}|_V+R^N_{x,S_{h(y,y)}x}|_V
\end{equation}
for all $x,y\in W$\,.
Since r.h.s.\ of~\eqref{eq:ci} leaves $V$ invariant, the same is true for l.h.s.\ of~\eqref{eq:ci}.
Furthermore, using that $R^N_{x,y}|_V\in\so(V)_+$\,, Eq.~\ref{eq:fettb2} and the rules for $\Z_2$-graded Lie algebras,
we see that r.h.s.\ of~\eqref{eq:ci} defines an element of $\so(V)_+$\,. 
Hence the same is true for l.h.s of~\eqref{eq:ci}, too.
Finally, because $h$ is symmetric, $\Lambda^2(U)=\Spann{h(x,x)\wedge h(y,y)}{x,y\in W}$ holds. 
We conclude that~\eqref{eq:cip1} 
holds also with the roles of $W$ and $U$ interchanged, i.e.\ we have
\begin{equation}\label{eq:cip2}
R^N(U\times U\times U)\subset U\ \text{and}\ 
R^N(U\times U\times W)\subset W\;.
\end{equation}

\bigskip
\begin{definition}\label{de:cip}
Let subspaces $W, U$ of $T_pN$ be given.
We will call $(W,U)$ a curvature invariant pair if both~\eqref{eq:cip1} and~\eqref{eq:cip2} hold.
In particular, then $W$ and $U$ both are curvature invariant subspaces of $T_pN$\,.
If additionally $W\bot U$\,, then $(W,U)$ is called an orthogonal curvature invariant pair. 
\end{definition}
We obtain the first criterion matching on the question posed in Section~\ref{se:overview} (cf.~\cite[Corollary~13]{J1}):

\bigskip
\begin{proposition}\label{p:cip}
Let  $(W,h)$ be an integrable 2-jet. Set $U:=\Menge{h(x,y)}{x,y\in W}_\R$\,. 
Then $(W,U)$ is an orthogonal curvature invariant pair.
\end{proposition}

An (orthogonal) curvature invariant pair $(W,U)$ which is induced by an integrable 2-jet
as in Proposition~\ref{p:cip} will be called {\em integrable}. 

Furthermore, it is known that every complete parallel submanifold of a
simply connected symmetric space whose normal spaces are curvature invariant is even a
symmetric submanifold (cf.~\cite[Proposition~9.3]{BCO}). Hence we obtain a result, which was already proved in~\cite{J1}:

\bigskip
\begin{corollary}\label{co:1-full}
Every 1-full complete parallel submanifold of a
simply connected symmetric space is a symmetric submanifold.
\end{corollary}
If $W$ is a curvature invariant subspace of $T_pN$\,, then
\begin{equation}\label{eq:La1}\frakh_W:=\Spann{R^N_{x,y}}{x,y\in W}\;.
\end{equation}
is a Lie subalgebra of $\so(T_pN)$\,. Further, there exist natural 
representations of $\frakh_W$ on both $W$ and $W^\bot$ (obtained by restriction, respectively). 
We are interested in the $\frakh_W$-invariant subspaces of $W^\bot$\,. For this, 
we recall the following result, which is a simple consequence of  Schur's Lemma.

Let $W^\bot=U_1\oplus\cdots\oplus U_k$ be a decomposition into
$\frakh_W$-irreducible subspaces. 
After a permutation of the indices, 
there exists some $r\geq 1$ and a sequence $1=k_1<k_2<\cdots <k_{r+1}=k+1$ 
such that $U_{k_i}\cong U_{k_i+1}\cong \cdots\cong U_{k_{i+1}-1}$ for $i=1,\ldots, r$ 
but $U_{k_i}$ is not isomorphic to $U_{k_j}$ for $i\neq j$\,.
Hence, there is also the decomposition $W^\bot=\oplus_{i=1}^r{\bf U}_{i}$ 
with ${\bf U}_{i}:=U_{k_i}+ U_{k_i+1}\cdots+ U_{k_{i+1}-1}$\,. 
Then every irreducible  $\frakh_W$-invariant subspace $U$ of $W^\bot$ is contained
in some ${\bf U}_{i}$\,. Furthermore,  the irreducible  $\frakh_W$-invariant subspaces of ${\bf U}_{i}$ 
are parameterized by the real projective space $\R\rmP^{k_{i+1}-k_i-1}$ (if $U_{k_i}$ 
is irreducible even over $\C$), the complex projective space
$\C\rmP^{k_{i+1}-k_i-1}$ (if $U_{k_i}\otimes\C$ decomposes into two non-isomorphic $\frakh_W$-modules) or 
the quaternionic projective space $\bbH\rmP^{k_{i+1}-k_i-1}$for $i=1,\ldots, r$ (otherwise). More precisely, 
let $\lambda_j:U_{k_i}\to U_{k_i+j}$ be an $\frakh_W$-isomorphism
($j=1,\ldots,k_{i+1}-k_i-1$). Further, set $\lambda_0:=\Id_{U_{k_i}}$
and $\lambda_c:=\sum_{j=0}^{k_{i+1}-k_i-1}c_j\lambda_j$ for every $c=(c_0,\ldots,c_{k_{i+1}-k_i-1})\in\R^{k_{i+1}-k_i}$\,.
Then $U:=\lambda_c({U_{k_i}})$ is an irreducible $\frakh_W$-invariant subspace
of ${\bf U}_{i}$\,. This gives the claimed parameterization in case  $U_{k_i}$ 
is irreducible even over $\C$\,. The other cases are handled similarly.

\bigskip
\begin{example}[Curvature invariant pairs of $\C\rmP^n$]\label{ex:CPm}
Consider the complex projective space $N:=\C\rmP^n$\,. 
Its curvature tensor is given by $R^N_{u,v}=-u\wedge v-J u\wedge J v-2\,\omega(u,v)J $ 
for all $u,v\in T_p\C\rmP^n$ (where $J$ denotes the complex structure of $T_pN$ and 
$\omega (u,v):=\g{J u}{v}$ is the  K\"ahler form).
The curvature invariant subspaces of $T_pN$ are known to be precisely the totally real and the complex subspaces. 
Let us determine the orthogonal curvature invariant pairs $(W,U)$\,: 

if $W$ is totally real, then $R^N_{x,y}=-x\wedge y-J\, x\wedge J\, y$ for all $x,y\in W$\,.
Hence the Lie algebra $\frakh_W$ (see~\eqref{eq:La1}) is given by the linear space $\Spann{x\wedge y+J\, x\wedge J\, y}{x,y\in W}$\,. 
In the following, we assume that $\dim(W)\geq 2$\,. Consider the decomposition $W^\bot=J W\oplus (\C\,W)^\bot$ 
(here $(\C\,W)^\bot$ means the orthogonal complement of $\C\,W$ in $T_pN$). Then  $\frakh_W$ acts irreducibly on $J(W)$ and trivially on $(\C\,W)^\bot$\,.
Further, Eq.~\ref{eq:cip1} shows that $U$ is $\frakh_W$-invariant.  
It follows that either $J(W)\subset U$ or $U\subset (\C\,W)^\bot$ (cf.~\cite[Proposition 2.3]{N3}). 
In the first case, we claim that actually $U=J(W)$ (and hence $V:=W\oplus U$ is a complex subspace of $T_pN$\,, cf.~\cite[Lemma 4.1]{N3}): 

let $\tilde U\subset (\C\,W)^\bot $ be chosen such that $U=JW\oplus \tilde U$\,.
Clearly, $U$ is not complex, thus $U$ is necessarily totally real, because $U$ is curvature invariant. 
Moreover, we have $\dim(U)\geq 2$\,, thus $\frakh_U$ (defined as above) acts irreducibly on 
$J(U)=W\oplus J(\tilde U)$\,. Since $W$ is $\frakh_U$-invariant (see~\eqref{eq:cip2}), we see that this is not possible unless $J(\tilde U)=\{0\}$\,. 
The claim follows.

In the second case, we claim that $U$ is totally real (and thus $V$ is totally real, too, cf.~\cite[Lemma 3.2]{N3})): 

in fact, otherwise $U$ would be a complex subspace of $(\C\,W)^\bot$\,. 
Then the Lie algebra $\frakh_U$ is given by $\R\, J\oplus \Spann{\xi\wedge \eta+J\, \xi\wedge J\, \eta}{\xi,\eta\in U}$\,. 
Thus $\frakh_U$ acts on $U^\bot$ via $\R\, J$\,. Further, $W$ is invariant under the action of $\frakh_U$ according  
to~\eqref{eq:cip2} implying that $W$ is complex, 
a contradiction. The claim follows.

Anyway, the linear space $V$ is curvature invariant unless $\dim(W)=1$\,. 
Therefore, by means of arguments given in the proof of Theorem~\ref{th:main}, we see that every higher dimensional totally real 
parallel submanifold of $\C\rmP^n$ is a Lagrangian symmetric submanifold of some 
totally geodesically embedded $\C\rmP^k$ or a symmetric submanifold of some totally geodesically embedded $\R\rmP^k$\,.

If $W$ is a complex subspace of $T_p\C\rmP^n$\,, then $\frakh_W|_{W^\bot}=\R
J|_{W^\bot}$\,. Hence, if $(W,U)$ is an orthogonal curvature invariant pair, 
then both $U$ and $V:=W\oplus U$ are complex subspaces, too. 
This shows that every complex parallel submanifold of $\C\rmP^n$ is a complex symmetric submanifold 
of some totally geodesically embedded $\C\rmP^k$\,. 
\end{example}

\subsection{Further necessary integrability conditions}
\label{se:necessary}
Let $N$ be a symmetric space, $K\subset \Iso(N)$ denote the isotropy subgroup
at $p$\,, $\frakk$ denote its Lie algebra and
$\rho:\frakk\to\so(T_pN)$ be the linearized isotropy representation. 
Recall that 
\begin{equation}\label{eq:symmetric_space}
R^N_{u,v}\in\rho(\frakk)
\end{equation}
for all $u,v\in T_pN$ (since $N$ is a symmetric space).

Given a 2-jet $(W,h)$ at $p$\,, we set $U:=\Spann{h(x,y)}{x,y\in W}$\,, $V:=W\oplus U$ and
\begin{equation}\label{eq:frakk_V}
\frakk_V:=\Menge{X\in\frakk}{\rho(X)(V)\subset V}\;.
\end{equation} 
Then there is an induced representation of $\frakk_V$ on $V$\,. 
Further, consider the endomorphisms of $T_pN$ given by
\begin{equation}
\label{eq:generators}
[\fetth_{x_1},\ldots[\fetth_{x_k},R^N_{y,z}]\ldots]
\end{equation}
with $x_1,\ldots,x_k,y,z\in W$ and $k\geq 0$\,.  
Furthermore, recall that the centralizer of a subalgebra $\frakg\subset \so(V)$ is given by
\begin{equation}\label{eq:centralizer}
Z(\frakg):=\Menge{A\in\so(V)}{\forall B\in\frakg:[A,B]=0}\;.
\end{equation}

\bigskip
\begin{theorem}\label{th:dec}
Let an integrable 2-jet $(W,h)$ be given, set $U:=\Spann{h(x,y)}{x,y\in W}$ and $V:=W\oplus U$\,.
Then, the endomorphisms of $T_pN$ given by~\eqref{eq:generators} leave $V$ invariant and hence they generate 
a subalgebra $\frakg\subset \so(V)$ (by restriction to $V$) with the following property:
for each $x\in W$ there exist $A_x\in\rho(\frakk_V)|_V\cap\so(V)_-$\,, $B_x\in Z(\frakg)\cap\so(V)_-$ such that $\fetth_x=A_x+ B_x$\,.
\end{theorem}

\begin{proof}
Since $(W,U)$ is a curvature invariant pair, we have $R^N_{x,y}(V)\subset V$ for all $x,y\in W$ according to~\eqref{eq:cip1}.
Thus~\eqref{eq:generators} leaves $V$ invariant also for $k>0$\,, see~\eqref{eq:fettb1}.
Further, note that applying $[\fetth_x,\,\cdot\,]$ to~\eqref{eq:generators} leaves the form of~\eqref{eq:generators} 
invariant with the natural number $k$ increased by one for every $x\in W$\,. 
Hence $[\fetth_x,\frakg]\subset\frakg$\,. Furthermore, the restriction of~\eqref{eq:generators} to $V$
belongs to $\so(V)_+$ or $\so(V)_-$ according to whether $k$ is even or odd, 
see~\eqref{eq:fettb2} and~\eqref{eq:cip1}.
Therefore, $\frakg$ is a graded Lie subalgebra of $\so(V)$\,, i.e. $\frakg=\frakg_+\oplus\frakg_-$ 
with $\frakg_+:=\frakg\cap\so(V)_+$ and $\frakg_-:=\frakg\cap\so(V)_-$\,.

Let $A_x$ denote the orthogonal projection of $\fetth_x$ onto $\frakg$ with
respect to the positive definite symmetric bilinear form on $\so(V)$ which
is given by $-\trace(A\circ B)$ for all $A,B\in\so(V)$\,. Since there is the orthogonal splitting $\frakg=\frakg_+\oplus\frakg_-$
and $\fetth_x\in\so(V)_-$ holds, we immediately see that $A_x\in\so(V)_-$
(cf.~\cite[Lemma~4.19]{J2}). Furthermore, using the invariance
property of the trace form (i.e. $\trace([A,B]\circ C)=\trace(A\circ [B,C])$), we
conclude from $[\fetth_x,\frakg]\subset\frakg$ that $B_x:=\fetth_x-A_x$
centralizes $\frakg$\,. Further, we have $B_x\in\so(V)_-$\,. It remains to show that $\frakg\subset\rho(\frakk_V)|_V$\,:

because of~\eqref{eq:symmetric_space}, r.h.s.\ of~\eqref{eq:cond3} belongs to  $\rho(\frakk_V)|_V$ and so does l.h.s. Thus, 
the restriction to $V$ of~\eqref{eq:generators} belongs to  $\rho(\frakk_V)|_V$ for every $k$\,, which gives our claim.

This proves the theorem. 
\end{proof}

Given an orthogonal curvature invariant pair $(W,U)$\,, we set $V:=W\oplus U$\,. Then
\begin{equation}\label{eq:La2}
\frakh:=\frakh_W|_V+\frakh_U|_V
\end{equation}
is a Lie subalgebra of $\so(V)_+$\,. Therefore, restricting the elements of $\frakh$ to $W$ or 
$U$ defines representations of $\frakh$ on $W$ and $U$\,, respectively. Hence, we introduce the linear spaces of homomorphisms
\begin{align}\label{eq:hom_1}
&\Hom(W,U):=\Menge{\lambda:W\to U}{\lambda\ \text{is linear}};\\
\label{eq:hom_2}
&\Hom_\frakh(W,U):=\Menge{\lambda\in\Hom(W,U)}{\forall A\in\frakh:\lambda\circ A|_W=A|_U\circ\lambda}\;.
\end{align}
Note that the natural map
\begin{equation}\label{eq:iso1}
\so(V)_-\to \Hom(W,U), A\mapsto A|_W
\end{equation}
is actually a linear isomorphism inducing an equivalence
\begin{equation}\label{eq:iso2}
Z(\frakh)\cap\so(V)_-\cong \Hom_{\frakh}(W,U),
\end{equation}
where $Z(\frakh)$ denotes the centralizer of $\frakh$ in $\so(V)$\,. 
Further, mapping $\lambda$ to its adjoint $\lambda^*$ defines an isomorphism 
\begin{equation}\label{eq:iso3}
\Hom_\frakh(W,U)\cong\Hom_\frakh(U,W)\;.
\end{equation}

As a corollary of Theorem~\ref{th:dec}, we derive the following obstruction against integrability:

\bigskip
\begin{corollary}\label{co:dec}
Let an integrable 2-jet $(W,h)$ be given. Set $U:=\Spann{h(x,y)}{x,y\in W}$\,, 
$V:=W\oplus U$ and suppose additionally that $\rho(\frakk_V)|_V\cap\so(V)_-=\{0\}$\,.
Let $\frakh$ be the Lie algebra~\eqref{eq:La2}. Then $\frakh=\frakh_W|_V$ and
\begin{equation}\forall x\in W: h(x,\cdot)\in\Hom_\frakh(W,U).
\label{eq:dec2}
\end{equation}
\end{corollary}
\begin{proof}
Consider the subalgebra $\frakg\subset \so(V)$ described in Theorem~\ref{th:dec}. 
First, we claim that $\frakh$ is a subalgebra of $\frakg$ (this is actually true for every integrable 2-jet):

since~\eqref{eq:generators} with $k=0$ leaves $V$ invariant and its
restriction to $V$ belongs to $\frakg$\,, we have $A(V)\subset V$ and
$A|_V\in\frakg$ for all $A\in\frakh_W$\,. Further, we have seen in the proof of Theorem~\ref{th:dec} 
that $\frakg$ is normalized by $\fetth_x$ for every $x\in W$\,.
Furthermore, because $h$ is a symmetric bilinear map whose image spans $U$\,, the linear space $\Lambda^2U$
is spanned by the 2-wedges $h(x,x)\wedge h(y,y)$ with $x,y\in W$\,. 
Thus~\eqref{eq:ci} implies that also $A(V)\subset V$ and $A|_V\in\frakg$ for all $A\in\frakh_U$ holds.
The claim follows.

Second,  in the notation of Theorem~\ref{th:dec}, by means of~\eqref{eq:symmetric_space} and 
the usual rules for $\Z_2$-graded algebras we have $$[\fetth_x,\frakh_W|_V]=[A_x,\frakh_W|_V]\subset\rho(\frakk_V)|_V\cap\so(V)_-=\{0\}\;.$$ 
Hence~\eqref{eq:generators} vanishes for every $k\geq 1$\,.
In particular, $\frakh_W|_V=\frakh=\frakg$ and $h_x\in Z(\frakh)$\,, i.e. $h(x,\cdot)\in \Hom_\frakh(W,U)$  
for each $x\in W$ according to~\eqref{eq:fettb1},\eqref{eq:fettb2},\eqref{eq:iso1} and~\eqref{eq:iso2}\,.
\end{proof}

Further, set 
\begin{equation}\label{eq:Kernh}
\Kern(h):=\Menge{x\in W}{\forall y\in W:h(x,y)=0}\;.
\end{equation}
Using~\eqref{eq:fettb1},~\eqref{eq:fettb2} and~\eqref{eq:iso1}, we immediately see that
\begin{equation}\label{eq:Kernfetth}
\Kern(h)=\Menge{x\in W}{\fetth(x)=0}\;.
\end{equation}

\bigskip
\begin{proposition}\label{p:invariant}
Let an integrable 2-jet $(W,h)$ be given. Then
$\Kern(h)$ is invariant under the action of $\frakh_W$ on $W$.
\end{proposition}
\begin{proof}
This follows from the curvature invariance of $W$\,, 
the symmetry of $h$ and~\eqref{eq:sp} (with $R=R^N$), cf.~\cite[Proof of Lemma 5.1]{N1}.
\end{proof}

\subsection{Parallel submanifolds with 1-dimensional first normal spaces}
\label{se:one-dimensional}

For a higher-dimensional {\em extrinsic sphere}, 
it is known that the second osculating spaces are curvature invariant, cf.~\cite[Theorem~9.2.2]{BCO}.
More generally, we have:

\bigskip
\begin{proposition}\label{p:sph}
Let $N$ be a symmetric space, $(W,h)$ be an integrable 2-jet and $U:=\Spann{h(x,y)}{x,y\in W}$\,. 
Assume that $\dim(U)=1$ and $\dim(W)\geq 2$\,. Choose a unit vector $\eta\in U$ and suppose 
that $\frakh_W$ acts irreducible on $W$\,. Then $V:=W\oplus U$ is a curvature
invariant subspace of $T_pN$\,.
\end{proposition}
\begin{proof}
Using Proposition~\ref{p:invariant}, we obtain that $\Kern(h)=\{0\}$\,. Thus  $\tilde h(x,y):=\g{h(x,y)}{\eta}$
defines a non-degenerate bilinear form on $W$\,. Further, in view of Proposition~\ref{p:cip}, 
it remains to show that $R^N_{x,\eta}(V)\subset V$ holds.
For this, we may proceed as in the proof of~\cite[Theorem~9.2.2]{BCO}:

we can assume that $x\neq 0$ in which case there exist $y,z\in W$
with $h(x,z)=\eta$ and $h(y,z)=0$ (since $\tilde h$ is non-degenerate and $\dim(W)\geq 2$).
Hence, using~\eqref{eq:cond2} with $k=1$,
we see that $R^N_{x,\eta}=[\fetth_z,R^N_{x,y}]$ holds on $V$. The result follows by means of~\eqref{eq:fettb1} and the curvature invariance of $W$\,.
\end{proof}

\subsection{Parallel submanifolds with curvature isotropic tangent spaces}
\label{se:flat}
Let $N$ be a symmetric space, $\Iso(N)$ denote the isometry group, $\fraki$ 
be its Lie algebra and $\fraki(N)=\frakk\oplus\frakp$ be the Cartan decomposition.
Recall that a Cartan algebra is a maximal Abelian subalgebra of $\frakp$ whose elements are semi-simple
(cf.~\cite[Remark~1]{FP}) and that any two Cartan algebras are
conjugate in $\frakp$ via some isometry from the connected component of $\Iso(N)$\,. The rank of $N$ is, by definition,
the dimension of a Cartan subalgebra of $\frakp$\,. If $N$ is of compact or non-compact type,
then every maximal Abelian subalgebra of $\frakp$ is already a Cartan subalgebra.
The following is well known:

\bigskip
\begin{lemma}\label{le:cfl}
Suppose that $N$ is of compact or non-compact type.
Let a linear subspace $W\subset T_pN$ be given. The following is equivalent:
\begin{enumerate}
\item the linear space $W$ is a curvature isotropic subspace of $T_pN$\,, 
i.e.\ the curvature endomorphism $R^N_{u,v}$ vanishes for all $u,v\in W$\,;
\item the totally geodesic submanifold $\exp^N(W)$ is a flat of $N$\,;
\item $W$ is contained in a Cartan subalgebra of $\frakp$\,;
\item the sectional curvature of $N$ vanishes on every 2-plane of $W$\,, i.e.
$\g{R^N(u,v,v)}{u}=0$ for all $u,v\in W$\,.
\end{enumerate}
\end{lemma}

\bigskip
\begin{proposition}\label{p:cfl}
Let $N$ be a symmetric space of compact or non-compact type and 
let an integrable 2-jet $(W,h)$ at $p$ be given. Suppose that
$d:=\dim(W)$ is equal to the rank of $N$ and that the 
sectional curvature of $N$ vanishes on $W$\,. Then there
exists an orthonormal basis $\{x_1,\ldots,x_d\}$ of $W$  such that
\begin{align}
\label{eq:cfl0}
& h(x_i,x_j)=0\ \text{whenever}\ i\neq j\,,\\
\label{eq:cfl1}
& \eta_i:=h(x_i,x_i)\ \text{satisfy}\ \g{\eta_i}{\eta_j}=0\ \text{whenever}\ i\neq j\,,\\
\label{eq:cfl2}
&R^N_{x_i,x_j}=R^N_{x_j,\eta_i}=R^N_{\eta_i,\eta_j}=R^N_{\eta_j,x_i}=0\ \text{for all}\ i\neq j\;.
\end{align}
In particular, both $W$ and $U$ are curvature isotropic.
\end{proposition}
\begin{proof}
Let a parallel submanifold $M\subset N$ be given such that $T_pM=W$ and $h_p=h$\,.
It is known that in this situation the sectional curvature of $N$
vanishes identically along the parallel submanifold $M$ (see~\cite[Proposition~3.14]{J1}). 
It follows that $R^N_{x,y}=0$ for all $x,y\in T_pM$ and all $p\in M$\,, i.e. $M$ is a ``curved flat'' 
in the sense of Ferus/Petit~\cite{FP}. Therefore, since we assume that $\dim(M)= \rank(N)$\,, 
the Riemannian space $M$ is intrinsically flat according to a result of~\cite{FP}. 
Furthermore, Equation~\eqref{eq:sp} shows that $R^\bot_{x,y}\xi=0$ for all $\xi\in U$\,.
Using the Equations of Gau\ss, Codazzi and Ricci for a parallel submanifold, i.e.
\begin{equation}\label{eq:Gauss_Ricci}
\forall x,y\in T_pM:R^N_{x,y}=R^M_{x,y}\oplus R^\bot_{x,y}+[\fetth_x,\fetth_y]\;,
\end{equation}
we obtain that $[\fetth_x,\fetth_y]=0$ for all $x,y\in W$\,. We claim that there exists an orthonormal basis 
$\{x_1,\ldots, x_d\}$ of $W$ such that~\eqref{eq:cfl0},\eqref{eq:cfl1} hold:

since $\Menge{h_x}{x\in W}$ is a set of pairwise commuting, skew-symmetric operators which map $W$ to $U$ and vice versa, 
there exist an orthonormal basis $\{x_1,\ldots, x_d\}$ of $W$ such that $\Kern(h)=\{x_1,\ldots, x_k\}_\R$ (see~\eqref{eq:Kernh},~\eqref{eq:Kernfetth}), 
an orthonormal basis $\{\xi_{k+1},\ldots, \xi_d\}$ of $U$ 
and linear maps $\lambda_i:W\to \R$ such that $\fetth_x=\sum_{i=k+1}^d\lambda_i(x)\, x_i\wedge\xi_i$\,. 
Using the symmetry of $h$\,, \[\lambda_j(x_i)\,\xi_j=\sum_{l=k+1}^d\lambda_l(x_i)\, x_l\wedge\xi_l\,(x_j)=h(x_i,x_j)=h(x_j,x_i)=\lambda_i(x_j)\,\xi_i\;.\]
It follows that $\lambda_i(x_j)=0$ for $i\neq j$\,. This gives our claim.

Moreover, by means of~\eqref{eq:cfl0}, we have
\[
\forall i\neq j:\, R^N_{x_i,\eta_j}|_V=R^N_{x_i,h(x_j,x_j)}|_V=-R^N_{h(x_j,x_i),x_j}|_V=0
\] where the second equality uses~\eqref{eq:cond2} (with $k=1$), i.e.\ the curvature endomorphism
$R^N_{x_i,\eta_j}$ vanishes on $V$ whenever $i\neq j$\,. Furthermore,~\eqref{eq:ci} implies that then 
also $R^N_{\eta_i,\eta_j}$ vanishes on $V$\,. Using Lemma~\ref{le:cfl} once more, $R^N_{x_i,\eta_j}$ and $R^N_{\eta_i,\eta_j}$ 
both vanish on $T_pN$ unless $i=j$\,. The result now follows.
\end{proof}

In the notation of Proposition~\ref{p:cfl},  set $V_i:=\{x_i,\eta_i\}_\R$ for $i=1,\ldots, d$\,.
The linear spaces $V_i$ are pairwise orthogonal 
and $R^N(u,v)=0$ whenever $(u,v)\in V_i\times V_j$ with $i\neq j$\,.

\bigskip
\begin{lemma}\label{le:V}
Let $\{V_i\}_{i=1,\ldots, d}$  be a collection of pairwise orthogonal subspaces of $T_pN$ such that $R^N_{u,v}=0$ whenever
$(u,v)\in V_i\times V_j$ with  $i\neq j$\,. Then there exist pairwise orthogonal 
curvature invariant subspaces of $T_pN$\,, denoted by $\bar V_i$\,, such that
\begin{align} 
\label{eq:property1}
&V_i\subset \bar V_i\text{ for }i=1,\ldots,d\;,\\
\label{eq:property2}
&R^N_{u,v}=0\ \text{whenever}\ (u,v)\in \bar V_i\times \bar V_j\ \text{with}\ i\neq j\;.
\end{align}
Moreover, then also the linear space 
\begin{equation}\label{eq:barV}
\bar V:=\bigoplus_{i=1}^d\bar V_i
\end{equation} 
is a curvature invariant subspace of $T_pN$\,.
\end{lemma}
\begin{proof}
Consider collections $\{\bar V_i\}_{i=1,\ldots, d}$ of pairwise orthogonal subspaces of $T_pN$ with 
Properties~\eqref{eq:property1},\eqref{eq:property2}\,.
Such collections exist, since at least one is given by $\bar V_i:=V_i$\,. 
Hence, for obvious reasons, there exists $\{\bar V_i\}_{i=1,\ldots, d}$ 
which is maximal in the following sense:

if $\{\tilde V_i\}_{i=1,\ldots,d}$ is another collection of pairwise orthogonal subspaces of $T_pN$ 
with properties~\eqref{eq:property1},\eqref{eq:property2} 
and such that $\bar V_i\subset \tilde V_i$ for $i=1,\ldots, d$\,, then $\bar V_i=\tilde V_i$ holds for all $i$\,.

Suppose that $\{\bar V_i\}_{i=1,\ldots, d}$ is maximal. We claim that each linear space $\bar V_i$ is curvature invariant in $T_pN$\,:

let $u_i,v_i,w_i\in \bar V_i$ and $w_j\in \bar V_j$ with $i\neq j$\,. 
Then, using a symmetry of $R^N$\,, $$\g{R^N(u_i,v_i,w_i)}{w_j}= \g{R^N(w_i,w_j,u_i)}{v_i}\stackrel{\eqref{eq:property2}}{=}0\;.$$ 
Therefore, the linear space $\tilde V_i:=\bar V_i+\R\, R^N(u_i,v_i,w_i)$ is contained in the orthogonal complement of $\bar V_j$\,, too. 
Further, note that  
\begin{equation}
R^N(u_i,v_i,w_j)=-R^N(w_j,u_i,v_i)-R^N(v_i,w_j,u_i)\stackrel{\eqref{eq:property2}}{=}0+0=0
\end{equation}
by the first Bianchi identity. Thus, the Jacobi-Identity for the Lie bracket on $\fraki(N)$ shows that
\[R^N_{w_j,R^N(u_i,v_i,w_i)}=-R^N_{R^N(u_i,v_i,w_j),w_i}+[R^N_{u_i,v_i},R^N_{w_j,w_i}]=0+0=0\;.\]
Therefore, the curvature endomorphism  $R^N_{u,v}$  vanishes whenever $(u,v)\in \tilde V_i\times \bar V_j$ with  $i\neq j$\,. 
Hence, $\tilde V_i=\bar V_i$ by maximality of $\{\tilde V_i\}$ and we conclude that $\bar V_i$ is curvature invariant for $i=1,\ldots, d$\,.

Further, we claim that then also $\bar V$ is curvature invariant:

let $u,v,w\in \bar V$ be given. We have to show that $R^N(u,v,w)\in \bar V$\,. For this, we can assume, 
by multilinearity of $R^N$\,, that each of these three vectors belongs to some $\bar V_i$\,.
If $(u_i,v_j,w_k)\in \bar V_i\times \bar V_j\times \bar V_k$ with $i\not\in\{j,k\}$\,, 
then $R^N(u_i,v_j,w_k)=R^N(w_k,u_i,v_j)=0$ by means of~\eqref{eq:property2}
and hence also $R^N(v_j,w_k,u_i)=0$ because of the first Bianchi identity. 
Therefore, $R^N(u,v,w)=0$ unless all three vectors $u,v,w$ belong to the same $\bar V_i$ in which case 
$R^N(u,v,w)\in\bar V_i\subset\bar V$ by the curvature invariance of $\bar V_i$\,.
\end{proof}

\bigskip
\begin{corollary}\label{co:circles}
In the situation of Proposition~\ref{p:cfl}, let $M$ be the complete parallel submanifold through $p$ whose 2-jet is given by $(W,h)$\,. 
Then there exists a simply connected totally 
geodesic submanifold $\bar M\subset N$ which splits as a 
Riemannian product $M_1\times\cdots \times M_d$ and there exist extrinsic circles $C_i\subset M_i$ 
such that $M$ is the Riemannian product $C_1\times\cdots \times C_d$\,. In particular, 
the parallel curved flat $M^d$ with $d=\rank(N)\geq 2$ is never full if $N$ is simply connected and irreducible.
\end{corollary}
\begin{proof}
Following the notation of Proposition~\ref{p:cfl}, set $V_i:=\{x_i,\eta_i\}_\R$ for $i=1,\ldots, d$\,.
By means of Lemma~\ref{le:V}, there exist curvature invariant spaces $\bar V_i\subset T_pN$ 
which satisfy~\eqref{eq:property1},\eqref{eq:property2}. Consider the corresponding totally geodesic submanifolds $M_i:=\exp^N(\bar V_i)$ and 
recall that our notion of submanifolds includes isometric immersions as well (see Footnote~\ref{fn:notion}). 
Hence we can assume that $M_i$ is simply connected (by means of passing to the universal covering space).
Further, the linear space $\bar V$ defined by~\eqref{eq:barV} is curvature invariant and we  can assume that the totally geodesic submanifold 
$\bar M:=\exp^N(\bar V)$ is simply connected, too. According to~\eqref{eq:property2},\eqref{eq:barV}, 
we thus obtain the Riemannian splitting $\bar M=M_1\times\cdots \times M_d$\,. Therefore, on the one hand, by means of~\eqref{eq:cfl0},~\eqref{eq:cfl1} 
we have $\osc_pM=\oplus_{i=1}^dV_i\subset \bar V$ which implies that $M\subset \bar M$ (reduction of the codimension)\,.  
Further, let $c_i:\R\to M_i$ be the geodesic line tangent to $x_i$ (in case $\eta_i=0$) 
or the extrinsic circle with $\dot c_i(0)=x_i$ and $\nabla^{M_i}_\partial \dot c_i(0)=\eta_i$ (otherwise) 
and set $C_i:=c_i(\R)$\,. Then, on the other hand, also the Riemannian product $\tilde M:=C_1\times\cdots\times C_d$ 
is a parallel submanifold of $\bar M$\,. Further, the 2-jets at $p$ of $M$ and $\tilde M$\,, 
respectively, are the same according to~\eqref{eq:cfl0},\eqref{eq:cfl1}. 
Therefore, $M=\tilde M$ since a complete parallel submanifold of $\bar M$ 
is uniquely determined by its 2-jet at one point.
\end{proof}

\subsection{Symmetric submanifolds of product spaces}
\label{se:naitoh}
We recall the following special case of~\cite[Theorem 2.2]{N2}:

\bigskip
\begin{theorem}[H.~Naitoh]\label{th:products}
Suppose that $N$ is a simply connected symmetric space
and that the de~Rham decomposition of $N$ has precisely two factors, $N=N_1\times N_2$\,. 
If $M\subset N$ is  a symmetric submanifold, 
then either $N_1=N_2$ and $M=\Menge{(p,g(p))}{p\in N_1}$  where $g$ is an isometry of $N_1$
(in particular, then $M$ is totally geodesic) or $M$ is a product $M_1\times M_2$ of symmetric
submanifolds $M_i\subset N_i$ for $i=1,2$\,. 
\end{theorem}
\begin{proof}
In fact, in case both factors of $N$ are of compact type, 
we can immediately apply~\cite[Theorem 2.2]{N2}.
In case both factors of $N$ are of non-compact type, we use the duality between
compact and non-compact spaces to pass to the previous case (note that the results of~\cite{N2} are mainly 
based on~\cite[Lemma~3.1]{N2} which is preserved under duality).
In the general case, we decompose $N\cong  N_c\times N_{nc}\times N_e$ 
into its compact, non-compact and Euclidean factor (where one or more factors may be trivial) and show as
in~\cite[p.562/563]{N2} that $M$ splits as the Riemannian product $M=M_c\times M_{nc}\times M_e$  
of symmetric submanifolds $M_c\subset N _c$\,, $M_{nc}\subset N_{nc}$ and
$M_e\subset N_e$\,, which finally establishes  Theorem~\ref{th:products}.
\end{proof}

\section{Parallel submanifolds of $\rmG^+_2(\R^{n+2})$}
\label{se:sebastian}
Let $n\geq 2$ and consider the simply connected compact Hermitian symmetric space $N:=\rmG^+_2(\R^{n+2})$ of rank two
which is given by the oriented 2-planes of $\R^{n+2}$\,. In standard notation,
we have $N\cong \SO(n+2)/\SO(2)\times\SO(n)$\,. Let $\{e_1,\ldots,
e_{n+2}\}$ be the standard orthonormal basis of $\R^{n+2}$ and set
$p:=\{e_{n+1},e_{n+2}\}_\R$\,. Then $p$ is an oriented 2-plane in $\R^{n+2}$ and
$T_pN=\Hom(\R^2,\R^{n})$ (here and in the following we identify $\R^2\cong
\{e_{n+1},e_{n+2}\}_\R$ and $\R^n\cong \{e_1,\ldots,e_n\}_\R$). 

The Hermitian structure on $T_pN$ is given by
\begin{equation}
J^N(\lambda):=\lambda\circ e_{n+1}\wedge e_{n+2}
\end{equation}
for all $\lambda\in \Hom(\R^2,\R^{n})$ (here we use the
natural isomorphism $\Lambda^2(\R^2)\cong\so(2)$ such that $e_{n+1}\wedge e_{n+2}$ is the rotation
in the positive sense by an angle of 90 degree in $\R^2$). Thus $T_pN$ is also an $n$-dimensional complex vector space 
where the multiplication with the imaginary unit $\i$ is given by $J^N$\,. Further, for every $\varphi\in\R$ set 
\begin{equation}
\Re(\varphi):=\Menge{\lambda\in \Hom(\R^2,\R^{n})}{\cos(\varphi)\lambda(e_{n+1})=-\sin(\varphi)\lambda(e_{n+2})}\;.
\end{equation} 
Then $\scrU:=\Menge{\Re(\varphi)}{\varphi\in\R}$ is a family of {\em real forms} of $T_pN$ (i.e.\ maximal totally real subspaces of
$T_pN$) and $\scrU=\Menge{\e^{\i\varphi}\,\Re}{\varphi\in\R}$ for every
$\Re\in\scrU$\,. Following the notation from~\cite{sebastian1}, 
we thus see that $\scrU$ is a ``circle'' of real forms.

Let $\so(n+2)=\frakk\oplus\frakp$ be the Cartan decomposition of $\so(n+2)$\,,
i.e. $\frakk=\so(2)\oplus\so(n)$ and $\frakp$ is the orthogonal complement of
$\frakk$ with respect to the positive definite invariant form defined by
$-\trace(A\circ B)$ for all $A,B\in \so(n+2)$\,. Then
$\frakp=\Menge{A\in\so(n+2)}{A(\R^2)\subset \R^{n},\ A(\R^n)\subset
  \R^{2}}$ and $\frakk=[\frakp,\frakp]$\,. Using the natural isomorphism $\frakp\to
T_pN, A\mapsto A|_{\R^2}$\,, the linearized isotropy representation $\rho:\frakk\to \so(T_pN)$
is given by $\rho(A)B= [A,B]$ for all $A\in\frakk$ and $B\in\frakp$\,. Further,
then we have $R^N(A,B,C)=-[[A,B],C]$ for all $A,B,C\in\frakp$ (since $N$ is a symmetric
space). Thus, we obtain that $\rho(\frakk)=\R\, J^N \oplus \so(\Re)$ and
\begin{equation}
\label{eq:crg}
\forall u,v\in T_pN:R^N_{u,v}=(\g{\Re(v) }{\Im(u)}-\g{\Re(u)}{\Im(v) }) J^N-\Re(u)\wedge
\Re(v)-\Im(u) \wedge \Im(v)
\end{equation}
for every $\Re\in\scrU$ if the scalar product $\g{A}{B}$ is chosen as $-1/2\,\trace(A\circ B)$ for all $A,B\in\frakp$\,.
Here $v=\Re (v)+\i\,\Im(v)$ denotes the splitting of $v$ with respect to the decomposition
$T_pN=\Re\oplus\i\,\Re$\,, and the Lie algebra $\so(\Re)$ acts on $T_pN$ via $A\,v=A\,\Re (v)+\i\, A\,\Im(v)$ for
all $A\in \so(\Re)$ and $v\in T_pN$\,. For an equivalent description of $R^N$\,, see~\cite[p.84, Eq.~(16)]{sebastian1} (note that 
there our metric gets scaled by a factor $1/2$).\footnote{Clearly, the curvature tensor itself does not change if one scales the metric by a constant factor, but r.h.s.\ of~\eqref{eq:crg} depends on the chosen scaling.} 

Recall that a subspace $W\subset T_pN$ is called curvature invariant if 
$R^N(x,y,z)\in W$ for all $x,y,z\in W$\,. Describing $R^N$ via the Lie triple bracket on $\frakp$\,, 
we see that there is a one-one correspondence between curvature invariant subspaces
of $T_pN$ and {\em Lie triple systems} in $\frakp$\,, i.e.\ linear spaces $W\subset \frakp$ satisfying $[[W,W],W]\subset W$\,. 
For the following result see~\cite[Theorem~4.1]{sebastian1}:

\bigskip
\begin{theorem}[S.~Klein]\label{th:ci}
For $N:=\rmG_2^+(\R^{n+2})$ with $n\geq 2$\,, there are precisely the following curvature invariant subspaces of $T_pN$\,:
\begin{itemize}
\item Type $(c_k)$\,: let $\Re\in\scrU$ and a $k$-dimensional  subspace $W_0\subset \Re$ be given. 
Then  $W:=\C\,W_0$ is curvature invariant. Here we assume that $k\geq 1$\,.
\item Type $(tr_{k,\ell})$\,: let $\Re\in\scrU$ and an orthogonal  pair of subspaces 
$W_1,W_2$ of $\Re$ be given. Then $W:=W_1\oplus \i W_2$ is curvature invariant.
Here the dimensions $k$ and $\ell$ of $W_1$ and $W_2$\,, respectively, are supposed to satisfy $k+\ell\geq 2$\,. 

\item Type $(c_k')$\,: let $\Re\in\scrU$ and a subspace  $W'\subset \Re$ 
equipped  with a Hermitian structure $I'$ be given. Then $W:=\Menge{v-\i\,I' v)}{v\in W'}$ is curvature invariant.
Here $k\geq 1$ denotes the complex dimension of $W'$\,.
\item Type $(tr_k')$\,: let $\Re\in\scrU$\,, a  subspace 
$W'\subset \Re$ equipped with  a Hermitian structure $I'$ and a real form 
$W_0'$ of the complex vector space $(W',I')$ be given. Then $W:=\Menge{v-\i\,I' v)}{v\in
    W_0'}$  is curvature invariant. 
Here $k\geq 2$ denotes the dimension of $W_0'$\,.
\item Type $(ex_3)$\,: let $\Re\in\scrU$ 
and an orthonormal system $\{e_1,e_2\}\subset\Re$ be given. 
The 3-dimensional linear space $W:=\{e_1-\i\,e_2,e_2+\i\,e_1,e_1+\i\,e_2\}_\R$ is curvature invariant.
\item Type $(ex_2)$ (only for $n\geq 3$)\,: let $\Re\in\scrU$ and an
  orthonormal system $\{e_1,e_2,e_3\}\subset\Re$ be given. 
The 2-dimensional linear space $W:=\{2\,e_1+\i\,e_2, e_2+\i(e_1+\sqrt{3}\,e_3)\}_\R$ is
  curvature invariant.
\item Type $(tr_1)$\,: let $u$ be a unit vector of $T_pN$\,. The 1-dimensional space $\R u$  is curvature invariant.
\end{itemize}
\end{theorem}

Our notation emphasizes that spaces of Types $(c_k)$ and $(c_k')$ both are
complex of dimension $k$ over $\C$ and those of
Types  $(tr_{k,\ell})$ and $(tr_k')$ are totally real of
dimensions $k+\ell$ and $k$\,, respectively. The spaces of Types  $(ex_3)$
and $(ex_2)$ are ``exceptional'' (in the sense that they do not 
occur in a series). 

As was mentioned at the beginning of
this paper, the totally geodesic submanifolds of $N$ were already classified in~\cite{chen-naganoI}.
However, there the totally geodesic submanifolds which are
associated with curvature invariant subspaces of Types $(ex_3)$ and $(ex_2)$
do not occur. For an explicit description of these submanifolds, see~\cite{sebastian4}.

\subsection{Curvature invariant pairs of $\rmG^+_2(\R^{n+2})$}
\label{se:cip}

\begin{table}[t]\caption{Orthogonal curvature invariant pairs of $\rmG^+_2(\R^{n+2})$} \label{table}    
\begin{tabular}{ | l | l | l | l|}    \hline\noalign{\smallskip}   Type  & Data & Conditions \\ \hline 
$(c_{k},c_{\ell})$     & $(\Re, W_0;\Re^*,U_0)$ & $\Re=\Re^*$\,, $W_0\bot U_0$ \\ \hline 
$(tr_{i,j},tr_{k,\ell})$  & $(\Re, W_1,W_2;\Re^*,U_1,U_2)$ & $\Re=e^{\i\varphi}\,\Re^*$\,, 
$W_1\oplus W_2\bot e^{\i\varphi} (U_1\oplus  U_2)$ \\ \hline 
$(tr_{j,k},tr_{\ell,j})$ &  $(\Re, W_1,W_2;\Re^*,U_1,U_2)$ & $\Re=\Re^*$\,, $W_1=U_2$\,, $W_2\bot U_1$ \\ \hline 
$(tr_{k,\ell},tr_{\ell,k})$ &  $(\Re, W_1,W_2;\Re^*,U_1,U_2)$ & $\Re=\Re^*$\,, $W_1=U_2$\,, $W_2=U_1$ \\ \hline 
$(tr_{1,k},tr_{\ell,1})$ &  $(\Re, W_1,W_2;\Re^*,U_1,U_2)$ & $\Re=\Re^*$\,, $W_1\bot U_1$\,, $W_2\bot U_2$\,, $W_2\bot U_1$ \\ \hline 
$(tr_{1,k},tr_{\ell,1})$ &  $(\Re, W_1,W_2;\Re^*,U_1,U_2)$ & $\Re=\Re^*$\,, $W_1\bot U_1$\,,  $W_2\bot U_2$\,, $W_2=U_1$ \\ \hline 
$(tr_{1,1},tr_{1,1})$ &  $(\Re, W_1,W_2;\Re^*,U_2,U_2)$ & $\Re=\Re^*$\,, $W_1\bot U_1$\,, $W_2\bot U_2$ \\ \hline 
$(tr_{k,\ell},tr_1)$ &  $(\Re, W_1,W_2;u)$ & $u\in (\C\,W_1\oplus \C\,W_2)^\bot$\\ \hline 
$(tr_{1,k},tr_1)$ &  $(\Re, W_1,W_2;u)$ & $\Re(u)\in W_1^\bot$\,, $u\in (\C\,W_2)^\bot$ \\ \hline 
$(tr_{1,1},tr_1)$ &  $(\Re, W_1,W_2;u)$ & $\Re(u)\in W_1^\bot$\,, $\Im(u)\in W_2^\bot$ \\ \hline 
$(c_k,c_\ell')$ &  $(\Re,W_0;\Re^*,U',I')$ & $\Re=\Re^*$\,, $W_0\bot U'$ \\ \hline 
$(c_k',c_\ell')$ &  $(\Re,W',I';\Re^*,U',J')$ & $\Re=\Re^*$\,, $W'\bot U'$ \\ \hline 
$(c_k',c_k')$ &  $(\Re,W',I';\Re^*,U',J')$ & $\Re=\Re^*$\,, $W'=U'$\,, $I'=-J'$ \\ \hline 
$(c_1',tr_1)$ &  $(\Re,W',I';u)$ & $u\in\bar W$ \\ \hline 
$(tr_j',tr_{k,\ell})$ & $(\Re,W',I',W_0';\Re^*, U_1,U_2)$ & $\Re=\Re^*$\,, $W'\bot U_1\oplus U_2$ \\ \hline 
$(tr_{k}',tr_{\ell}')$ & $(\Re,W',I',W_0';\Re^*,U',J',U_0')$ & $\Re=\Re^*$\,, $W'\bot U'$ \\ \hline  
$(tr_{k}',tr_{k}')$ & $(\Re,W',I',W_0';\Re^*,U',J',U_0')$ & $\Re=\Re^*$\,, $W'=U'$\,, $U_0'=I'(W_0')$\,, $J'=I'$ \\ \hline 
$(tr_{k}',tr_{k}')$ & $(\Re,W',I',W_0';\Re^*,U',J',U_0')$ & $\Re=\Re^*$\,, 
$W'=U'$\,, \\ & & $U_0'=\exp(\theta\,I')(W_0')$\,, $J'=-I'$ \\ \hline 
$(tr_{2}',tr_{2}')$ & $(\Re,W',I',W_0';\Re^*,U',J',U_0')$ & $\Re=\Re^*$\,, $W'=U'$\ \text{and there exists}\\ 
& & $\tilde J\in\SU(W',\tilde I)\cap \so(W')$\ \text{such that}\\ 
& &$U_0'=\tilde J(W_0')$ and $J'=\tilde J\circ I'\circ\tilde J^{-1}$\ \ \ \ \;$(*)$ \\ \hline
$(tr_{k}',tr_1)$ & $(\Re,W',I',W_0';u)$ & $u\in (\C\,W')^\bot$ \\\hline 
$(ex_3,tr_1)$ & $(\Re,\{e_1,e_2\};u)$ & $u=\pm\,1/\sqrt 2\,(\e_2-\i\,e_1)$ \\\hline 
$(tr_1,tr_1)$ & $(u;v)$ & $u\,\bot\, v$ \\\hline       
\end{tabular}\\ 
$(*)$\ If $W$ is of Type $tr_{2}'$ defined by $(\Re,W',I',W_0')$\,, then a second Hermitian structure\\ on $W'$ is given by $\tilde I:=e_1\wedge
e_2+I' e_1\wedge I' e_2$ for some orthonormal basis $\{e_1,e_2\}$ of $W_0'$\,.
\end{table}

In this section, we determine the orthogonal curvature invariant pairs of $T_pN$\,.  
Note that $(W,U)$ is a curvature invariant pair if and only if $(U,W)$ has this property. 
Since Theorem~\ref{th:ci} provides seven types of
curvature invariant subspaces of $T_pN$\,, there are, roughly said, $7\cdot 8/2=28$ 
possibilities to consider. 

Our approach is briefly explained as follows: given a curvature invariant subspace
$W$ of $T_pN$\,, we will first determine the Lie algebra $\frakh_W$ (see~\eqref{eq:La1}) and
the  $\frakh_W$-invariant subspaces of $W^\bot$\,. Second, we will also determine those skew-symmetric endomorphisms of $T_pN$ which belong to 
$\rho(\frakk)$ and leave $W$ invariant, see~\eqref{eq:symmetric_space}. Once this information is available for linear
spaces of Types x and y, we will determine all curvature invariant pairs of Type (x,y), see Table~\ref{table}.

\bigskip
\begin{lemma}\label{le:type(c_k)}
Let $W$ be curvature invariant of Type $(c_k)$ defined by the data $(\Re, W_0)$\,.
\begin{enumerate}
\item We have  \begin{equation} \frakh_W=\R\, J^N\oplus\so(W_0)\;. 
\end{equation}
\item A subspace of $W^\bot$ is $\frakh_W$-invariant if and only if it is a complex subspace.
\item Let $A\in\so(\Re)$ and $a\in\R$ be given. The linear map $a\,J^N+A$ leaves $W$ invariant if and only if $A(W_0)\subset W_0$\,.  
\end{enumerate}
\end{lemma}
\begin{proof}
By means of~\eqref{eq:crg}, the curvature endomorphism $R^N_{\i\,x,x}$ is given by 
$J^N$ for every unit vector $x\in W_0$\,.
Further, $R^N_{x,y}=R^N_{\i\, x,\i\,y}=-x\wedge y$ for all $x,y\in W_0$ and $R^N_{x,\i\,y}=0$ if 
$x,y\in W_0$ with $\g{x}{y}=0$\,. Part~(a) follows. For~(b), 
note that $\frakh_W|_{W^\bot}=\R\, J^N|_{W^\bot}$\,.  Part~(c) is obvious. 
\end{proof}

\bigskip
\begin{corollary}\label{co:type(c_k,c_l)}
Let $W$ and $U$ be curvature invariant of Types $(c_k)$ and $(c_\ell)$
defined by the data $(\Re,W_0)$ and $(\Re^*,U_0)$\,, respectively. If $\Re=\Re^*$ and $W_0\bot U_0$\,, 
then $(W,U)$ is an orthogonal curvature invariant pair. Moreover, every orthogonal curvature invariant 
pair of Type $(c_k,c_\ell)$ can be obtained in this way.
\end{corollary}
\begin{proof}
Using Lemma~\ref{le:type(c_k)}, the first part of the corollary is obvious.
For the last assertion, since the linear space $W$ is  
determined also by the tuple $(e^{\i\varphi}\,\Re,e^{\i\varphi}\,W_0)$ for all
$\varphi\in\R$\,, we can assume that $\Re=\Re^*$\,. Thus the condition $W\bot U$
implies that $W_0\bot U_0$\,.
\end{proof}

\bigskip
\begin{corollary}\label{co:type(tr_ij,c_k)}
There are no orthogonal curvature invariant pairs of Types $(c_j,tr_{k,\ell})$\,,  
$(c_j,tr_k')$\,, $(c_j,ex_3)$\,, $(c_j,ex_2)$ and $(c_j,tr_1)$\,.
\end{corollary}
\begin{proof}
If $W$ is of Type $(c_j)$\,, then any $\frakh_W$-invariant subspace of $W^\bot$
is complex according to Lemma~\ref{le:type(c_k)}~(b). Since spaces of Types $(tr_{k,\ell})$\,,  
$(tr_k')$\,, $(ex_3)$\,, $(ex_2)$ and $(tr_1)$ are not complex, 
this proves the result. \end{proof}

\bigskip
\begin{lemma}\label{le:type(tr_kl)}
Let $W$ be of Type $(tr_{k,\ell})$ defined by the data $(\Re,W_1,W_2)$\,.
\begin{enumerate}
\item We have \begin{equation}\label{eq:frakhW2}
\frakh_W=\so(W_1)\oplus\so(W_2)\;.
\end{equation}
\item If $k,\ell\neq 1$\,, then a subspace of $W^\bot$ is $\frakh_W$-invariant if and only if it is
  equal to $\i W_1$\,, $W_2$\,, a subspace of the orthogonal complement of $\C\,W_1\oplus \C\,W_2$\,, or a sum of
  such spaces.
 If $k=1$ and  $\ell\geq 2$\,, then a subspace of $W^\bot$ is $\frakh_W$-invariant if and only if it is
  equal to $W_2$\,, a subspace of $\i W_1\oplus(\C\,W_1\oplus \C\,W_2)^\bot$ or a sum of
  such spaces. 
If $k=\ell=1$\,, then any subspace of $W^\bot$ is $\frakh_W$-invariant.
\item Let $A\in\so(\Re)$ and $a\in\R$ be given. The linear map $a\,J^N+A$ leaves $W$ 
invariant if and only if $a=0$ and $A(W_i)\subset W_i$ for $i=1,2$\,.
\end{enumerate}
\end{lemma}
\begin{proof}
For~(a), see the proof of Lemma~\ref{le:type(c_k)}. For~(b), consider the 
decomposition $W^\bot= \i W_1\oplus W_2\oplus (\C\,W_1\oplus \C\,W_2)^\bot$ into $\frakh_W$-invariant
subspaces. Then $\frakh_W$ acts trivially on $(\C\,W_1\oplus \C\,W_2)^\bot$ and
irreducibly on both $\i W_1$ and $W_2$\,. In particular, the linear spaces $\i W_1$ and $W_2$ are trivial $\frakh_W$-modules only if $k=1$ or $\ell=1$\,, respectively. Moreover, they are
non-isomorphic $\frakh_W$-modules unless $k=\ell=1$\,. The result follows.
Part~(c) is straightforward.
\end{proof}

\bigskip
\begin{corollary}\label{co:type(tr_ij,tr_kl)}
Let $W$ and $U$ be curvature invariant of Types $(tr_{i,j})$ and $(tr_{k,\ell})$
defined by the data $(\Re,W_1,W_2)$ and  $(\Re^*,U_1,U_2)$\,, respectively. 
If one of the following conditions holds, then $(W,U)$ is an orthogonal curvature invariant pair: 

\begin{itemize}
\item the real number $\varphi$ is chosen such that  
$\Re=e^{\i\varphi}\,\Re^*$ and $e^{\i\varphi} (U_1\oplus U_2)$ belongs 
to the orthogonal complement of $W_1\oplus W_2$;
\item $\Re=\Re^*$\,, $W_2=U_1$ and $W_1=U_2$;
\item $\Re=\Re^*$\,, $W_2\bot U_1$ and $W_1=U_2$;
\item $i=\ell=1$\,, $\Re=\Re^*$\,, $W_1\bot U_1$\,, $W_2\bot U_2$  and  $W_2\bot U_1$\,;
\item $i=\ell=1$\,, $\Re=\Re^*$\,, $W_1\bot U_1$\,, $W_2\bot U_2$ and $W_2=U_1$;
\item $(i,j)=(k,\ell)=(1,1)$\,, $\Re=\Re^*$\,, $W_1\bot U_1$ and $W_2\bot U_2$\,.
\end{itemize}
Moreover, every orthogonal curvature invariant pair of Type $(tr_{i,j},tr_{k,\ell})$ can be obtained in this way.
\end{corollary}
\begin{proof}
Obviously, the pairs $(W,U)$ mentioned above satisfy $W\bot U$\,.
Further, the fact that they are curvature invariant pairs is verified by means
of Lemma~\ref{le:type(tr_kl)}. Conversely, let us see that these conditions are also necessary: 

we have 
\begin{align*}
&u_1=e^{-\i \varphi}e^{\i \varphi}\,u_1=\cos(\varphi)e^{\i \varphi}\,u_1-
\i\sin(\varphi)e^{\i \varphi}\,u_1\;,\\
&\i\,u_2=e^{-\i \varphi}e^{\i \varphi}\,\i\,u_2=\sin(\varphi)e^{\i \varphi}\,u_2+\i\,\cos(\varphi)e^{\i \varphi}\,u_2\;,
\end{align*}
with $e^{\i \varphi}\,u_1\in \Re$ and $\i\,e^{\i \varphi}\,u_2\in \i\,\Re$ 
for all $(u_1,u_2)\in U_1\times U_2$\,. Thus, the condition $U\bot W$ implies that
\begin{align*}
&0=\g{v_1}{u_1}=\cos(\varphi)\g{v_1}{e^{\i \varphi}\,u_1}\;,\\
&0=\g{v_1}{\i\,u_2}=\sin(\varphi)\g{v_1}{e^{\i \varphi}\,u_2}\;,\\
&0=\g{\i\,v_2}{u_1}=-\sin(\varphi)\g{v_2}{e^{\i \varphi}\,u_1}\;,\\
&0=\g{\i\,v_2}{\i\,u_2}=\cos(\varphi)\g{v_2}{e^{\i \varphi}\,u_2}
\end{align*}
for all $(v_1,v_2)\in W_1\times W_2$ and 
$(u_1,u_2)\in U_1\times U_2$\,. 
Hence, in case $\varphi\notin \pi/2\,\Z$\,, the condition $W\bot U$ necessarily 
implies that $e^{\i \varphi} (U_1\oplus U_2)\bot W_1\oplus W_2$\,.

In case $\varphi\in\Z\,\pi/2$\,, interchanging, 
if necessary, $U_1$ with $U_2$\,, we can assume that $\Re=\Re^*$\,. Further, suppose that $j\geq 2$\,. 
On the one hand, since $(W,U)$ is a curvature invariant pair and $\frakh_W\subset\so(\Re)$ 
by means of Lemma~\ref{le:type(tr_kl)}~(a), the linear space $U_1$ is an $\frakh_W$-invariant subspace of $W_1^\bot\cap\Re$\,. 
Using Lemma~\ref{le:type(tr_kl)}~(b), we conclude that
$U_1\bot W_1\oplus W_2$ or $U_1=W_2\oplus \tilde U$ for some $\tilde U\subset \Re$ 
which belongs to the orthogonal complement of $W_1\oplus W_2$\,. 
We claim that the latter is not possible unless $\tilde U=\{0\}$\,:

otherwise, since $(W,U)$ is a curvature invariant pair and $\frakh_U\subset\so(\Re)$\,, 
the linear space $\i W_2$ is an $\frakh_U$-invariant subspace of $U^\bot\cap\i\,\Re$\,. 
Moreover, the condition $U_1=W_2\oplus \tilde U$ implies that $k\geq j\geq 2$\,. Therefore, by means of 
Lemma~\ref{le:type(tr_kl)}~(b), we have $W_2\bot U_1\oplus U_2$ (which is clearly not given) or $W_2=U_1\oplus \tilde W$ 
for some $\tilde W\bot U_1\oplus U_2$\,. Hence $U_1=U_1\oplus \tilde W\oplus \tilde U$\,, thus $\tilde W=\tilde U=\{0\}$\,.

We conclude that $U_1\bot W_1\oplus W_2$ or $U_1=W_2$ unless $j=1$\,. Similarly, 
we can show that $U_2\bot W_1\oplus W_2$ or $U_2=W_1$ unless $i=1$\,.
Clearly, the same conclusions  hold with the roles of $W$ and $U$ interchanged. This finishes the proof.
\end{proof}

\bigskip
\begin{corollary}\label{co:type(tr_ij,line)}
Let $W$ and $U$ be curvature invariant of Types $(tr_{k,\ell})$ and $(tr_1)$ 
defined by the data $(\Re,W_1,W_2)$ and a unit vector $u\in T_pN$\,, respectively.
If one of the following conditions holds, then $(W,U)$ is an orthogonal curvature invariant pair:
\begin{itemize}
\item $k,\ell\neq 1$ and $u$ belongs to the orthogonal complement of $\C\,W_1\oplus \C\,W_2$;
\item $k=1$\,, $\ell\geq 2$\,, $\Re(u) \bot W_1$ and $u\,\bot \C\,W_2$;
\item $k=\ell=1$\,, $\Re(u)\bot W_1$ and $\Im(u) \bot W_2$\,.
\end{itemize}
Moreover, every orthogonal curvature invariant pair of Type $(tr_{k,\ell},tr_{1})$ can be obtained in this way.
\end{corollary}
\begin{proof}
Note, the pair $(W,U)$ is an orthogonal curvature invariant pair if and only if $u\in
W^\bot$ and $\frakh_W$ annihilates the vector $u$\,. If $k,\ell\neq 1$\,, this is equivalent to $u\,\bot \C\,W_1\oplus
\C\,W_2$ according to Lemma~\ref{le:type(tr_kl)}~(b). 
If $k=1$ and $\ell\geq 2$\,, we use the
same argument as before; however, now it is
allowed that $\Im(u)$ has a component in $W_1$\,.  
The case  $k\geq 2$\,, $\ell=1$ also follows (by passing from $\Re$ to $\i\,\Re$). 
In case $k=\ell=1$\,, the Lie algebra $\frakh_W$ is trivial and the only condition
is $u\in W^\bot$\,.
\end{proof}

\bigskip
\begin{lemma}\label{le:type(c_k')}
Let $W$ be of Type $(c_k')$ determined by the data $(\Re,W',I')$\,. 
Further, let $\bar W$ denote its complex conjugate in $T_pN$ with respect 
to the real form $\Re$\,.
\begin{enumerate}
\item We have \begin{equation}
\label{eq:type(c_k')_0}
\frakh_W=\su(W')\oplus\R(I'+k\, J^N)\;.
\end{equation} 
\item In case $k\geq 2$\,, a subspace of $W^\bot$ is $\frakh_W$-invariant if and only if it is
  equal to $\bar W$\,, a complex subspace of $(\C\,W')^\bot$ or a sum of
  such spaces. In case $k=1$\,, the previous statement remains correct 
if we replace the phrase ``equal to $\bar W$'' by ``contained in $\bar W$''.
Anyway, the linear space $\bar W$ as well as any $\frakh_W$-invariant subspace of $(\C\,W')^\bot$ is complex.
\item Let  $a\in\R$ and $A\in\so(\Re)$\,. The linear map $a\,J^N+A$ leaves $W$ 
invariant if and only if $A(W')\subset W'$ and $A|_{W'}\in\fraku(W',I')$\,.
\end{enumerate}
\end{lemma}
\begin{proof}
For~(a), note that 
\begin{equation}\label{eq:type(c_k')_1}
R^N_{u-\i\,I'u,v-\i\,I'v}=-2\,\g{I' u}{v}J^N-u\wedge v-I'u\wedge I'v
\end{equation} 
for all $u,v\in W'$ because of~\eqref{eq:crg}. In particular, for every unit vector $u\in W'$ and $v=I'u$
\begin{equation}\label{eq:type(c_k')_2}
R^N_{u-\i\,I'u,v-\i\, I'v}=-2\,J^N-2\,u\wedge I'u\;.
\end{equation}
Similarly, if $u,v\in W'$  are unit vectors with $\g{u}{I'v}=0$\,, then 
\begin{equation}\label{eq:type(c_k')_3}
R^N_{u-\i\,I'u,v-\i\,I'v}=-u\wedge v-I'u\wedge I' v\;.
\end{equation}
It follows from~\eqref{eq:type(c_k')_2},\eqref{eq:type(c_k')_3} that $A\in\frakh_W$ 
if and only if there exists some $B\in\fraku(W',I')$ 
such that $A=-\i\,\tr_\C(B)\, J^N+B$\,. Now~\eqref{eq:type(c_k')_0} is straightforward.

For~(b), note that \begin{equation}\label{eq:I'=pmi}
\forall v\in W':\;\i(v\mp \i\, I'v)=\pm I' v+i\, v=\pm (I'v\mp\i\, I'I'v)\;,
\end{equation} hence $J^N|_W=I'|_W$ and $J^N|_{\bar W}=-I'|_{\bar W}$\,. 
In particular, the linear space $\bar W$ is a complex subspace of $T_pN$\,.  
The fact that $I'=-J^N$ on $\bar W$ and Part~(a) together imply that $\frakh_W|_{\bar W}=\{0\}$ for $k=1$
and $\frakh_W|_{\bar W}=\fraku(W',I')$ for $k\geq 2$\,.
Further, we have $\fraku(W',I')=\fraku(W',-I')\cong\fraku(\bar W)$ where the second equality uses~\eqref{eq:I'=pmi}. 
Therefore, the linear space $\bar W$ is an irreducible $\frakh_W$-module of real dimension $2\,k$ for $k\geq 2$\,. 
Furthermore, the Lie algebra $\frakh_W$ acts on $(\C\,W')^\bot$ via $\R\, J^N$ and hence $(\C\,W')^\bot$ decomposes
as a direct sum of irreducible complex 1-dimensional subspaces. 
Part (b) follows.

For~(c), recall that $J^N|_W=I'|_W$ according to~\eqref{eq:I'=pmi}. 
Thus $W$ is actually complex and we can assume in the following that $a=0$\,. 
Since \begin{equation}
\label{eq:type(c_k')_4}
\forall v\in W':\;A(v+\i\, I'v)=A\,v+\i\,A\,I'v\;,
\end{equation}
for all $A\in\so(\Re)$\,, we see that $A(W)\subset W$ if and only if $A(W')\subset W'$ 
and $A|_{W'}\in\fraku(W',I')$\,. Part (c) follows.
\end{proof}

\bigskip
\begin{corollary}
Let $W$ and $U$ be of Types $(c_k)$ and $(c_\ell')$ 
determined by the data $(\Re,W_0)$ 
and $(\Re^*,U',I')$\,, respectively. If $\Re=\Re^*$ and $W_0\bot U'$\,, 
then $(W,U)$ is an orthogonal curvature invariant pair. 
Moreover, every orthogonal curvature invariant pair of Type $(c_k,c_\ell')$ 
can be obtained in this way.
\end{corollary}
\begin{proof}
Obviously, the pairs $(W,U)$ mentioned above satisfy $W\bot U$\,.
Further, the fact that these are curvature invariant pairs is verified by means
of Lemmas~\ref{le:type(c_k)} and~\ref{le:type(c_k')}, Parts~(a) and~(c).
Conversely, let us see that the conditions are also necessary:

here we can assume that $\Re=\Re^*$ (cf.\ the proof of
Corollary~\ref{co:type(c_k,c_l)}). Since $U\bot W_0$\,,
\begin{equation}\label{eq:g}
0=\g{u-\i\,I'u}{v}=\g{u}{v}
\end{equation} 
for all $u\in U'$ and $v\in W_0$\,, i.e. $W_0\bot U'$\,.
\end{proof}

\bigskip
\begin{corollary}\label{co:type(c_k',c_l')}
Let $W$ and  $U$ be of Types $(c_k')$ and $(c_\ell')$ 
determined by the data $(\Re,W',I')$ and $(\Re^*,U',J')$\,, respectively. 
If one of the following conditions holds, then $(W,U)$ is an orthogonal curvature invariant pair:
\begin{itemize}
\item $\Re=\Re^*$\,, $U'=W'$ and $I'=-J'$\;;
\item $\Re=\Re^*$ and $U'\bot W'$\;.
\end{itemize}
Moreover, every orthogonal curvature invariant pair of 
Type $(c_k',c_\ell')$ can be obtained in this way.
\end{corollary}
\begin{proof}
Note, if  $\Re=\Re^*$\,, $U'=W'$ and $I'=-J'$\,, then $U=\bar W$\,. 
Further, if $\Re=\Re^*$ and $U'\bot W'$\,, then $\C\, W'\bot \C\, U'$\,.
Thus the fact that these are orthogonal 
curvature invariant pairs follows by means of Lemma~\ref{le:type(c_k')}~(b). 

Conversely, the Hermitian structure  $I'$ extends to $W'\oplus\i W'$ 
(via complexification) and the linear space $W$ is  determined also by
the data $(e^{\i \varphi}\,\Re,e^{\i \varphi}\,W',I'|_{e^{\i \varphi}\,W'})$\,. 
Hence, we can assume that
$\Re=\Re^*$\,. In the following, we further suppose that $k\geq\ell$\,. 
If $k\geq 2$\,, then by means of Lemma~\ref{le:type(c_k')}, 
either $U \bot \C\, W'$ or $U=\bar W\oplus \tilde U$ with $\tilde U \bot \C\,W'$\,.
In the first case, obviously $W'\bot U'$\,. In the second case, we have $\tilde U=\{0\}$ 
(since $\ell\leq k$), i.e. $U=\bar W$\,.

In case $k=\ell=1$\,,  by means of Lemma~\ref{le:type(c_k')} we have $U=\tilde U\oplus U^\#$ 
with $\tilde U\subset \bar W$ and $U^\#\bot \C\,W'$\,.
Let $v\in W'$ and $\xi\in U'$ be given such that $v+\i\,I' v=\xi-\i\,J'\xi$\,. 
We obtain $v=\xi$ and $I'v=-J'v$\,. Thus, if $v\neq 0$\,, 
then  $W'=\{v, I'v\}_\R=\{\xi,J'\xi\}_\R=U'$ and $J'=-I'$ (since $k=\ell=1$). 
Therefore, the condition $\tilde U\neq \{0\}$ implies that $U=\bar W$\,.  This finishes the proof.
\end{proof}

\bigskip
\begin{corollary}\label{co:type(c_k')}
\begin{enumerate}
\item There are no orthogonal curvature invariant pairs of Type $(tr_{j,k},c_\ell')$
\item
There are no curvature invariant pairs of Type $(c_k',tr_1)$ for $k\geq 2$\,. 
\item
Let $W$ and $U$ be of Types $(c_1')$ and $(tr_1)$ determined by the data $(\scrR,W',I',W_0')$ 
and a unit vector $u\in \bar W$\,, respectively.
Then $(W,U)$ is a curvature-invariant pair. Moreover, 
any curvature invariant pair of Type $(c_1',tr_1)$ can be obtained in this way.
\end{enumerate}
\end{corollary}
\begin{proof}
For~(a), let $W$ and $U$ be of Types  $(tr_{j,k})$ and $(c_\ell')$ defined by the data 
$(\Re^*,W_1,W_2)$ and $(\Re,U',I')$\,. Then we can assume that $\Re=\Re^*$\,, cf.\ the 
proof of Corollary~\ref{co:type(c_k,c_l)}.
Therefore, the condition $W\bot U$ implies that 
\begin{align}\label{eq:orthogonal_3}
&0=\g{v_1}{u-\i\,I' u}=\g{v_1}{u}\;,\\ 
\label{eq:orthogonal_4}
&0=\g{\i\,v_2}{I'u+\i\,u}=\g{v_2}{u}
\end{align}
for all $v_1\in W_1$\,, $v_2\in W_2$ and $u\in U'$\,. Thus
$W_1$\,, $W_2$ and $U'$ are mutually orthogonal subspaces of $\Re$\,.
In particular, the linear space $W$ is contained in the orthogonal complement of $\C\,U'$\,.
We hence see by the $\frakh_U$-invariance of $W$ that the latter 
would be complex according to Lemma~\ref{le:type(c_k')}~(b), a contradiction.

For~(b) and~(c), according to Lemma~\ref{le:type(c_k')}~(b), the 1-dimensional 
subspace $\R u$ of $W^\bot$ is $\frakh_W$ invariant if and only if $k=1$ and $u\in \bar W$\,.
\end{proof}

\bigskip
\begin{lemma}\label{le:type(tr_k')}
Let $W$ be of Type $(tr_k')$ determined by the data $(\Re,W',I',W_0')$\,.
\begin{enumerate}
\item The Lie algebra $\frakh_W$ is given by $\Menge{A\in\fraku(W',I')}{A(W_0')\subset W_0'}$\,.
\item An $\frakh_W$-invariant subspace of $W^\bot$ is
  contained in the orthogonal complement of the complex space $\C\,W'$\,,
  belongs to a distinguished family of $k$-dimensional totally real subspaces
  of $\C\,W'\cap W^\bot$ -- which 
can be parameterized by the real projective space
  $\R\rmP^2$ (for $k\geq 3$) or the complex projective space $\C\rmP^2$ (for
  $k=2$) -- or is a direct sum of such spaces.
\item Let $A\in\so(\Re)$, $a\in\R$ be given and set $B:=a\, I'+A$\,. Then $B\in\so(\Re)$ holds and the linear map $a\,J^N+A$ leaves 
$W$ invariant if and only if $B(W_0')\subset W_0'$ and $B\, I'v=I'B\, v$ for all $v\in W_0'$\,.
\end{enumerate}
\end{lemma}

\begin{proof}
For~(a), we use that the curvature endomorphism $R^N_{u-\i\,I' u,v-\i\,I' v}$ is given by
  $-u\wedge v -I'u\wedge I'v$ for all $u,v\in W_0'$ according to~\eqref{eq:type(c_k')_3}. 
For~(b), in order to avoid any confusion in case $k=2$ (see below), 
we will temporarily drop the notation $\i\, v$ for $J^Nv$ with $v\in T_pN$\,. Thus, set $\lambda_0(v):=v+J^N\,I' v$\,,
$\lambda_1(v):=J^N\, v$ and $\lambda_2(v):=I' v$ for all $v\in W_0'$\,. Then
$\lambda_i$ is an isomorphism of $\frakh_W$-modules defined from $W_0'$ onto $\bar W$\,, $J^N(W_0')$ and $I'(W_0')$\,,
respectively. Therefore,  
\begin{equation}\label{eq:CW'cap W^bot=}
(W'\oplus J^N(W'))\cap W^\bot=\bar W\oplus J^N(W_0')\oplus I'(W_0')
\end{equation} is an orthogonal decomposition into three irreducible, pairwise
equivalent $\frakh_W$-modules each being isomorphic to $W_0'$\,. 
Moreover, we note that $\frakh_W|_{W_0'}=\so(W_0')$\,. 
Hence the linear space $W_0'$ is an irreducible $\so(W_0')$-module even over $\C$ for $k\geq 3$\,.
For  $k=2$\,, let $\{e_1,e_2\}$ be an orthonormal basis of $W_0'$ 
and consider the Hermitian structure on $W'$ given by 
\begin{equation}\label{eq:tildeI}
\tilde I:=e_1\wedge e_2+I' e_1\wedge I' e_2\;. 
\end{equation} 
Then $\tilde I$ extends to $W'\oplus J^N(W')$ (via complexification by $J^N$) such that 
$\frakh_W\subset\fraku(W'\oplus J^N(W'),\tilde I)$\,. 
Further, then $\tilde I(W_0')\subset W_0'$ and $\lambda_i$ commutes with $\tilde I$ for $i=0,1,2$\,.
Therefore, as was mentioned in Section~\ref{se:definition_of_cip}, 
there exists $(c_0:c_1:c_2)\in\bbK\rmP^2$ with $\bbK=\R$ (for $k\geq 3$) or $\bbK=\C$ (for $k=2$) such that
\begin{equation}\label{eq:lambda(W_0')=Menge}
U=\Menge{c_0\,v+c_2\,I' v +J^N\,(c_0 \,I' v+c_1\,v)}{v\in W_0'}
\end{equation}
(where in case $k=2$ multiplication with the complex numbers $c_i$ is now defined via $\tilde I$). 
Part~(b) follows. 

For~(c): since $W$ is totally real and the complexification $W\oplus\i W$ is of Type $(c_k')$ 
defined by the data  $(\Re,W',I')$\,, we  have $J^N|_{W}=I'|_W$ in accordance with~\eqref{eq:I'=pmi}. 
In particular, the linear map $J^N-I'$ leaves $W$ invariant, which reduces the question to the case $a=0$\,.
It remains to determine those $A\in\so(\Re)$ which leave the linear space
$W_0'$ invariant and satisfy $A\,I' v=I' A\,v$ for all $v\in W_0'$\,,
i.e.\ those for which $A(W')\subset W'$\,, $A|_{W'}\in\fraku(W')$ and $A(W_0')\subset W_0'$ holds. 
This proves our result.
\end{proof}

\bigskip
\begin{corollary}\label{co:type(tr_i',tr_kl)}
Let $W$ and $U$ be of Types $(\tr_j')$ 
and $(\tr_{k,\ell})$  defined by the data $(\Re,W',I',W_0')$ and $(\Re^*,U_1,U_2)$\,,
respectively. If $\Re=\Re^*$ and the linear space $U_1\oplus U_2$ 
is contained in the orthogonal complement of $W'$\,, then  $(W,U)$ is an orthogonal curvature
invariant pair. Every orthogonal curvature invariant pair of Type $(\tr_j',\tr_{k,\ell})$ can be obtained in this way.
\end{corollary}
\begin{proof}
Obviously, the pairs $(W,U)$ mentioned above satisfy $W\bot U$\,.
Further, the fact that these are curvature invariant pairs 
is verified by means of Lemmas~\ref{le:type(tr_kl)} and~\ref{le:type(tr_k')}, Parts~(a) and~(c). 
Conversely, let us see that our conditions are also necessary:

suppose that $(W,U)$ is an orthogonal curvature invariant pair. Note that $W$ is defined also by the data
$(e^{\i \varphi}\,\Re,e^{\i \varphi}\,W',I', e^{\i \varphi}\,W_0'(-\varphi))$ 
with $W_0'(-\varphi):=\Menge{\cos(\varphi)\,v-\sin(\varphi)\, I'v}{v\in W_0'}$
for every $\varphi\in \R$\,, hence we can assume that $\Re=\Re^*$\,. 
Since $U$ is $\frakh_W$-invariant, there exists a decomposition $U=U^\#\oplus
\tilde U$ into $\frakh_W$-invariant subspaces $U^\#\subset \C\,W'$  and $\tilde U\subset (\C\,W')^\bot$ 
according to Lemma~\ref{le:type(tr_k')}~(b)\,. We claim that the only possibilities are $U^\#=\{0\}$\,, 
$U^\#=\i W_0'$\,, $U^\#=I'(W_0')$ or  $U^\#=I'(W_0')\oplus \i W_0'$\,:

first, the condition $W\bot U$ implies that
$0=\g{u}{v-\i\,I'v}=\g{u}{v}
$ for all $u\in U_1$ and $v\in W_0'$\,. Hence $U_1\subset W_0'^\bot$\,, thus $U_1\cap W'\subset I'(W_0')$\,. 
Similarly, we can show that $U_2\cap W'\subset W_0'$\,.  Thus,
\begin{equation}\label{eq:type(tr_i',tr_kl)}
U^\#=U\cap \C\,W'=U_1\cap W'\oplus\i(U_2\cap W')\subset I'(W_0')\oplus \i W_0'\;.
\end{equation}
Since~\eqref{eq:CW'cap W^bot=} gives a decomposition of $\C W'\cap W^\bot$ 
into irreducible $\frakh_W$-modules, we conclude that 
$U_1\cap W'\in \{\{0\}\,, I'(W_0')\}$ and $U_2\cap W'\in\{\{0\}\,,W_0'\}$\,.
Our claim follows.

Next, we claim that $U^\#=\{0\}$\,: 

assume, by contradiction, that $I'(W_0')\subset U$\,. 
Since $\dim(W_0')\geq 2$\,, 
there exists an orthonormal pair of vectors $u,v\in W_0'$\,. Then $\{I'u,I'v\}\subset U\cap\Re=U_1$\,, hence
$A:=R^N_{I'u,I'v}$ leaves $W$ invariant since $(W,U)$ is a curvature invariant pair. Further, 
by means of~\eqref{eq:crg}, we have $A=-I'u\wedge I'v$\,. It follows, 
in particular, that $A\in\so(\Re)$ and $A|_{W_0'}=0$\,. 
Therefore, applying Lemma~\ref{le:type(tr_k')}~(c) (with $a=0$), 
we obtain that $A=0$ (since $W_0'$ is a real form of $(W',I')$), a contradiction.
A similar argument shows that neither $\i W_0'$ is contained in
$U$\,. We conclude that $U^\#=\{0\}$\,, i.e. $U\subset (\C\,W')^\bot$\,. 
Clearly, this implies that $U_1\oplus U_2\bot W'$\,, which finishes our proof.
\end{proof}

Spaces of Type $(tr_k')$ are neither 1-dimensional nor do they contain 
any complex subspaces. Hence Lemma~\ref{le:type(c_k')}~(b) implies:

\bigskip
\begin{corollary}\label{co:type(tr_k',c_j')}
There are no orthogonal curvature invariant pairs $(W,U)$ of Type  $(tr_k',c_\ell')$\,.
\qed
\end{corollary}

\bigskip
\begin{corollary}\label{co:type(tr_k',tr_l')}
Let $W$ and $U$ be of Types $(\tr_k')$ and $(\tr_\ell')$ defined 
by the data $(\Re,W',I',W_0')$ and $(\Re^*,U',J',U_0')$\,,
respectively. Further, in case $k=2$\,, let $\{e_1,e_2\}$ be an orthonormal basis of $W_0'$
and let $\tilde I$ be the Hermitian structure of $W'$ defined by~\eqref{eq:tildeI}.

If $\Re=\Re^*$ and one of the following conditions holds, 
then $(W,U)$ is an orthogonal curvature invariant pair:
\begin{itemize}
\item
we have $U'\bot W'$\,;
\item $k=\ell\geq 3$\,, $U'=W'$\,, $I'=J'$ and $U_0'=I'(W_0')$\,;
\item $k=\ell\geq 3$\,, $U'=W'$\,,  $I'=-J'$ and $U_0'=\exp(\theta\, I')(W_0')$ for some
  $\theta\in\R$\,\,;
\item  $k=\ell=2$\,,  $U'=W'$ and there exists some $\tilde J\in\SU(W',\tilde I)\cap\so(W')$ such that 
$U_0'=\tilde J(W_0')$ and $J'=\tilde J\circ I'\circ\tilde J^{-1}$\,.
\end{itemize} 
Moreover, every orthogonal curvature invariant pair of Type $(\tr_k',\tr_\ell')$ can be obtained in this way.
\end{corollary}
\begin{proof}
In the one direction, in order to see that the given pairs $(W,U)$ are actually curvature
invariant, we proceed as follows: 

the case $U'\bot W'$ is handled by means of
Lemma~\ref{le:type(tr_k')},~(a) and~(c). In the other cases, we have $U=J^N(W)$\,, 
$U=\exp(-\theta\, I')(\bar W)$ or $U=\tilde J(W)$\,, respectively.
If $U=\i W$\,, then 
\begin{equation}\label{eq:frakh_W=frakh_U1}
\frakh_U=\Menge{J^N\circ A\circ J^N}{A\in\frakh_W}=\frakh_W\;,
\end{equation}
where the first equality is straightforward and the second uses that $J^N$ commutes with any curvature endomorphism of $T_pN$\,.
If $U=\exp(-\theta\, I')(\bar W)$\,, then 
\begin{equation}\label{eq:frakh_W=frakh_U2}
\frakh_U=\Menge{\exp(-\theta\, I')\circ A\circ\exp(\theta\, I')}{A\in\frakh_{\bar W}}=\frakh_{\bar W}=\frakh_W\;,
\end{equation}
since $\frakh_W=\frakh_{\bar W}\subset \fraku(W',I')$ according to Lemma~\ref{le:type(tr_k')}~(a). If $k=2$\,, 
then  $\frakh_W=\R\,\tilde I$ according to 
Lemma~\ref{le:type(tr_k')}~(a) and~\eqref{eq:tildeI}, hence, with $U=\tilde J(W)$\,,
\begin{equation}\label{eq:frakh_W=frakh_U3}
\frakh_U=\Menge{\tilde J\circ A\circ \tilde J}{A\in\frakh_W}=\R\, \tilde J\circ \tilde I\circ
\tilde J\stackrel{\tilde J\in\SU(W',\tilde I)}{=}\R\,\tilde I=\frakh_W\;.
\end{equation}
Moreover, if $\frakh_W=\frakh_U$\,, then $(W,U)$ 
is a curvature invariant pair (by the curvature invariance of both $W$ and $U$). This shows that
the pairs in question are actually curvature invariant pairs.

It remains to verify that $U \bot W$\,. This is straightforward in case $U'\bot W'$\,. Further, we have 
$W\bot \i W$ (since $W$ is totally real) and $e^{-\i \theta}\,\bar W\bot W$ for any $\theta$ 
(since even $\C\,\bar W\bot \C W$\,, see Corollary~\ref{co:type(c_k',c_l')}).
If $k=2$\,, then $f_1:=e_1$ and $f_2:=I' e_1$ defines a Hermitian basis of
$(W',\tilde I)$\,. Consider the complex matrix $(g_{ij})$ defined by
\begin{equation}\label{eq:gij_1}
g_{ij}:=\g{f_i}{\tilde J\,f_j}+\i\,\g{\tilde I\, f_i}{\tilde J\,f_j}\;:
\end{equation}
Then $(g_{ij})$ belongs to $\SU(2)\,\cap\,\su(2)$\,, hence there exist $t\in\R$ and $w\in\C$
with $t^2+|w|^2=1$ such that 
\begin{equation}\label{eq:gij_2}
\left ( \begin{array}{cc}
g_{11}&g_{12}\\
g_{21}&g_{22}
\end{array}
\right )=
\left ( \begin{array}{cc}
\i\,t &-\bar w\\
w&-\i\,t
\end{array}
\right )
\end{equation} 
holds. Using the skew-symmetry of $\tilde J$ and~\eqref{eq:gij_2}, we calculate
\begin{align}\label{eq:letzte_Gleichung}
&\g{e_i-J^N\,I' e_i}{\tilde J(e_i-J^N\,I' e_i)}=\g{e_i}{\tilde J\,e_i}+\g{I' e_i}{\tilde J\,I' e_i}=0+0=0\
\text{for}\ i=1,2\;,\\
\label{eq:allerletzte_Gleichung}
&\g{e_2-J^N\,I' e_2}{\tilde J(e_1-J^N\,I' e_1)}=\g{\tilde I\,f_1}{\tilde J\,f_1}+\g{\tilde I\,f_2}{\tilde J\, f_2}=\Im(g_{11}+g_{22})=0\;,\\
\label{eq:allerallerletzte_Gleichung}
&\g{e_1-J^N\,I' e_1}{\tilde J(e_2-J^N\,I' e_2)}=-\g{e_2-J^N\,I' e_2}{\tilde J(e_1-J^N\,I' e_1)}=0\;.
\end{align} 
This shows that $W\bot \tilde J(W)$\,.

In the other direction, let $(W,U)$ be an orthogonal 
curvature invariant pair of Type $(\tr_k',\tr_\ell')$  defined by the data $(\Re,W',I',W_0';\Re^*,U',J',U_0')$\,. 
Then we can assume that $\Re=\Re^*$ (cf.\ the proof of
Corollary~\ref{co:type(tr_i',tr_kl)}). Clearly, we can also suppose that $\ell\leq k$\,. 
Therefore, since $U$ is $\frakh_W$-invariant with $\dim(U)\leq k$\,, either $U\bot \C\,W'$  
or there exists $(c_0:c_1:c_2)\in\bbK\rmP^2$ with $\bbK=\R$ (for $k\geq 3$) or $\bbK=\C$ (for $k=2$) 
such that $U$ is given by r.h.s.\ of~\eqref{eq:lambda(W_0')=Menge} according to  Lemma~\ref{le:type(tr_k')}~(b).

Suppose that $U\bot \C\,W'$\,. Then
\begin{align}\label{eq:orthogonal_1}
&0=\g{u-J^N\,I' u}{v}=\g{u}{v}\;,\\
\label{eq:orthogonal_2}
&0=\g{u-J^N\,I' u}{J^N\,v}=-\g{I' u}{v}
\end{align}
for all $u\in U_0'$ and $v\in W'$\,, i.e.\ we obtain that $W'\bot U'$\,. 

We are left with the case that there exists $(c_0:c_1:c_2)\in\bbK\rmP^2$ such that $U$ is given by r.h.s.\ 
of~\eqref{eq:lambda(W_0')=Menge}. In particular, then $U\subset \C\, W'$\,, 
i.e. $U_0'\subset \Re\cap \C\,W'=W'$ and $J'(U_0')\subset W'$\,, hence $U'=U_0'\oplus J'(U_0') =W'$\,. 
Furthermore, we claim that here $\frakh_W=\frakh_U$\,:
 
given $A\in\frakh_U$\,, 
by means of Lemma~\ref{le:type(tr_k')}~(a), we have, in particular, $A\in\so(W')$\,.
Further, we have $A(W)\subset W$ since $(W,U)$ is assumed to be a curvature invariant pair.
Thus, we obtain from Lemma~\ref{le:type(tr_k')}~(c) (with $a=0$) that 
$A\in\fraku(W,I')$ and $A(W_0')\subset W_0'$\,. Then $A\in\frakh_W$ 
again by means of  Lemma~\ref{le:type(tr_k')}~(a). This shows 
that $\frakh_U\subset \frakh_W$ holds. The other inclusion is proved in a similar way. 
This gives our claim.

For $k\geq 3$\,, we claim that $U=\i W$  or $U=\e^{\i \theta}\,\bar W$ 
for some $\theta\in\R$\,: 

taking real and imaginary parts in~\eqref{eq:lambda(W_0')=Menge}, we obtain that 
\begin{equation}\label{eq:71}
u:=c_0\,v+c_2\,I' v
\end{equation}
belongs to $U_0'$ for every $v\in W_0'$ and
\begin{equation}
\label{eq:81}
J'u=-c_1\,v-c_0 \,I'v\;.
\end{equation}
Moreover, any $u\in U_0'$ can be uniquely obtained from some $v\in W_0'$ via~\eqref{eq:71}.
Further, assume that $v$ is a unit vector. Therefore, preparing our notation already for the case $k=2$ (see below), 
\begin{align}\label{eq:reelle_Koeffizienten1}
&\lvert c_0\rvert^2+\lvert c_2\rvert^2\,\stackrel{\eqref{eq:71}}{=}\,\lvert u\rvert^2=\lvert J' u\rvert^2
 \stackrel{\eqref{eq:81}}{=}\lvert c_0\rvert^2+\lvert c_1\rvert^2\;,\\
\label{eq:reelle_Koeffizienten2}
&0=\g{u}{J'\, u}=-\bar c_0\,c_1-\bar c_2\, c_0\;.
\end{align}
Then $c:=\lvert c_0\rvert^2+\lvert c_1\rvert^2$ does not vanish 
(otherwise $c_0=c_1=c_2=0$ according to~\eqref{eq:reelle_Koeffizienten1} which is not allowed).
Thus, we can assume that $c=1$ (since we consider only the ratio $(c_0:c_1:c_2)$).
Then~\eqref{eq:reelle_Koeffizienten1} becomes
\begin{equation}\label{eq:reelle_Koeffizienten3}
\lvert c_0\rvert^2+\lvert c_2\rvert^2=\lvert c_0\rvert^2+\lvert c_1\rvert^2=1\;
\end{equation}
Therefore, by means of~\eqref{eq:reelle_Koeffizienten2},\eqref{eq:reelle_Koeffizienten3}, 
the matrix $(g_{ij})$ defined by
\begin{equation}\label{eq:gij_3}
\left ( \begin{array}{cc}
g_{11}&g_{12}\\
g_{21}&g_{22}
\end{array}
\right ):=\left ( \begin{array}{cc}
c_0 &-c_1\\
c_2&-c_0
\end{array}
\right )
\end{equation}
belongs $\rmO(2)$\,. If $g\in\SO(2)$\,, then $c_0=0$ and $c_1=c_2=\pm 1$\,. 
Then~\eqref{eq:71},\eqref{eq:81} imply that $U_0'=I'(W_0')$ and $J'=I'$\,.
Otherwise, there exists $\theta\in\R$ 
such that $c_0=\cos(\theta)$ and $c_2=-c_1=\sin(\theta)$\,, hence $J'=-I'$ and 
$U_0'=\Menge{\cos(\theta)+\sin(\theta)\,I'v}{v\in W_0'}$ 
according to~\eqref{eq:71},\eqref{eq:81}. This finishes the proof for $k\geq 3$\,. 

For $k=2$\,, we first recall that $\tilde I$ equips the linear space $W'$ with a second Hermitian structure such that $\tilde I(W_0')\subset W_0'$ and $I'$ belongs to $\rmU(W',\tilde I)$\,.
Now it is straightforward by means of~\eqref{eq:71},\eqref{eq:81} that also $\tilde I(U_0')\subset U_0'$ and $J'\in\rmU(W',\tilde I)$\,. 
Then it follows on the analogy of~\eqref{eq:reelle_Koeffizienten1}-\eqref{eq:reelle_Koeffizienten3}
that the complex matrix $(g_{ij})$ defined by~\eqref{eq:gij_3} belongs $\rmU(2)$\,. Moreover, 
since our considerations depend only on the complex ratio $(c_0:c_1:c_2)$\,, 
we can even assume that $(g_{ij})$ belongs $\SU(2)$\,.
Then necessarily $c_0=-\bar c_0$ and $c_1=\bar c_2$\,, hence $(g_{ij})$ takes 
the form~\eqref{eq:gij_2} which implies that 
$(g_{ij})\in\SU(2)\,\cap\,\su(2)$\,. Further, recall that $f_1:=e_1$ and $f_2:=I' e_1$ defines a Hermitian basis of
$(W',\tilde I)$\,. Thus, we obtain a unique element of $\SU(W',\tilde I)\cap\so(W')$
via $\tilde J f_i:=g_{1i}\,f_1+g_{2i}\,f_2$\,. 
Then, using the previous and~\eqref{eq:71},\eqref{eq:81}, we conclude that $U_0'=\tilde J(W_0')$ and $\tilde J\circ I' =J'\circ\tilde J$\,. 
The details of this part of the proof are left to the reader.
\end{proof}

\bigskip
\begin{corollary}\label{co:type(tr_k',line)}
Let $W$ and $U$ be of Types $(tr_k')$
and $(tr_1)$ defined by the data $(\Re,W',I',W_0')$ and a unit vector $u$ of $T_pN$\,,
respectively. The pair $(W,U)$ is an orthogonal curvature invariant pair if and only if $u$ belongs to $(\C\,W')^\bot$\,.
\qed
\end{corollary}

\bigskip
\begin{lemma}\label{le:type(ex_3)}
Let $W$ be of Type $(ex_3)$ defined by the data $(\Re,\{e_1,e_2\})$\,.
\begin{enumerate}
\item  The Lie algebra $\frakh_W$ is the linear space which is generated by $J^N +e_1\wedge e_2$\,.
\item A subspace of $W^\bot$ is $\frakh_W$-invariant if and only if it is the
  1-dimensional space $\R(e_2-\i\,e_1)$\,, a complex subspace 
of the orthogonal complement of $\{e_1,e_2\}_\C$\,, or a sum of such spaces. 
\item Let $A\in\so(\Re)$ and $a\in\R$ be given. The linear map $a\,J^N+A$ leaves 
$W$ invariant if and only if $A-a\,e_1 \wedge e_2$ vanishes on $\{e_1,e_2\}_\R$\,.
\end{enumerate}
\end{lemma}
\begin{proof}
Consider the Hermitian structure $I':=e_1\wedge e_2$ on $W':=\{e_1,e_2\}_\R$ and put 
$x_1:=e_1-\i\,e_2$\,, $x_2:=e_2+\i\,e_1$ and $x_3:=e_1+\i\,e_2$\,. A straightforward calculation shows that $R^N_{x_1,x_3}=R^N_{x_2,x_3}=0$\,.
Further, let $\tilde W$ be the curvature invariant space of Type $(c_1')$ defined by $(\Re,W',I')$\,.
Thus $W=\tilde W\oplus\R\,x_3$\,, hence $\frakh_W=\frakh_{\tilde W}$\,. Now Part~(a) follows from Lemma~\ref{le:type(c_k')}~(a) (with $k=1$). 
Clearly, the intersection $\C\, W'\cap W^\bot$ is given by $\R(e_2-\i\,e_1)$\,. Thus Part~(b) follows from Lemma~\ref{le:type(c_k')}~(b) (with $k=1$). 
For~(c), since $J^N +I'$ leaves $W$ invariant (by means of~(a) and since $W$ is curvature invariant),
we can assume that $a=0$\,. If $A$ leaves $W$ invariant, then $A\,x_1=A\,e_1-\i\,A\,e_2$ is necessarily a linear combination of $x_2$ and
$x_3$\,, say $A\,x_1=c\,x_2+d\,x_3$\,. It follows that 
$$d=\g{c\,x_2+d\,x_3}{e_1}=\g{A\,x_1}{e_1}=\g{A\,e_1-\i\,A\,e_2}{e_1}=\g{A\,e_1}{e_1}=0\;,$$ 
hence $A\,x_1=c\,x_2$\,, i.e. $A\,e_1=c\,e_2$ and $A\,e_2=-c\,e_1$\,. 
Thus 
\[
W\ni A\,x_3=A\,e_1+\i\,A\,e_2=c(e_2-\i\,e_1)\in W^\bot\;,
\] 
hence $A\, x_3\in W\cap W^\bot=\{0\}$\,. It follows that $c=0$\,. This implies that $A\,e_1=A\,e_2=0$ which proves our claim.
\end{proof}

\bigskip
\begin{corollary}\label{co:type(ex_3,line)}
Let $W$ and $U$ be of Types $(ex_3)$ and $(tr_1)$ defined 
by the data $(\Re,\{e_1,e_2\})$ and a unit vector $u$ of $T_pN$\,, respectively.
Then $(W,U)$ is an orthogonal curvature invariant pair 
if and only if $u=\pm\, 1/\sqrt 2\, (\e_2-\i\,e_1)$\,.
\qed
\end{corollary}

\bigskip
\begin{corollary}\label{co:type(ex_3,other)}
There do not exist any orthogonal curvature invariant pairs of Types $(ex_3,c_k')$\,, 
$(ex_3,tr_k')$\,, $(ex_3,tr_{k,\ell})$ and $(ex_3,ex_3)$\,.
\end{corollary}
\begin{proof}
Let $W$ be of Type $(ex_3)$ defined by the data  $(\Re,\{e_1,e_2\})$ and $U$ be a subspace of $W^\bot$ 
such that $(W,U)$ is a curvature invariant pair. Recall that the $\frakh_W$-invariance of $U$ implies that 
there is the splitting $U=\tilde U\oplus U^\#$ 
into a totally real space $U^\#\subset \R(e_2-\i\, e_1)$ 
and a complex subspace $\tilde U$ of the orthogonal complement of $\{e_1,e_2\}_\C$
according to Lemma~\ref{le:type(ex_3)}~(b).

Hence, if $U$ is of Type $(c_k')$ defined by the data $(\Re^*,U',I')$\,, then $U^\#=\{0\}$ 
(since $U$ is complex) and thus $U\bot\{e_1,e_2\}_\C$\,. 
Further, we can assume that $\Re^*=\Re$\,. Thus $\{e_1,e_2\}_\R\bot U'$ (see~\eqref{eq:g}). 
Therefore, we obtain that $W\bot \C\,U'$ and whence the $\frakh_U$-invariant space $W$ is complex according to Lemma~\ref{le:type(c_k')}~(b), which is not given. 

Furthermore, if $U$ is of Type $(tr_k')$ or $(tr_{k,\ell})$\,, 
then $\tilde U=\{0\}$ (since $U$ is totally real) and hence $U$ is at most 1-dimensional, which is not given.

If $U$ is of Type $(ex_3)$\,, too, defined 
by  $(\Re^*,\{f_1,f_2\})$\,, then $U$ is defined also by
$(e^{\i\varphi}\,\Re^*,\{f_1(\varphi),f_2(\varphi)\})$ with
  $f_1(\varphi):=e^{\i\varphi} (\cos(\varphi) f_1+\sin(\varphi)f_2)$ and 
$f_2(\varphi):=e^{\i\varphi} (-\sin(\varphi) f_1+\cos(\varphi)f_2)$\,. Hence we
can assume that $\Re=\Re^*$\,. Further, an orthogonal decomposition $U=U^\#\oplus\tilde U$ into a totally real
space $U^\#$ and a complex space $\tilde U$ is unique (if it exists). We conclude that
 $U^\#=\R(\i\,f_2+f_1)$ and $\tilde U=\{f_1-\i\,f_2\,, f_2+\i\,f_1\}_\R$\,.
 Thus, on the one hand, $\{f_1,f_2\}_\R\bot \{e_1,e_2\}_\R$\,. On the other hand,
$\i\,f_2+f_1=\pm (e_2-\i\,e_1)$\,, a contradiction.
\end{proof}

\bigskip
\begin{lemma}\label{le:type(ex_2)}
Let $W$ be of Type $(ex_2)$ defined by the data $(\Re,\{e_1,e_2,e_3\})$\,.
\begin{enumerate}
\item The Lie algebra $\frakh_W$ is the linear space which is generated by
$J^N +e_1 \wedge e_2 +\sqrt{3}\,e_2\wedge e_3$\,. 
\item A subspace $U$ of $W^\bot$ is $\frakh_W$-invariant if and only
 if it is the complex space $\C(-e_1+\sqrt{3}\,e_3+2\,\i\,
 e_2)$\,, belongs to a distinguished family 
of (real) 2-dimensional subspaces of the linear space \begin{equation}\label{eq:pmi}
\{2\,e_2+\i(-3\,e_1+1/\sqrt 3\,e_3),e_1+5/\sqrt{3}\,e_3-2\,\i\,e_2\}_\R\oplus\{e_1,e_2,e_3\}_\C^\bot\;,\end{equation}
or is a sum of such spaces.
\item Let $A\in\so(\Re)$ and $a\in\R$\,. 
The linear map $a\,J^N+A$ leaves $W$ invariant if and only if  $A-a(e_1 \wedge
e_2+\sqrt{3}\,e_2\wedge e_3)$ vanishes on $\{e_1,e_2,e_3\}_\R$\,.
\end{enumerate}
\end{lemma}
\begin{proof}
For~(a), set $x_1:=2\,e_1+\i\,e_2$ and $x_2:=e_2+\i(e_1+\sqrt{3}\,e_3)$\,.
A straightforward calculation shows that the curvature endomorphism 
$R^N_{x_1,x_2}$ is given by $-J^N-e_1 \wedge e_2 -\sqrt{3}\,e_2\wedge e_3$\,. 

For~(b), we first verify that the 
eigenvalues of $A:=R^N_{x_1,x_2}$ (seen as a complex-linear endomorphism of
$T_pN$) are given by $\{\i, -\i,-3\,\i\}$\,. The complex eigenspace for the eigenvalue $-3\i$ is a subspace of $W^\bot$\,,
given by $\C(-e_1+\sqrt{3}\,e_3+\i\,2\,e_2)$\,. Furthermore, we have  $A^2=-\Id$  on the $(2n-4)$-dimensional 
subspace of $T_pN$ which is given by~\eqref{eq:pmi}, i.e.\ the linear map $A$
defines a second complex structure on~\eqref{eq:pmi}. This proves~(b). 

For~(c): since $W$ is curvature
invariant, the endomorphism $J^N +e_1 \wedge e_2 +\sqrt{3}\,e_2\wedge
e_3$ leaves $W$ invariant. This reduces the problem to the case $a=0$\,. If
$A(W)\subset W$\,, then $A\,x_1=c\,x_2$ and $A\,x_2=-c\,x_1$ for some $c\in \R$ (since $A$
is skew-symmetric and $\left\Vert x_1\right\Vert=\sqrt{5}=\left\Vert x_2\right\Vert$\,).
Considering the action of $A$ on the real and imaginary parts of $x_1$ and $x_2$\,, respectively, this implies that $A\,e_2=-2\,c\,e_1$
and  $A\,e_2=c\,(e_1+\sqrt{3}\,e_3)$\,, a contradiction unless $c=0$\,. 
Thus $A\,x_1=A\,x_2=0$ and hence $A|_{\{e_1,e_2,e_3\}_\R}=0$ since $A\in\so(\Re)$\,. This  finishes the proof.
\end{proof}

\bigskip
\begin{corollary}
If $W$ is of Type $(ex_2)$\,, then there are no orthogonal
curvature invariant pairs $(W,U)$ at all.
\end{corollary}
\begin{proof}
Let $\Re\in\scrU$ and an orthonormal system $\{e_1,e_2,e_3\}$ of $\Re$ be 
given such that $W$ is spanned by $x_1:=2\,e_1+\i\,e_2$ 
and $x_2:= e_2+\i(e_1+\sqrt{3}\,e_3)$\,. Suppose further, by contradiction,
that there exists some curvature invariant subspace $U$ of $T_pN$ such that
$(W,U)$ is an orthogonal curvature invariant pair.

For Type $(c_k,ex_2)$\,, see Corollary~\ref{co:type(tr_ij,c_k)}.
If $U$ is of Type $(c_k')$ or $(ex_3)$\,, then $W$ is a 2-dimensional
$\frakh_U$-invariant subspace of $U^\bot$ but not a complex subspace of $T_pN$ according to Lemma~\ref{le:type(ex_3)}~(c).
However, this is not possible, because of Parts~(b) of Lemmas~\ref{le:type(c_k')} and~\ref{le:type(ex_3)}, respectively.

Now suppose that $U$ is of Type $(tr_{i,j})$ determined by the data $(\Re^*,U_1,U_2)$\,. 
Using Lemmas~\ref{le:type(tr_kl)}~(c) and~\ref{le:type(ex_2)}~(a), we see that
$\frakh_W(U)\subset U$ does not hold.

Similarly, the case that $U$ is of Type $(tr_1)$ can not occur.

Suppose that $U$ is of Type $(tr_k')$ determined by the quadruple
$(\Re^*,U',I',U_0')$\,. Then we can assume that $\Re=\Re^*$\,. 
Using Lemma~\ref{le:type(tr_k')}~(b), the fact that $W$ is 2-dimensional
linear subspace of $T_pN$ which is invariant under $\frakh_U$  implies that
either $W\subset\C\,U'^\bot$ or $W$ is a 2-dimensional
$\frakh_U$-invariant subspace of $\C\,U'$\,. 

In the first case, we have $\g{\Re(x_i)}{u}=\g{\Im(x_i)}{u}=0$ for all $u\in U'$ and $i=1,2$\,. With $i=1$\,, it follows
that $\g{e_1}{u}=\g{e_2}{u}=0$\,, then the previous with $i=2$ implies that also
$\g{e_3}{u}=0$ for all $u\in U'$\,. Thus Lemma~\ref{le:type(ex_2)}~(a) and the fact
that $\frakh_W(U)\subset U$ show that $U$ is a complex subspace of $T_pN$\,, a contradiction.

In the second case, we have $\dim(U_0')=2$\,, hence $\dim(U')=4$\,. Further, both $\Re(x_i)$ and $\Im(x_i)$
belong to $U'$ for $i=1,2$\,. Thus we conclude that $\{e_1,e_2,e_3\}\subset U'$\,. 
Let $\{u_1,u_2\}$ be an orthonormal basis of $U_0'$\,. According to Lemma~\ref{le:type(tr_k')}, 
the curvature endomorphism $R^N_{u_1-\i\,I' u_1,u_2-\i\,I' u_2 }$ is given by $A:=-u_1\wedge u_2-I' u_1\wedge I' u_2$\,.
Hence, since  $(W,U)$ is a curvature invariant pair, we obtain that $A(W)\subset W$\,. 
Using Lemma~\ref{le:type(ex_2)}~(c) (with $a=0$), we obtain that $A$ vanishes on $\{e_1,e_2,e_3\}_\R$\,. 
Therefore, since $A\in\so(U')$\,, the rank of $A$ would be at most one, which is not possible unless
$A=0$\,, a contradiction.

Consider the case that $U$ is of Type $(ex_2)$\,, too.
Then there exists some $\Re^*\in\scrU$ and an orthonormal 
system $\{f_1,f_2,f_3\}$ of $\Re^*$ such that $U$ is spanned by $u_1:=2\,f_1
+\i\,f_2$ and $u_2:=f_2+\i(f_1+\sqrt{3}\, f_3)$\,. 
Let $\varphi$ be chosen such that $e^{\i \varphi}\,\Re^*=\Re$\,.
In accordance with Lemma~\ref{le:type(ex_2)},
the curvature endomorphism $R_{1,2}:=R^N_{u_1,u_2}$ is given by $-J^N +A$ with
$A:=-f_1\wedge f_2 -\sqrt{3}\,f_2\wedge f_3$\,. 
We decompose $f_i=f_i^\top+f_i^\bot$ such that $f_i^\top\in e^{-\i \varphi}\{e_1,e_2,e_3\}_\R$ and $f_i^\bot\bot e^{-\i \varphi}\{e_1,e_2,e_3\}_\R$\,.  
Since $R_{1,2}(W)\subset W$\,,  Lemma~\ref{le:type(ex_2)}~(c) (with $a=-1$) 
shows that $$e_1 \wedge e_2+\sqrt{3}\,e_2\wedge e_3
=f_1^\top\wedge f_2^\top+\sqrt{3}\,f_2^\top\wedge f_3^\top$$
(both sides seen as elements of $\fraku(T_pN)$). 
Comparing the length of the tensors on the left and right hand side above, we
see that $$\lvert f_1^\top\rvert=\lvert f_2^\top\rvert=\lvert f_3^\top\rvert=1,$$
i.e. $e^{\i \varphi}f_i\in \{e_1,e_2,e_3\}_\R$ for $i=1,2,3$\,. Hence we can assume that $n=3$\,. 
Since $U$ is $\frakh_W$-invariant but not complex, it follows from  Lemma~\ref{le:type(ex_2)}~(b) that $U$
is the linear space spanned by  $\tilde u_1:=2\,e_2+\i(-3\,e_1+1/\sqrt 3\,e_3)$ and $\tilde u_2:=e_1+5/\sqrt{3}\,e_3-2\,\i\,e_2$\,.
A short calculation shows that the curvature endomorphism  $R^N_{\tilde u_1,\tilde u_2}$ is given by $8/3\, J^N -4(e_1\wedge
e_2 +\sqrt{3}\,e_2\wedge e_3)$\,. Thus we obtain that  $\frakh_U$ does not leave $W$ invariant. Hence $(W,U)$
is not a curvature invariant pair.
\end{proof}

\subsection{Integrability of the curvature invariant pairs of $\rmG^+_2(\R^{n+2})$}
\label{se:integrability}
Let $(W,U)$ be an orthogonal curvature invariant pair of $\rmG^+_2(\R^{n+2})$ such that $\dim(W)\geq 2$\,.
It remains the question whether $(W,U)$ or $(U,W)$ is integrable. 
By means of a case by case analysis of the possible pairs (see Table~\ref{table}), 
we will show that the answer is ``no'' unless $V:=W\oplus U$ is curvature invariant.

Let $\frakk$ denote the isotropy Lie algebra of $N:=\rmG^+_2(\R^{n+2})$ and 
$\rho:\frakk\to \so(T_pN)$ be the linearized isotropy representation. 
Recall that $\rho(\frakk)=\R\, J^N\oplus \so(\Re)$\,.
Further, by definition, the Lie algebra $\frakk_V$ is the maximal subalgebra of $\frakk$ such that
$\rho(\frakk_V)|_V$ is a subalgebra of $\so(V)$\,, see~\eqref{eq:frakk_V}.

\paragraph{Type ${\bf (c_k,c_\ell)}$}
Let $W$ and $U$ be of Types $(c_{k})$ and $(c_{\ell})$ 
defined by the data $(\Re, W_0)$ and $(\Re^*,U_0)$\,, respectively. If $(W,U)$ is a curvature invariant pair, 
then the only possibility is $\Re=\Re^*$ and $W_0\bot U_0$\,. Then $V$ is curvature invariant of 
Type $(c_{k+\ell})$ defined by the data $(\Re, W_0\oplus U_0)$\,.

\paragraph{Type ${\bf (tr_{i,j},tr_{k,\ell})}$}
Let $W$ and $U$ be of Types $(tr_{i,j})$ and $(tr_{k,\ell})$
defined by the data $(\Re,W_1,W_2)$ and $(\Re^*,U_1,U_2)$\,, respectively.
Let $\varphi$  be chosen such that  $\Re=e^{\i \varphi}\,\Re^*$\,. Substituting, if necessary, 
$\i\,\Re^*$ for $\Re^*$\,, we can assume that $\varphi\in[-\pi/4,\pi/4]$\,.

\begin{itemize}
\item Case $i=j=1$\,. Suppose that $(W,U)$ is integrable and let $M$ be 
a parallel submanifold through $p$ such that $T_pM=W$ and $\bot^1_pM=U$\,.
Since the sectional curvature of $W$ vanishes, according to Corollary~\ref{co:circles} there exists a 
simply connected totally geodesic submanifold $\bar M\subset N$\,, a Riemannian splitting $\bar M=M_1\times M_2$ 
and extrinsic circles $C_i\subset M_i$ such that $M=C_1\times C_2$\,. 
Recall that $\rmG_2^+(\R^{4})\cong \C\rmP^1\times\C\rmP^1$ whereas 
the symmetric space $\rmG_2^+(\R^{n+2})$ is irreducible if 
$n\geq 3$ since then its root-system is of Type $B_n$ (see~\cite{sebastian1})\,. 
Further, any rank-one symmetric space is irreducible, too. 
Therefore, according to Theorem~\ref{th:ci}, the only possibilities are $\bar M=\rmS^k\times\rmS^\ell$ ($k,\ell\geq 2$)
or $\bar M=\C\rmP^1\times\C\rmP^1$\,. 
In the first case, applying reduction of the codimension to each factor, we can even assume that $k=\ell=2$.
Therefore, $\dim(\bar M)=4$ and hence $V=T_p\bar M$ is curvature invariant  of Type $(tr_{k,\ell})$ or $(c_2)$\,.
\qed

In the remaining cases, at least one of the indices $\{i,j\}$ is strictly greater than $1$ 
and hence (possibly after substituting $\i\,\Re$ for $\Re$), 
we can suppose that $i\geq 2$\,. Then we have to consider the possibilities 
$\Re=\Re^*$ and $W_1=U_2$\,, or $W_1\bot e^{\i \varphi}\,U_2$\,.

\item 
Case $i=\ell\geq 2$\,, $\Re=\Re^*$ and $W_1=U_2$\,. Here we have 
$W_2\bot U_1$\,, or $W_2=U_1$\,, or $j=k=1$\,.
In case $W_2=U_1$\,, the linear space $V$ is curvature invariant of Type $c_{k+\ell}$ 
defined by $(\Re,U_1\oplus W_1)$\,. Otherwise, we claim that $\rho(\frakk_V)|_V\cap\so(V)_-=\{0\}$\,:

let $a\in\R$\,, $B\in\so(\Re)$\,, set $A:=a\, J^N+B$ and suppose that $A(V)\subset V$ and $A|_V\in\so(V)_-$ holds. 
Then $A(W)\subset U$ and $A(U)\subset W$\,.
We aim to show that $A=0$\,. Let $v_2\in W_2$\,. Thus $A\,\i\,v_2\in U$\,. 
It follows that $a\,v_2\in U_1$ and $B\, v_2\in U_2$\,. 
In the same way, 
$a\,u_1\in W_2$ and $B\, u_1\in W_1$ for all $u_1\in U_1$\,. 
Hence $a=0$\,, since $W_2=U_1$ would be a different case. Further, setting $V_i:=W_i\oplus U_i$ for $i=1,2$\,, we have $A|_{V_i}\in\so(V_i)_-$\,.
Since the maps $\so(V_1)_-\to \Hom(W_1, U_1), A\mapsto A|_{W_1}$ and $\so(V_2)_-\to \Hom(U_2, W_2), A\mapsto A|_{U_2}$ 
both are linear isomorphisms according to~\eqref{eq:iso1}, 
for the vanishing of $A$ it suffices to show that $A|_{W_1}=0$ and $A|_{U_2}=0$\,:

on the one hand, $A(W_1)=A(U_2)\subset W_2$\,.
On the other hand, $A(W_1)\subset U_1$ since $A\in\so(V_1)_-$\,. Hence $A(W_1)\subset W_2\cap U_1$\,. 
Further, the linear space $W_2\cap U_1$ is trivial if $W_2\bot U_1$ or if $j=k=1$ and $W_2\neq U_1$\,.
Therefore, $A|_{W_1}=0$ unless $W_2=U_1$\,.
Similar considerations show that also $A|_{U_2}=0$ unless $W_2=U_1$\,. This establishes our claim.

Assume, by contradiction, that $(W,U)$ is integrable but $W_2\neq U_1$\,. Thus $\rho(\frakk_V)|_V\cap\so(V)_-=\{0\}$\,.
Therefore, according to Corollary~\ref{co:dec}, there exists a symmetric bilinear map $h:W\times W\to W^\bot$ whose image spans $U$ 
and which satisfies~\eqref{eq:dec2}. Note, the Lie algebra $\frakh$~\eqref{eq:La2} is given by $\so(W_1)\oplus \so(W_2)\oplus\so(U_1)$\,. 
Then $\so(W_2)\oplus\so(U_1)$ is the direct sum representation on
$\i W_2\oplus U_1$ and $\so(W_1)$ acts diagonally on $W_1\oplus \i W_1$ (i.e. $A(v_1+\i\,v_2)=A\, v_1+\i\,A\,v_2$ 
for all $A\in \so(W_1)$ and $(v_1,v_2)\in W_1\times W_2$).
In particular, each of the linear spaces $W_1,\, \i W_1,\, U_1$ and $\i\,W_2$ is an $\frakh$-module and 
the induced $\frakh$-action on both $W_1$ and $\i\,W_1$ is non-trivial and irreducible (since $i=\ell\geq 2$).
Therefore, Schur's Lemma implies that $\Hom_\frakh(W,U)\subset\Hom_\frakh(W_1,\i\,W_1)\oplus\Hom_\frakh(\i W_2,U_1)$\,.  
We conclude from the previous that $h(x,y)\in \i\,W_1$  and $h(x,\i\,z)\in U_1$ for all $x\in W$\,, $y\in W_1$  and $z\in W_2$\,, hence 
\begin{align}\label{eq:extrinsic_product_1}
&h(W_1\times \i W_2)=h(\i W_2\times W_1)\subset U_1\cap \i\,W_1=\{0\}\;,\\
\label{eq:extrinsic_product_2}
&\i\,W_1=\Spann{h(x,x)}{x\in W_1}\ \text{and}\ U_1=\Spann{h(x,x)}{x\in \i W_2}\;.
\end{align} 
We claim that $W_2=\{0\}$\,:

let $x,y\in W_1\times W_2$\,. Then $R^N_{x,\i\, y}=0$ (by the condition $W_1\bot W_2$\,, see~\eqref{eq:crg}) and hence~\eqref{eq:cond2} (with $k=1$) 
yields
\begin{equation}\label{eq:extrinsic_product_3}
0=[\fetth_{x},R^N_{x,\i\,y}|_V]=R^N_{h(x,x),\i\,y}|_V+R^N_{x,h(x,\i\,y)}|_V\stackrel{\eqref{eq:extrinsic_product_1}}{=}R^N_{h(x,x),\i\,y}|_V+0\;.
\end{equation}
Thus, we have $R^N_{h(x,x),\i\,y}=0$ for all $(x,y)\in W_1\times W_2$ according to Lemma~\ref{le:cfl}\,. 
Therefore, 
\begin{equation}\label{eq:extrinsic_product_4}
0=R^N_{\i\,x,\i\, y}\stackrel{\eqref{eq:crg}}{=}y\wedge x
\end{equation} for all $(x,y)\in W_1\times W_2$ according to~\eqref{eq:extrinsic_product_2}.
Thus~\eqref{eq:extrinsic_product_4} implies that $x=0$ or $y=0$ by the condition $W_1\bot W_2$\,. 
Since $W_1\neq\{0\}$\,, this gives our claim.
 
But then also  $U_1=\{0\}$ by means of~\eqref{eq:extrinsic_product_2}, i.e. $W=W_1$ and $U=\i W_1$ which implies that $V$ 
is curvature invariant of Type $(c_i)$ defined by $(\Re,W_1)$\,.

\item  Case $i\geq 2$ and $W_1\bot e^{\i\varphi}\,U_2$\,. The remaining possibilities are $\varphi=0$ and $W_2=U_1$\,, or $W_2\bot e^{\i\varphi}\,U_1$\,, 
or $j=k=1$\,. In case $\varphi=0$ and $W_2\bot U_1$\,, we obtain that $V$ is curvature 
invariant of Type $(tr_{i+k,j+l})$ defined by $(\Re,W_1\oplus U_1,W_2\oplus U_2)$\,. 
Otherwise, we claim that $\rho(\frakk_V)|_V\cap\so(V)_-=\{0\}$\,:

let $a\in\R$\,, $B\in\so(\Re)$ be given such that $A:=a\, J^N+B$ satisfies $A(V)\subset V$ and $A|_V\in\so(V)_-$\,. 
Thus $A\,v\in U$ for every unit vector $v\in W_1$\,, i.e. $A\,v=u_1+\i\, u_2$
for suitable $u_1\in U_1$ and $u_2\in U_2$\,. Since
$$e^{\i \varphi} (A\, v)=e^{\i \varphi} (a\,\i\, v+ B\,v)=a\,\i\,\cos(\varphi)\,v 
-a\,\sin(\varphi)\,v+\cos(\varphi)\, B\,v+\i\,\sin(\varphi)\,B\,v\;,$$ 
we see that 
\begin{align}\label{eq:case_C_1}
&e^{\i \varphi}\,u_1=\Re(e^{\i \varphi} (A\,v))=-a\,\sin(\varphi)\,v+\cos(\varphi) \,B\,v\;,\\
\label{eq:case_C_2}
&e^{\i \varphi}\,u_2=\Im(e^{\i \varphi} (A\,v))=a\,\cos(\varphi)\,v +\sin(\varphi)\, B\,v\;.
\end{align}
The condition $W_1\bot e^{\i \varphi}\,U_2$ implies that
$$
0=\g{v}{e^{\i \varphi}\,u_2}\stackrel{\eqref{eq:case_C_2}}
=a\,\cos(\varphi)\g{v}{v}+\sin(\varphi)\g{B\, v}{v}=a\,\cos(\varphi),
$$
since $v$ is a unit vector and $B\in\so(\Re)$\,. Thus $a=0$\,, because $\varphi\in [-\pi/4,\pi/4]$\,. Therefore,
$A=B\in\so(\Re)$ anyway and, in particular, we have 
\begin{align*}&\cos(\varphi)\, A\,v\stackrel{\eqref{eq:case_C_1}}{=}e^{\i \varphi}\,u_1\;,\\
&\sin(\varphi)\, A\,v\stackrel{\eqref{eq:case_C_2}}{=}e^{\i \varphi}\,u_2\;.\end{align*}
We conclude that 
$$
0=\g{u_1}{u_2}=\g{e^{\i \varphi}\,u_1}{e^{\i \varphi}\,u_2}=\sin(\varphi)\,\cos(\varphi)\,\g{A\,v}{A\,v}\;.
$$
Hence $\varphi= 0$ or $A\,v=0$ for all $v\in W_1$\,. In the same way, 
we can show that $\varphi= 0$ or $A\,v=0$ for all $v\in W_2$\,.
By means of~\eqref{eq:iso1}, we conclude that $A|_V=0$ unless $\varphi= 0$\,. 

In case $\varphi=0$\,, set $V_1:=W_1\oplus U_1$ and $V_2:= W_2\oplus U_2$\,.
Note that $A|_{V_1}\in\so(V_1)_-$ and $A|_{V_2}\in\so(V_2)_-$\,.
Recall that we suppose that $W_1\bot U_2$ holds but that the condition $W_2\bot U_1$ fails.
Hence $W_2=U_1$ unless $j=k=1$\,. 
If $W_2=U_1$\,, then $A(U_1)=A(W_2)\subset W_1\cap U_2=\{0\}$ and hence $A|_{V_1}=A|_{V_2}=0$ 
according to~\eqref{eq:iso1}, as in the previous case. 
If $j=k=1$, then, using that $W_1\bot U_2$\,,
$$
\g{A\, v_1}{v_2}=-\g{v_1}{A\, v_2}=0
$$ for all $v_1\in W_1$ and $v_2\in W_2$\,. Further, by our assumptions, the linear form $\g{v_2}{\cdot }$ 
defines an isomorphism $U_1\to\R$ for every $v_2\in W_2$ which is not equal to zero. 
Thus  we conclude that $A|_{W_1}=0$ and hence $A|_{V_1}=0$\,, since 
$\so(V_1)_-\to \Hom(W_1,U_1), A\mapsto A|_{W_1}$ is a linear isomorphism according to~\eqref{eq:iso1}. 
For the same reason, $A|_{U_2}=0$ and hence $A|_{V_2}=0$\,. 
We conclude that $A|_V=0$\,. This establishes our claim.
 
Assume, by contradiction, that $(W,U)$ is integrable but at least one of the conditions $\varphi=0$ or $W_2\bot
U_1$ fails. We have just seen that this implies that $\rho(\frakk_V)|_V\cap\so(V)_-=\{0\}$\,. 
Thus, there exists a symmetric bilinear map $h:W\times W\to U$  
whose image spans $U$ and which satisfies~\eqref{eq:dec2}.
Note, the Lie algebra $\frakh$ defined in~\eqref{eq:La2} 
is given by $\so(W_1)\oplus \so(W_2)\oplus\so(U_1)\oplus\so(U_2)$ acting as a direct sum
representation on $W_1\oplus \i W_2\oplus U_1\oplus \i\,U_2$ 
where  $\so(W_1)$ acts non-trivially and irreducibly anyway (since $i\geq 2$). 
Therefore, by means of Schur's Lemma, $\Hom_\frakh(W_1,U)=\{0\}$\,, i.e. $\Hom_\frakh(W,U)\subset \Hom(\i W_2,U)$\,. 
If $j\neq 1$\,, then we even have $\Hom_\frakh(W,U)=\{0\}$\,, 
hence $h=0$ which is not possible.  Otherwise, if $j=1$\,, we thus see that
$h(x,y)=h(y,x)=0$ for all $x\in W_1$ and $y\in W$\,, i.e. $h(W\times W)=h(\i W_2\times \i W_2)$ 
which spans a 1-dimensional space, a contradiction (since
$k+\ell\geq 2$).\qed
\end{itemize}

\paragraph{Type $\mathbf{(tr_{k,\ell},tr_1)}$}
Suppose that $W$ is of Type $(tr_{k,\ell})$ defined by the data $(\Re,W_1,W_2)$ and $U$ is spanned by a unit vector $u$\,.  
\begin{itemize}
\item Case $k=\ell=1$\,. Similar as for Type $(tr_{1,1},tr_{1,1})$\,, if  $(W,U)$ is integrable, 
then $M$ is the Riemannian product of an extrinsic circle and a geodesic line in
a simply connected totally geodesic Riemannian product $\bar M=\C\rmP^1\times\R$ 
such that $V=T_p\bar M$ is of Type $(tr_{2,1})$ or $(ex_3)$\,. \qed

\item
Case $k\leq \ell$ with $\ell \geq 2$\,. Let us write $u=u_1+\i\,u_2$ with $u_1,u_2\in\Re$\,. 
Then we have $u_1\bot W_1\oplus W_2$ and $u_2\bot W_2$\,.
Further, if $u_2=0$ or if $u_1=0$ and $u_2\,\bot W_1$\,, then $W\oplus U$ is curvature invariant of 
Type $(tr_{2,\ell})$ or $(tr_{1,\ell+1})$ defined by the triples $(\Re,W_1\oplus U,W_2)$ 
or $(\Re,W_1,W_2\oplus \i\,U)$\,, respectively.
Otherwise, we claim that the linear space $\rho(\frakk_V)|_V\cap \so(V)_-$ is trivial:

let $a\in\R$ and $B\in\so(\Re)$ be given and suppose that $A:=a\,J^N\oplus B$ 
satisfies $A(V)\subset V$ and $A|_V\in\so(V)_-$\,. Then there exists a linear form $\lambda:W_2\to \R$ such that 
\begin{equation}\label{eq:case_trkell_tr1_1}
\forall y\in W_2:A\, \i\, y=-a\,y+ \i\,B\,y=\lambda(y)(u_1+\i\, u_2)\;.
\end{equation}
Comparing the real parts of the last equation, we obtain that $-a\,y=\lambda(y)\,u_1$ 
for all $y\in W_2$\,, hence $a=0$ (since $\ell\geq 2$)\,.
Thus there  exists a linear form $\mu:W_1\to \R$ such that 
\begin{equation}\label{eq:case_trkell_tr1_2}
\forall x\in W_1:B\,x=\mu(x)(u_1+\i\, u_2)\;.
\end{equation}
Comparing the imaginary parts of the last equation and recalling that $u_2\neq0$\,, we obtain that $B|_{W_1}=0$\,. 
Suppose now, by contradiction,  that there exists $y\in W_2$ with $B\,y\neq 0$\,. 
Then $\lambda(y)\neq 0$ and hence $u_1=0$ by means of~\eqref{eq:case_trkell_tr1_1}. Further, 
\begin{equation}\label{eq:case_trkell_tr1_3}
0=\g{y}{B\,x}=-\g{B\,y}{x}=-\lambda(y)\,\g{u_2}{x}
\end{equation}
for all $x\in W_1$\,. Since $\lambda(y)\neq 0$\,, we obtain that $u_2$ 
belongs to the orthogonal complement of $W_1$\,, i.e.\ 
we have shown that $u_1=0$ and $u_2\,\bot W_1$\,, which is a different case.
This proves our claim.

Assume that neither the case $u_2=0$ nor the case $u_1=0$ and $u_2\,\bot W_1$ holds but, 
by contradiction, that there exists an integrable symmetric bilinear map $h:W\times W\to U$ whose image spans $U$\,.
Note, the Lie algebra $\frakh$ from Corollary~\ref{co:dec} is given by $\so(W)$ 
acting irreducibly and non-trivially on $\i W_2$ (since $\ell\ge 2$) 
and trivially on $U$\,. Hence $\Hom_\frakh(W,U)\subset\Hom(W_1,U)$\,.
By means of~\eqref{eq:dec2}, we obtain that
\begin{align}\label{eq:case_trkell_tr1_4}
&h(W_1\times\i W_2)=h(\i W_2\times W_1)=h(\i W_2\times\i W_2)=0\;,\\
&\label{eq:case_trkell_tr1_5}
U=\Spann{h(x,x)}{x\in W_1}\;.
\end{align}
Further, we have $R^N_{x,\i\,y}=0$ for  every $(x,y)\in W_1\times W_2$\,. 
Therefore, on the one hand, Eq.~\ref{eq:cond2} (with $k=1$) yields 
$R^N_{h(x,x),\i\,y}|_V=0$ and thus $R^N_{h(x,x),\i\,y}=0$ for all $(x,y)\in W_1\times W_2$ according to Lemma~\ref{le:cfl}\,.
It follows that $0=R^N_{u,\i\,y}=-\g{u_1}{y}\, J^N-u_2\wedge y$ for all $y\in W_2$ which implies that $u_2=0$ 
(since $\ell\geq 2$)\,, a contradiction.\qed
\end{itemize}

\paragraph{Type $\mathbf{(c_k',c_\ell')}$}
Let $W$ and $U$ be of Types $(c_k')$ and
$(c_\ell')$ defined by the data $(\Re,W',I')$ and $(\Re^*,U',J')$ with $\Re=\Re^*$\,. If $U'=W'$ and $J'=-I'$\,, 
then $U=\bar W$ and $V=W\oplus\bar W=\C\,W'$ is curvature invariant
of Type~$(c_{2k})$\,. If $W'\bot U'$\,, then $V$ is curvature invariant of Type $(c_{k+\ell}')$
defined by $(\Re,W'\oplus U',I'\oplus J')$\,.\qed

\paragraph{Type $\mathbf{(c_1',tr_1)}$}
Let $W$ and $U$ be of Types $(c_1')$ and $(tr_1)$\,, respectively, with $U\subset \bar W$\,.
The action of $\frakh_W$ on $W$ is given by $\so(W)$ and hence $W$ is an irreducible $\frakh_W$-module 
(see Lemma~\ref{le:type(c_k')}~(a)). Therefore, if $(W,U)$ is integrable, 
then the linear space $W\oplus U$ is curvature invariant
according to Proposition~\ref{p:sph}.\qed

\paragraph{Type $\mathbf{(c_k,c_\ell')}$}
Let $W$ and $U$ be of Types $(c_k)$ and $(c_\ell')$ determined by the data $(\Re,W_0)$ and  $(\Re^*,U',I')$\, respectively. 
Suppose further that  $\Re=\Re^*$ and $W_0\bot U'$ holds.
Let $u\in U'$ be a unit vector and consider $A:=J^N+u\wedge I' u$\,. Recall that $A\in\frakh_U$ according to~\eqref{eq:type(c_k')_2}.
Further, $A\, (u-\i\,I'u)=2\, (I'u+\i\, u)$ according to~\eqref{eq:I'=pmi} and we have $A|_W=J^N|_W$\,. 
Thus $A^2:=A\circ A$ acts by the scalars $-4$ and $-1$ on the linear spaces $\C(u-\i\,I'u)$ and $W$\,, respectively.
Let $\lambda\in \Hom_\frakh(U,W)$ be given, then
\begin{align*}
-4\,\lambda(u-\i\,I'u)=\lambda(A^2 (u-\i\,I'u))=A^2\,\lambda(u-\i\,I' u)=-\lambda(u-\i\,I'u)\;,\\
-4\,\lambda(I'u+\i\,u)=\lambda(A^2 (I'u+\i\,u))=A^2\,\lambda(I'u+\i\,u)=-\lambda(I'u+\i\,u)\;.
\end{align*}
It follows that $\lambda|_{\C(u-\i\,I'u)}=0$\,. Since $u$ is arbitrary, we conclude that $\Hom_\frakh(U,W)=\{0\}$ and thus $\Hom_\frakh(W,U)=\{0\}$\,, 
too, because of~\eqref{eq:iso3}.
Furthermore, we claim  that $\rho(\frakk_V)|_V\cap\so(V)_-=\{0\}$\,:

let $a\in\R$ and $B\in\so(\Re)$ be given, set $A:=a\,J^N\oplus B$ and suppose that $A(V)\subset V$ and $A|_V\in\so(V)_-$\,. 
If $v$ is a unit vector of $W_0$\,, then $v,\i\,v\in W$ and thus 
$$
0=\g{A\,v}{\i\,v}=a\,\g{\i\,v}{\i\,v}\;,
$$
i.e. $a=0$\,. Hence $A\in\so(\Re)$ and $A\,v$ belongs to $U\cap\Re=\{0\}$\,, i.e. $A\,\i\,v=\i\,A\,v=0$  for all $v\in W_0$\,. 
Therefore, $A|_V=0$ because of~\eqref{eq:iso1}.

Whence, Corollary~\ref{co:dec} implies that neither $(W,U)$ nor $(U,W)$ is integrable.\qed

\paragraph{Type $\mathbf{(tr_{j,k},tr'_\ell)}$}
Let $(W,U)$ be an integrable orthogonal curvature
invariant pair with $W$ and $U$ of Types $(tr_{j,k})$ and $(tr_\ell')$ determined by the data $(\Re,W_1,W_2)$ and $(\Re^*,U',I',U_0')$\,, respectively.
By means of Corollary~\ref{co:type(tr_i',tr_kl)}, we can assume that $\Re=\Re^*$ and that $W_1\oplus W_2$ 
is contained in the orthogonal complement of  $U'$ in $\Re$\,.
We claim that the linear space  $\rho(\frakk_V)|_V\cap\so(V)_-$ is trivial:

let $a\in\R$ and $B\in\so(\Re)$ be given, set $A:=a\,J^N\oplus B$ and suppose that $A(V)\subset V$ and $A|_V\in\so(V)_-$ holds. 
If $v_1\in W_1$\,, then $a\,\i\,v_1$ is the imaginary part of $A\,v_1$\,. Since $A\,v_1\in U$\,, we see that
$A\,v_1=a(I' v_1+\i\,v_1)$\,. In particular, $a\,I' v_1\in U_0'\subset U'$\,. Because $W_1\cap U'=\{0\}$\,, this implies $a=0$\,, i.e. $A$ vanishes on $W_1$\,.
In the same way, we can show that $A$ vanishes on $\i W_2$\,, too. Hence, we see
that $A|_V=0$\,, since~\eqref{eq:iso1} is a linear isomorphism.
This establishes our claim.

Further, according to Lemma~\ref{le:type(tr_k')}~(a), the action of $\frakh_U$ on $U$ is given by $\so(U)$ 
and $\frakh_U$ acts trivially on $(\C\, U')^\bot$\,.
Thus, $\Hom_\frakh(W,U)=\{0\}$\,. Therefore, Corollary~\ref{co:dec} implies that neither $(W,U)$ nor $(U,W)$ is integrable.\qed

\paragraph{Type $\mathbf{(tr_k',tr_\ell')}$}
Let $W$ and $U$ be of Types $(tr_k')$ and
$(tr_\ell')$ defined by the data $(\Re,W',I',W_0')$ and $(\Re^*,U',J',U_0')$\,,
respectively. We can assume that $\Re=\Re^*$\,. 

If $W'$ is orthogonal to $U'$\,, then $W\oplus U$ is curvature invariant of Type $(\tr_{k+\ell}')$
defined by $(\Re,W'\oplus U',I'\oplus J',W_0'\oplus U_0')$\,. If $U=\i W$\,, then $W\oplus U$ 
is curvature invariant of Type $(c_k')$ defined by $(\Re,W',I')$\,.

Suppose that $k=\ell\geq 3$ and $U=e^{-\i\theta}\,\bar W$ for some $\theta\in\R$\,. 
We claim that neither $(W,U)$ nor $(U,W)$ is integrable. Since $W=e^{-\i\theta}\,\bar U$\,, it suffices to prove the first assertion.
In order to explain the idea of our proof, first consider the case 
$\theta=0$\,. Then, the linear space  $V$ is curvature invariant of Type $(c_k)$ defined by $(W',I')$ and
the totally geodesic submanifold $\exp^N(V)$ is isometric to a product $\rmS^k\times\rmS^k$ such that $p=(o,o)$ (where $o$ is some origin of $\rmS^k$) 
such that the linear space  $W$ is given by $\Menge{(x,x)}{x\in T_o\rmS^k}$\,. 
If we assume, by contradiction, that $(W,U)$ is integrable, then the corresponding 
complete parallel submanifold through $p$ would be contained in $\rmS^k\times\rmS^k$ by means of reduction of the codimension and, moreover, it would 
even be a symmetric submanifold of $\rmS^k\times\rmS^k$ according to Corollary~\ref{co:1-full}. 
However, this is not possible, since a symmetric submanifold $M\subset \rmS^k\times\rmS^k$ 
through the point $p=(o,o)$ with $T_pM=\Menge{(x,x)}{x\in T_o\rmS^k}$ 
is totally geodesic according to Theorem~\ref{th:products}.
 In the general case, the linear space $V$ is not curvature invariant, but a similar idea shows that $(W,U)$ is not integrable, as follows.

\bigskip
\begin{definition}\label{de:RH}
Let $A\in\so(W')$ be given. We say that $A$ is real, holomorphic or anti-holomorphic if $A(W_0')\subset W_0'$\,, $A\circ I'=I'\circ A$ 
or $A\circ I'=-I'\circ A$\,, respectively.
 \end{definition}

Consider the linear map $J_\theta$ on $W'\oplus \i W'$ which is given on
$W_0'\oplus \i\,I'(W_0')$ by $J_\theta(v-\i\,I' v):=e^{-\i\theta} (v+\i\,I'
v)$ and $J_\theta(v+\i\,I' v):=-e^{\i\theta} (v-\i\,I' v)$ 
for all $v\in W_0'$ and which is extended to $W'\oplus \i W'$ by $\C$-linearity (note, $W_0'\oplus \i\,I'(W_0')$ is a real form of $W'\oplus \i W'$).

\bigskip
\begin{lemma}\label{le:Iphi}
Let $W$ be of Type $(tr_k')$ defined by the data $(\Re,W',I',W_0')$\,. Set $U:=e^{-\i\theta}\,\bar W$ and $V:=W\oplus U$\,.
\begin{enumerate}
\item
$J_\theta$ is a Hermitian structure on $W'\oplus \i W'$ 
such that $W$ gets mapped onto $U$ and vice versa. In particular, 
$V$ is a complex subspace of $(W'\oplus \i W',J_\theta)$ and  $J_\theta|_V$ belongs to $\so(V)_-$\,.
\item
Let $A\in\so(W')$ and suppose that $A$ is real. As usual, we extend both $A$ and $I'$ to complex linear maps on $W'\oplus \i W'$ via complexification. 
If $A$ is holomorphic, then $A$ commutes with $J_\theta$ for all $\theta\in\R$\,.
If $A$ is anti-holomorphic, then $\exp(\theta\,I')\circ A$ anti-commutes with $J_\theta$ for all $\theta\in\R$\,.
\end{enumerate}
\end{lemma}
\begin{proof}
Let $\{e_1,\ldots,e_k\}$ be an orthonormal basis of $W_0'$ and set $v_i:=1/\sqrt{2} (e_i-\i\,I' e_i)$\,. 
Then $\{v_1,\ldots, v_k,\bar v_1,\ldots\bar v_k\}$ is a Hermitian basis of $W'\oplus \i W'$\,. 
We define a unitary map $J$ on $W'\oplus \i W'$ via $J(v_i):=\bar v_i$ and $J(\bar v_i):=-v_i$\,. 
Further, set $I:=I'$ and  $K:=I\circ J$\,. Then $I^2=J^2=-\Id$ and $I\circ J=-J\circ I$ by means of~\eqref{eq:I'=pmi}\,, i.e.\ the usual quaternionic relations hold. Furthermore, $K\, v=-\i\,\bar v$ for all $v\in W'\oplus\i W'$\,.
Note that \begin{equation}\label{eq_Jphi}
J_\theta=\exp(\theta\, I)\circ  J=J\circ \exp(-\theta\, I)=\exp(\theta/2\, I)\circ  J\circ \exp(-\theta/2\, I)\;.
\end{equation}
It follows that $J_\theta$ defines another complex structure on $W'\oplus \i W'$\,. 
Since $W=\{v_1,\ldots,v_k\}_\R$ and $\bar W=\{\bar v_1,\ldots,\bar v_k\}_\R$\,, we see from~\eqref{eq:I'=pmi} that $J_\theta(W)= e^{-\i\theta}\,\bar W$
and $J_\theta(\bar W)=\e^{\i\theta}\,W$\,, i.e. $J_\theta(V)=V$ and $J_\theta\in\so(V)_-$\,. This proves the first part of the lemma. 
Moreover, if $A\in\so(W')$ is holomorphic or anti-holomorphic, then $A$ commutes 
or anti-commutes with $I$ on $W'\oplus \i W'$\,, respectively. 
If $A$ is additionally real, then the same is true for $J$ instead of $I$\,: 

in fact, since $A$ is real, we have $A(\bar v)=\overline{A\,v}$ for all $v\in W'$\,,
hence $A\circ K=K\circ A$ on $W'\oplus \i W'$ and thus 
\begin{equation}\label{eq:A_circ_J=pm_J_circ_A}
A\circ  J =A\circ K\circ I=K\circ A\circ I=\pm\, K\circ I\circ A=\pm\, J\circ A\;,
\end{equation} 
where the sign $\pm$ is chosen according
to whether $A$ is holomorphic ($+$) or anti-holomorphic ($-$)\,. Our claim follows.

Therefore, if $A$ is real and holomorphic, then $A$ commutes with both $J$ and $I$\,, hence $A$ commutes also with $J_\theta$ according to~\eqref{eq_Jphi}.
Suppose that $A\in\so(W')$ is real and anti-holomorphic. Using~\eqref{eq_Jphi} and~\eqref{eq:A_circ_J=pm_J_circ_A} (with the negative sign),
we have \begin{align*}
&J_\theta\circ \exp(\theta\, I)\circ
 A\stackrel{\eqref{eq_Jphi}}{=}J\circ \exp(-\theta\, I)\circ\exp(\theta\, I)\circ A=J\circ A\\
&=-A\circ J=-A\circ\exp(-\theta\, I)\circ J_\theta=-\exp(\theta\, I)\circ A\circ J_\theta\;.
\end{align*}
Since $I=I'$ on $W'\oplus \i W'$\,, we see that $\exp(\theta\, I')\circ A$  anti-commutes with $J_\theta$\,.
\end{proof}

\bigskip
\begin{lemma}\label{le:kv}
Let $W$ be of Type $(tr_k')$ defined by the data $(\Re,W',I',W_0')$\,. Set $U:=e^{-\i\theta}\,\bar W$ 
for some $\theta\in\R$ and $V:=W\oplus U$\,.
\begin{enumerate}
\item The linear map \begin{equation}\label{eq:map}
F:W_0'\oplus W_0'\to V, (v,u)\mapsto 1/2[v-\i\,I' v+J_\theta(v-\i\,I' v)+ u-\i\,I' u-J_\theta(u-\i\,I' u)]
\end{equation}
is an isometry such that the linear spaces 
$\Menge{(v,v)}{v\in W_0'}$ and $\Menge{(v,-v)}{v\in W_0'}$ get identified with  $W$ and $U$\,, respectively.
\item By means of~\eqref{eq:map}, the direct sum Lie algebra $\so(W_0')\oplus\so(W_0')$  gets identified with 
the Lie algebra $\rho(\frakk_V)|_V$ such that $(A,A)\in\rho(\frakk_V)|_V\cap\so(V)_+$ and $(A,-A)\in\rho(\frakk_V)|_V\cap\so(V)_-$ for every $A\in\so(W_0')$\,.

\item
The complex structure $J_\theta|_V$ commutes with every element 
of $\rho(\frakk_V)|_V\cap\so(V)_+$ whereas it anti-commutes with every element of $\rho(\frakk_V)|_V\cap\so(V)_-$\,.
\end{enumerate} 
\end{lemma}
\begin{proof}
For~(a): we have $F(v,v)=v-\i\,I'v\in W$ and $F(v,-v)=J_\theta(v-\i\,I'v)=e^{-\i\theta} (v+\i\,I' v)\in \e^{-\i\theta}\,\bar W=U$\,. 
Since $\dim(W)=\dim(U)=\dim(W_0')$\,, we conclude that $F$ is actually a linear isometry onto $V$ with the properties described above.

For~(b): given $A\in\so(W_0')$\,, we associate therewith linear maps $\hat A$ and $\tilde A$ on $W'$ 
defined by $\hat A(v+I'v):=A\,v+I'A\,v$ and $\tilde A(v+I' v):=A\, v-I' A\,v$ for every $v\in W_0'$\,. By definition,
both $\hat A$ and $\tilde A$ are real, further, $\hat A$ is holomorphic
whereas $\tilde A$ is anti-holomorphic. Furthermore, we consider the second splitting $V=V_1\oplus V_2$ 
with $V_1=\Menge{v-\i\,I' v+J_\theta(v-\i\,I' v)}{v\in W_0'}$ and $V_2=\Menge{v-\i\,I' v-J_\theta(v-\i\,I' v)}{v\in
W_0'}$\,. Note that both $V_1$ and $V_2$ are naturally isomorphic to $W_0'$\,. Hence, 
this splitting induces a monomorphism of Lie algebras $\so(W_0')\oplus\so(W_0')\hookrightarrow \so(V)$\,.
We claim that this monomorphism is explicitly given by 
\begin{equation}\label{eq:(A,B)}
(A,B)\mapsto 1/2[(\widehat{A+B})+\exp(\theta\, I')\circ(\widetilde{A-B})]
\end{equation} 

for each $A\in\so(W_0')$ we have
\begin{align*}
&1/2 (\hat A+\exp(\theta\, I')\circ\tilde A)(v-\i\,I'v+ J_\theta(v-\i\,I'v))
=1/2 [(A\,v-\i\,I'A\,v)+e^{-\i\theta}(A\,v+\i\,I'A\,v)\\
&+e^{-\i\theta}(A\,v+\i\,I'A\,v)+e^{\i\theta}e^{-\i\theta}(A\,v-\i\,I'A\,v)]=(A\,v-\i\,I'A\,v)+e^{-\i\theta}(A\,v+\i\,I'A\,v)\;,\\
&1/2 (\hat A+\exp(\theta\, I')\circ\tilde A)(v-\i\,I'v- J_\theta (v-\i\,I' v)=1/2[(A\,v-\i\,I'
A\,v)-e^{-\i\theta}(A\,v+\i\,I'A\,v)+e^{-\i\theta}(A\,v+\i\,I'A\,v)\\
&-e^{\i\theta}e^{-\i\theta}(A\,v-\i\,I'A\,v)]=0
\end{align*}
for all $v\in W_0'$\,. This establishes our claim in case $B=0$\,. For $A=0$\,, a similar calculation works.

Further, we claim that in this way $\so(W_0')\oplus\so(W_0')\cong
\rho(\frakk_V)|_V$ such that $(A,A)\cong\hat A\in\rho(\frakk_V)|_V\cap\so(V)_+$ 
and $(A,-A)\cong\exp(\theta\, I')\circ\tilde A\in\rho(\frakk_V)|_V\cap\so(V)_+$\,.

For ``$\subseteq$'': we have $\hat A(v\pm\, \i\,I'v)=A\,v\pm\, \i\,I'A\,v$ for all $v\in W_0'$ and $A\in\so(W_0')$\,, thus $\hat A$ maps $W$ to $W$
and $U$ to $U$\,. Further, $\hat A\in\so(\Re)$\,. 
This shows that $\hat A\in\rho(\frakk_V)|_V\cap\so(V)_+$\,. Furthermore, $\tilde A (v\pm\,
\i\,I'v)=A\,v\mp\i\, I' A\,v$ and $I'(v\pm\,\i\,I'v)=\mp\i(v\pm\,\i\, I'v)$ for all $v\in
W_0'$\,, thus $\exp(\theta\, I')\circ \tilde A$ maps $W$ to $U$ and vice versa. 
Finally, note that $\exp(\theta\, I')\circ \tilde A$ is in fact skew-symmetric since 
\[
(\exp(\theta\, I')\circ \tilde A)^*=\tilde A^*\circ\exp(\theta\, I')^*=-\tilde A\circ \exp(-\theta\, I')=-\exp(\theta\, I')\circ \tilde A\;.
\]
Hence $\exp(\theta\, I')\circ\tilde A \in\rho(\frakk_V)|_V\cap\so(V)_-$\,.

For ``$\supseteq$'': conversely, let some $A\in\rho(\frakk_V)|_V$ be given. 
We will distinguish the cases $A\in\so(V)_+$ and $A\in\so(V)_-$\,.
Anyway, we have $A=a\,J^N|_V+B|_V$ with $B\in\so(W')$ and $a\in\R$\,. Set $B':=a\,I'|_V+B|_V $\,. 
Note that $I'=J^N$ on $W+\i W$ and $I'=-J^N$ on $\bar W+\i\bar W$\,, hence 
\begin{align}\label{eq:B1}
&A=B'\ \text{on}\ W+\i W\;,\\
\label{eq:B2}
&A=B'+2\,a\,J^N\ \text{on both $\bar W$ and $U$}\;.
\end{align}  
If $A\in\so(V)_+$\,, then $A(W)\subset W$ and hence we conclude from~\eqref{eq:B1} that $B'$ is real and holomorphic; 
thus $B'=C$ with $C:=B'|_{W_0'}$\,.
In particular, $B'(\bar W)\subset \bar W$ and hence $B'(U)\subset U$\,.
Thus $a=0$ because of~\eqref{eq:B2} and since $J^N(U)\subset U^\bot$ where $U^\bot$ denotes the orthogonal complement of $U$ in $T_pN$\,. 

If  $A\in\so(V)_-$\,, then $A$ maps $W$ to $U$ and vice versa, hence $B'(W)\subset U$ 
because of~\eqref{eq:B1}, thus $e^{\i\theta}B'(W)\subset \bar W$ which shows 
that the linear endomorphism $C:=\exp(-\theta\, I')\circ B'$ is real and anti-holomorphic on $W'$\,.
Hence $B'$ is anti-holomorphic, too. Therefore, also $\exp(-\theta\, I')(B'(\bar W))\subset W$\,, thus $B'(U)\subset W$\,. 
We conclude that $a=0$ according to~\eqref{eq:B2} (since 
$J^N(U)\subset \bar W+\i\bar W\subset W^\bot$). This establishes our claim. Part~(b) follows.

For~(c), recall that $\hat A$ commutes with $J_\theta|_V$ whereas $\tilde A$
anti-commutes with $J_\theta|_V$ according to Lemma~\ref{le:Iphi} for every $A\in\so(W_0')$\,.
Hence, the result is a consequence of Part~(b) and~\eqref{eq:(A,B)}.
\end{proof}

\bigskip
\begin{corollary}\label{co:kgeq3}
Suppose that $W$ is of Type $(tr_k')$ with $k\geq 3$\,. The curvature invariant pair
$(W,\e^{-\i\theta}\,\bar W)$ is not integrable.
\end{corollary}

\begin{proof}
Let $W$ be defined by the data $(\Re,W',I',W_0')$\,, set  $U:=e^{-\i\theta}\,\bar W$  and $V:=W\oplus U$\,. 
Assume, by contradiction, that $(W,U)$ is integrable and let $\frakg$ be the subalgebra 
of $\so(V)$ described in Theorem~\ref{th:dec}. Recall that there exist $A_x\in\rho(\frakk_V)|_V\cap\so(V)_-$ and 
$B_x\in Z(\frakg)\cap\so(V)_-$  such that $\fetth_x=A_x + B_x$ for every $x\in W$\,. 
We claim that here $B_x=0$\,:

for this, recall that $\frakh_W|_V$ is a subalgebra of
$\rho(\frakk_V)|_V\cap\so(V)_+$ (see~\eqref{eq:cip1} and~\eqref{eq:symmetric_space}). 
Therefore $[J_\theta|_V,R^N_{x,y}|_V]=0$ for all $x,y\in W$ and
$A_x\circ J_\theta|_V=-J_\theta|_V\circ A_x$ according to Lemma~\ref{le:kv}~(c). 
It follows, on the one hand, that $J_\theta|_V$ anti-commutes with $[R^N_{x,y}|_V,A_z]$ for all $x,y,z\in W$\,. 
Assume, by contradiction, that $B_x\neq 0$ for some $x\in W$\,. We claim that
this implies, on the other hand, that $J_\theta|_V$ commutes with $[R^N_{x,y}|_V,A_z]$ (which is
not possible unless $[R^N_{x,y}|_V,A_z]=0$\,, since $J_\theta|_V$ is a complex structure):

consider the Lie algebra $\frakh$ defined in~\eqref{eq:La2}. By means of
Lemma~\ref{le:type(tr_k')}~(a), we have
\begin{equation}\label{eq:equality}
\frakh=\frakh_W|_V=\Menge{A\in \fraku(W',I')}{A(W_0')\subset W_0'}\;.
\end{equation} 
Hence $\frakh$ acts on $W$ and $U$ via $\so(W)$ and $\so(U)$\,, respectively. 
Thus $W$ and $U$ both are irreducible $\frakh$-modules. Further, recall that $J_\theta|_V\in \so(V)_-$ according to Lemma~\ref{le:Iphi}.
Hence, by the above, $J_\theta|_V\in Z(\frakh)\cap\so(V)_-$\,.
Moreover, since $k\geq 3$\,, Schur's Lemma shows that $\Hom_{\frakh}(W,U)$ is at most a 
1-dimensional space. Therefore, because of~\eqref{eq:iso2}, 
the linear space $Z(\frakh)\cap\so(V)_-$ is spanned by $J_\theta$\,. 
Further, since $\frakh\subset\frakg$ (cf.\ the proof of Corollary~\ref{co:dec}), 
we have $B_x\in Z(\frakg)\cap\so(V)_-\subset Z(\frakh)\cap\so(V)_-$\,.
Hence there
exists $0\neq b\in\R$ with $J_\theta|_V=b\,B_x\in Z(\frakg)$\,. 
Thus, since $[R^N_{x,y}|_V,A_z]\in\frakg$\,, we
see that $J_\theta|_V$ commutes with $[R^N_{x,y}|_V,A_z]$\,. This gives our claim. 

Therefore, we conclude that
$[R^N_{x,y}|_V,\fetth_z]=[R^N_{x,y}|_V,\fetth_z-B_z]=[R^N_{x,y}|_V,A_z]=0$ for
all $x,y,z\in W$\,. As remarked above, this shows that $\fetth_z\in\R\, J_\theta$
for all $z\in W$\,. But this would imply that $\fetth=0$ since $\fetth$ is injective 
or zero according to~\eqref{eq:Kernh},~\eqref{eq:Kernfetth} and  Proposition~\ref{p:invariant}, a contradiction.

Thus $B_x=0$\,, i.e. $\fetth_x=A_x\in \rho(\frakk_V)|_V$ for all $x\in W$\,. Let
us choose some  $o\in\rmS^k$\,, a linear isometry $f:T_o\rmS^k\to W_0'$ and consider the
Riemannian  product $\tilde N:=\rmS^k\times\rmS^k$ whose curvature tensor will
be denoted by $\tilde R$\,. On the analogy of~\eqref{eq:map},
$$
F:T_{(o,o)}\tilde N\to T_pN, (x,y)\mapsto 1/2[f(x)+f(y)-\i\,I'(f(x)+f(y))+J_\theta (f(x)-f(y)-\i\,I'(f(x)-f(y)))]
$$
is an isometry onto $V$ such that $\Menge{F^{-1}\circ A|_V\circ F}{A\in\rho(\frakk_V)}$ 
is the direct sum Lie algebra $\so(T_o\rmS^k)\oplus \so(T_o\rmS^k)$\,. 
Note, the latter is the image $\tilde\rho(\tilde\frakk)$ of the linearized isotropy representation  of $\tilde N$\,. 
Put $\tilde W:=F^{-1}(W)$\,, $\tilde h:=F^{-1}\circ h\circ F\times F$ and
$\tilde U:=\Spann{\tilde h(u,v)}{u,v\in\tilde W}$\,. Then $\tilde U=F^{-1}(U)$
and hence $T_{(o,o)}\tilde N=\tilde W\oplus \tilde U$ holds. 
Furthermore, we claim that $(\tilde W,\tilde h)$ is an integrable 2-jet in $T_{(o,o)}\tilde N$\,:

Let $\iota:\rmS^k\to\rmS^k\times\rmS^k, p\mapsto (p,p)$\,. Then
$T_{(o,o)}\iota(\rmS^k)=\Menge{(x,x)}{x\in T_o\rmS^k}$\,, hence $F(T_{(o,o)}\iota(\rmS^k))=W$\,, 
i.e. $T_{(o,o)}\iota(\rmS^k)=\tilde W$\,. Further, on the one hand, we have 
\begin{align*}
R^N(F(x,x),F(y,y),F(z,z))&=R^N(f(x)-\i\,I'f(x),f(y)-\i\,I'f(y),f(z)-\i\,I'f(z))\\
&=-f(x)\wedge f(y)\,f(z)+\i(I'f(x)\wedge I'f(y))I'f(z)
\end{align*}
for all $x,y,z\in T_o\rmS^k$ according to Lemma~\ref{le:type(tr_k')}~(a). 
On the other hand, 
\begin{align*}
F(x\wedge y\,z, x\wedge y\,z)&= f(x)\wedge f(y)\,f(z) - \i\,I'f(x)\wedge f(y)\,f(z)\\
&= f(x)\wedge f(y)\,f(z) - \i\,I'f(x)\wedge I'f(y)\,I'f(z)\;.\end{align*}
This shows that $F\circ \tilde R^N_{(x,x),(y,y)}|_{\tilde  W}=R^N_{F(x,x),F(y,y)}\circ F|_{\tilde  W}$\,. 
Furthermore,~\eqref{eq:cip1},~\eqref{eq:symmetric_space} and Lemma~\ref{le:kv}~(b) show 
that $\tilde R^N_{(x,x),(y,y)}$ and $F^{-1}\circ R^N_{F(x,x),F(y,y)}\circ F$ both belong to $\tilde\rho(\tilde\frakk)_+$\,,
i.e.\ there exist $A,B\in\so(T_o\rmS^k)$
with  $\tilde R^N_{(x,x),(y,y)}=A\oplus A$ and $F^{-1}\circ R^N_{F(x,x),F(y,y)}\circ F=B\oplus B$\,.
Thus, since the direct sum endomorphism $A\oplus A$ is uniquely determined by its restriction to $\tilde W$ for every $A\in\so(T_o\rmS^k)$\,,
we conclude that  $F\circ \tilde R^N_{(x,x),(y,y)}=R^N_{F(x,x),F(y,y)}\circ F$\,. Therefore,
$\tilde W $ is curvature invariant and $\tilde h$ is semi-parallel in $\tilde N$\,. Moreover,
since  $\fetth_x\in \rho(\frakk_V)|_V$ for all $x\in W$\,, we have $\tilde
\fetth_x\in \tilde\rho(\tilde\frakk)$ for all $x\in \tilde
W$ which shows that Eq.~\ref{eq:cond2} for $(\tilde W,\tilde h)$ is 
implicitly given for all $k$\,. Hence, by means of Theorem~\ref{th:integrable}, we obtain 
that $(\tilde W,\tilde h)$ is an integrable 2-jet in $\tilde N$\,. 

Thus, there exists a complete parallel submanifold $\tilde M\subset \tilde N$
through $(o,o)$ whose 2-jet is given by $(\tilde
W,\tilde h)$\,. The fact that $T_{(o,o)}\tilde N=\tilde W\oplus \tilde U$ 
holds implies that $\tilde M$ is 1-full in $\tilde N$\,, i.e.\ extrinsically symmetric according to
Corollary~\ref{co:1-full}. Further, since $\tilde M$ is tangent to $\iota(\rmS^k)$ at $(o,o)$\,, there do not 
exist submanifolds $\tilde M_1\subset \rmS^k$ and $\tilde M_2\subset
\rmS^k$ such that $\tilde M=\tilde M_1\times\tilde M_2$\,. 
Therefore, by means of Theorem~\ref{th:products}, $\tilde M$ is
totally geodesic, i.e. $h=0$\,, a contradiction.
\end{proof}

Suppose that  $W$ is of Type $(tr_2')$ defined by the data $(\Re,W',I', W_0')$\,. 
Let $\{e_1,e_2\}$ be an orthonormal basis of $W_0'$ and $\tilde I$ be defined according to~\eqref{eq:tildeI}.
Further, let $\tilde J\in\SU(W',\tilde I)\cap\so(W')$ be given and set $U:=\tilde J(W)$\,. 
We will show that neither $(W,U)$ nor $(U,W)$ is integrable unless $V:=W\oplus U$ is curvature invariant. 
First, we claim that it suffices to prove the first assertion:

recall that here $U$ is also of Type $(tr_2')$\,, defined by the triple $(\Re, U',U_0',J')$ 
with $U':=W'$\,, $J':=\tilde J\circ I'\circ\tilde J^{-1}$ and $U_0':=\tilde J(W_0')$\,. Further,
\[
\tilde I=\tilde J\circ\tilde I\circ\tilde J^{-1}\stackrel{\eqref{eq:tildeI}}{=}\tilde J\, e_1\wedge \tilde J\, e_2
+\tilde J\, I'e_1\wedge \tilde J\,I'e_2=
\tilde J\, e_1\wedge \tilde J\, e_2+J'\tilde J\,e_1\wedge J'\tilde J\,e_2\;.
\]
Hence, since $\{\tilde J\, e_1,\tilde J\, e_2\}$ is an orthonormal basis of $U_0'$\,, 
the Hermitian structure $\tilde I$ may also be defined on the analogy of~\eqref{eq:tildeI}
via the triple $(U_0',\{\tilde J\, e_1,\tilde J\, e_2\},J')$\,. 
Further, we have $U=\tilde J(W)$\,, hence also $W=\tilde J(U)$ (since $\tilde J^2=-\Id$)\,. This proves the claim. 

\bigskip
\begin{lemma}\label{le:kv2}
Suppose that  $k=2$ and $W$ is of Type $(tr_k')$ defined by the data $(\Re,W',I', W_0')$\,. 
Let $\{e_1,e_2\}$ be an orthonormal basis of $W_0'$ and let $\tilde I$ be defined according to~\eqref{eq:tildeI}.
Further, let $\tilde J\in\SU(W',\tilde I)\cap\so(W')$ be given, set $U:=\tilde J(W)$ and $V:=W\oplus U$\,.
Then we have $\rho(\frakk_V)|_V\cap\so(V)_-=\R \tilde J|_V$ unless $\tilde J=\pm\, I'$\,.
\end{lemma}
\begin{proof}
First, we claim that $\tilde J\in\rho(\frakk_V)$ and $\tilde J|_V\in\rho(\frakk_V)|_V\cap\so(V)_-$\,: 

we have $\tilde J\in\so(W')\subset\so(\Re)\subset\rho(\frakk)$\,. 
Further, $\tilde J(W)=U$ and $U=\tilde J(W)$\,, hence $\tilde J(V)=V$ and $\tilde J|_V\in\so(V)_-$\,. This gives our claim.

Conversely, let $A\in\rho(\frakk_V)$ be given. We aim to show that $A$ 
is a multiple of $\tilde J|_V$ unless $\tilde J=\pm\, I'$\,.
Let $a\in\R$ and $B\in\so(\Re)$ such that $A:=a\,J^N+B$ satisfies $A(V)\subset V$ and $A|_V\in\so(V)_-$\,.
With $\tilde A:=a\,I'+B|_{W'}$\,, we have $\tilde A\in\so(\Re)$ and $A|_W=\tilde A|_W$ according to~\eqref{eq:I'=pmi}\,, 
hence 
\begin{equation}\label{eq:tildeA_1}
\tilde A (v-\i\, I'v)\in \Menge{\tilde J\,x-\i\tilde J\,x}{x\in W_0'}
\end{equation} 
for all $v\in W_0$\,. Thus, using~\eqref{eq:tildeA_1} and passing to real and imaginary parts, 
we conclude that $\tilde A|_{W'}\in\so(W')$ such that $\tilde A (W_0')\subset \tilde J(W_0')$\,.  
In particular, the endomorphism $C:=\tilde J\circ \tilde A$ on $W'$ is real and holomorphic
(see Definition~\ref{de:RH}). Further,  $\tilde A^*=-\tilde A$ which implies that 
\begin{equation}\label{eq:solution}
\tilde J\circ C=C^*\circ \tilde J\;.
\end{equation}
We claim that $C=c\,\Id$ for some $c\in\R$ or $\tilde J=\pm\, I'$\,:

for this, let $\mathrm{RH}$ denote the algebra of real and holomorphic maps on $W'$\,. 
Note, $\tilde I$ is real and holomorphic, hence there is the splitting $\mathrm{RH}=\mathrm{RH}_+\oplus \mathrm{RH}_-$ 
with 
\begin{align*}
&\mathrm{RH}_+:=\Menge{A\in \mathrm{RH}}{A\circ\tilde I=\tilde I\circ A}\;,\\
&\mathrm{RH}_-:=\Menge{A\in \mathrm{RH}}{A\circ\tilde I=-\tilde I\circ A}\;.
\end{align*} 
Then $\mathrm{RH}_+=\{\Id,\tilde I\}_\R$ and $\mathrm{RH}_-=\{\sigma,\sigma\circ \tilde
I\}_\R$\,, where $\sigma$ denotes the conjugation of $(W',\tilde I)$ with respect to the real form $\{e_1,I'e_1\}_\R$\,. 
Further, consider the involution on  $\End(W')$ defined by $C\mapsto -\tilde J\circ C^*\circ \tilde J$\,. This
map preserves both $\mathrm{RH}_+$ and $\mathrm{RH}_-$ and its
fixed points in $\mathrm{RH}$ are the solutions to~\eqref{eq:solution}. 
It follows that a solution to~\eqref{eq:solution} with $C\in \mathrm{RH}$ decomposes as
$C=C_++C_-$ such that $C_\pm\in \mathrm{RH}_\pm$ and $C_\pm$ is a solution to~\eqref{eq:solution}, too.

Then we have $\tilde J\circ \tilde I=\tilde I\circ \tilde J=-\tilde
I^*\circ\tilde J$ since $\tilde I$ is skew-symmetric and commutes with $\tilde J$\,.
Hence  a solution to~\eqref{eq:solution} with $C\in \mathrm{RH}_+$ is given only if $C$ is a multiple of $\Id$\,.
If $C\in \mathrm{RH}_-$ is a solution to~\eqref{eq:solution}, then $C\circ \tilde I$ is a
solution to this equation, too, since $$\tilde J\circ C\circ \tilde
I=C^*\circ\tilde J\circ \tilde I=C^*\circ\tilde I\circ \tilde J=(I^*\circ C)^*\circ \tilde J=(-\tilde I\circ C)\circ \tilde J=C\circ \tilde I\circ \tilde J\;.$$
Thus, since  $\mathrm{RH}_-$ is invariant under multiplication from the right by $\tilde I$\,, 
the intersection of the solution space to~\eqref{eq:solution} with
$\mathrm{RH}_-$ is either trivial or all of $\mathrm{RH}_-$\,. Hence, to finish the proof of our claim, 
it suffices to show that $C:=\sigma$ is not a solution to~\eqref{eq:solution}
unless $\tilde J=\pm\, I'$\,:

for this, recall that there exist $t\in\R$ and $w\in \C$ with $t^2+|w|^2=1$ 
such that the matrix of $\tilde J$ with respect to the Hermitian basis $\{e_1,I'e_1\}$ of $(W',\tilde I)$ 
is given by Eqs.~\eqref{eq:gij_1}-\eqref{eq:gij_2}. Clearly, $\sigma=\sigma^{-1}$ and $\sigma^*=\sigma$\,. Hence, if~\eqref{eq:solution} holds
for $C:=\sigma$\,, then $\sigma\circ\tilde J\circ \sigma=\tilde J$\,, i.e.
\begin{align}
\left ( \begin{array}{cc}
\i\,t&-\bar w\\
w&-\i\,t
\end{array}
\right )=
\left ( \begin{array}{cc}
-\i\,t&- w\\
\bar w& \i\,t
\end{array}
\right )\;. 
\end{align}
Thus $t=0$ and $w=\pm\, 1$\,, i.e. $\tilde J=\pm\, I'$\,. 

This proves our claim. Therefore, if $\tilde J\neq \pm\, I'$\,, then $\tilde A=-c\,\tilde J$\,. 
Hence $A|_V=-c\,\tilde J|_V+ a (J^N|_V-I'|_V)$\,.
It remains to show that $a (J^N|_V-I'|_V)=0$\,: 

we have $J^N|_W=I'|_W$ and $a (J^N|_V-I'|_V)=A|_V+c\,\tilde J|_V\in\so(V)_-$\,. 
Since $\so(V)_-\to\Hom(W,U), A\mapsto A|_W$ is an isomorphism,  $a (J^N|_V-I'|_V)=0$\,. This finishes our proof.
\end{proof}

\bigskip
\begin{corollary}\label{co:kis2}
In the situation of Lemma~\ref{le:kv2}, 
the curvature invariant pair $(W,U)$ is not integrable unless $V$ is a curvature invariant subspace of $T_pN$\,.
\end{corollary}
\begin{proof}
Note, if $\tilde J= \pm\, I'$\,, then $V:=W\oplus U$ is curvature invariant of Type $(c_2')$ defined by $(\Re,W',I')$\,.
Otherwise, if $\tilde J\neq \pm\, I'$\,, then we will show that $(W,U)$ is not integrable. 
Assume, by contradiction, that $(W,U)$ is integrable but $\tilde J\neq \pm\, I'$\,. Thus, 
there exists an integrable symmetric bilinear map 
$h:W\times W\to U$ such that $U=\Spann{h(x,y)}{x,y\in W}$\,.  Further,
let $\{I,J,K\}$ be a quaternionic basis of $\su(V,\tilde I)$ 
defined as follows: set $I|_W:=\tilde I|_W$\,, $I|_U:=-\tilde I|_U$\,, $J:=\tilde J|_V$ and $K:=I\circ J$\,.
Since $\tilde I$ commutes with $\tilde J$\,, we have $I\circ J=-J\circ I$ and then the usual
quaternionic relations hold, i.e. $I^2=J^2=K^2=-\Id$\,, $J\circ K=-K\circ J=I$
and  $K\circ I=-I\circ K=J$\,. We claim that $\fetth_x\in\{J,K\}_\R$ for all $x\in W$\,:

note, the set $\{\tilde I,I,J,K\}$ is a basis of $\fraku(V,\tilde I)$ and
$I,\tilde I\in\so(V)_+$ whereas $J,K\in\so(V)_-$\,, hence
\begin{equation}\label{eq:fraku_is_JK}
\fraku(V,\tilde I)\cap\so(V)_-=\{J,K\}_\R\;.
\end{equation}
Moreover, recall that $\frakh_W=\R\,\tilde I$ according to Lemma~\ref{le:type(tr_k')}.
Therefore, by virtue of Theorem~\ref{th:dec}, there exist $A_x\in\rho(\frakk_V)|_V\cap\so(V)_-$ 
and $B_x\in\fraku(V,\tilde I)\cap\so(V)_-$ with $\fetth_x=A_x+B_x$\,. Furthermore, since $\tilde J\neq \pm\, I'$\,,
we have  $\rho(\frakk_V)|_V\cap\so(V)_-=\R \tilde J$ as a consequence of Lemma~\ref{le:kv2}. 
Thus both $A_x$ and $B_x$ belong to $\fraku(V,\tilde I)$\,, 
hence $\fetth_x\in\fraku(V,\tilde I)\cap\so(V)_-$ for all $x\in W$\,, too.
The claim follows by means of~\eqref{eq:fraku_is_JK}.

Further, $\frakh_W$ acts on $W$ via $\so(W)$ which implies that 
$\fetth:W\to\so(V)_-$ is injective according to~\eqref{eq:Kernh},~\eqref{eq:Kernfetth} and  Proposition~\ref{p:invariant}.
Hence there exists some $x\in W$ with $\fetth_x=K$\,. Furthermore, set $y:=e_1-\i\, I'e_1$\,, $z:=e_2-\i\, I'e_2$ 
and recall that $R^N_{y,z}=-e_1\wedge e_2-I'e_1\wedge I'e_2=-\tilde I$\,. Therefore, 
with $x,y,z$ chosen as above, Eq.~\eqref{eq:cond2} with $k=1$ means that 
\begin{equation}\label{eq:[K,tilde I]}
[K,\tilde I]=-R^N(K\,y\wedge z+y\wedge K\,z)|_V\;.
\end{equation}
Note that l.h.s.\ of the last Equation vanishes. In order to evaluate r.h.s.\ of~\eqref{eq:[K,tilde I]}, 
note that $z=\tilde I\,y$\,, hence 
$$K\,y=I\,J\,y=-J\, I\,y=-\tilde I\,\tilde J\,y
=-\tilde J\,\tilde I\,y=-\tilde J\,z\;,$$ 
thus $z=J\,K\,y=-K\,J\,y$ and 
$K\,z=J\,y=\tilde J\,y$\,, which gives $$K\,y\wedge z+y\wedge K\,z=z\wedge \tilde J\,z+y\wedge \tilde J\,y\;.$$ 
Let $c\in\R$ and $A\in\so(\Re)$ be given such that
$R^N(K\,y\wedge z+y\wedge K\,z)=c\, J^N+A$\,. Using Eq.~\eqref{eq:crg}, the real part $A$ is given as follows,
\begin{align*}
A=\tilde J\,e_1\wedge e_1+\tilde J\, I'e_1\wedge I'e_1+\tilde J\,e_2\wedge e_2+\tilde J\,I'e_2\wedge I'e_2=-2\,\tilde J\;,
\end{align*}
(the last equality uses that $\{e_1,e_2,I'e_1,I'e_2\}$ is an orthonormal basis of $W'$ and that $\tilde J\in\so(W')$).
Therefore, since r.h.s.\ of~\eqref{eq:[K,tilde I]} vanishes, we conclude that $2\,\tilde J=c\, J^N$ on $V$\,.
Hence $c=\pm 2$ (since both $\tilde J$ and $J^N|_{\C\,W'}$ are isometries of $\C\,W'$), i.e. $\tilde J=\pm\, J^N$ on $V$\,. In particular, 
$\mp\, J^N+\tilde J$ vanishes on $W$\,. With $B:=\tilde J\mp\, I'$ and $a:=\mp$\,, Lemma~\ref{le:type(tr_k')}~(c) 
in combination with~\eqref{eq:I'=pmi} implies that $B$ vanishes identically on $W'$\,. This shows that $\tilde J=\pm\, I'$\,, 
a contradiction.
\end{proof}

\paragraph{Type $\mathbf{(tr_k',tr_1)}$}
Let $(W,U)$ be an integrable orthogonal curvature
invariant pair of Types $(tr_k',tr_1)$\,. Since the action of $\frakh_W$ on $W$ is given by $\so(W)$ (see
Lemma~\ref{le:type(tr_k')}~(a)), Proposition~\ref{p:sph} shows that here the linear space
$W\oplus U$ is curvature invariant.\qed

\paragraph{Type $\mathbf{(ex_3,tr_1)}$} Let $W$ and $U$ be  of Types $(ex_3)$ and $(tr_1)$ defined by 
the data $(\Re,\{e_1,e_2\})$ and a unit vector $u\in T_pN$\,,  respectively. 
Then $u=\pm\, 1/\sqrt 2 (e_2-\i\,e_1)$ and  the linear space $W\oplus U$ is 
curvature invariant of Type $(c_2)$  defined by the data $(\Re,\{e_1,e_2\}_\R)$\,.\qed

\paragraph{}ACKNOWLEDGMENTS.\ \ \ The author
wants to point out that his understanding of parallel 
submanifolds with curvature isotropic tangent spaces of maximal possible dimension greatly 
benefited from a conversation with Jost Eschenburg from Augsburg.

\bibliographystyle{amsplain}

\begin{thebibliography}{999}

\bibitem{BCO}
J.~Berndt, S.~Console, C.~Olmos:
\newblock Submanifolds and holonomy,
\newblock Research Notes in Mathematics, Chapman\&Hall 434, 2003.

\bibitem{BR}
E.~Backes, H.~Reckziegel:
\newblock {\em On Symmetric Submanifolds of Spaces of Constant Curvature},
\newblock Math. Ann. {\bf 263}  (1983), 421-433.

\bibitem{chen-naganoI} B.-Y.~Chen, T.~Nagano: 
\newblock {\em Totally geodesic submanifolds of symmetric spaces I},
\newblock Duke Math. J. {\bf 44, No 4}  (1977), 745-755.

\bibitem{chen-naganoII} B.-Y.~Chen, T.~Nagano: 
\newblock {\em Totally geodesic submanifolds of symmetric spaces II},
\newblock Duke Math. J. {\bf 45, No 2}  (1978), 405-425.


\bibitem{D}
P.~Dombrowski:
\newblock {\em Differentiable maps into Riemannian manifolds of constant stable osculating rank, part 1},
\newblock J. Reine Angew. Math. {\bf 274/275}  (1975), 310-341.

\bibitem{Fe1}
D.~Ferus:
\newblock {\em Immersions with parallel second fundamental form},
\newblock Math. Z. {\bf 140} (1974), 87-93.

\bibitem{Fe2}
D.~Ferus:
\newblock {\em Symmetric submanifolds of euclidian space},
\newblock Math. Ann. {\bf 247} (1980), 81-93.

 \bibitem{FP}
 D.~Ferus,\;F.~Pedit:
 \newblock {\em Curved flats in symmetric spaces},
 \newblock manuscripta math. {\bf No 91} (1996), 445-454.

\bibitem{He}
S.~Helgason:
\newblock Differential Geometry, Lie Groups and Symmetric Spaces,
\newblock  American Mathematical Society Vol 34, 2001.

\bibitem{J1}
T.~Jentsch:
\newblock {\em The extrinsic holonomy Lie algebra of a parallel submanifold},
\newblock  Ann. Global Anal. Geom. {\bf 38} (2010), 335-371.

\bibitem{J2}
T.~Jentsch:
\newblock {\em Extrinsic homogeneity of parallel submanifolds},
\newblock manuscripta mathematica, {\bf DOI: 10.1007/s00229-011-0469-2} (2011).

\bibitem{J3}
T.~Jentsch:
\newblock {\em Extrinsic homogeneity of parallel submanifolds II},
\newblock Differential Geometry and its Applications {\bf 29} (2011), 214-232.

\bibitem{JR}
T.~Jentsch, H.~Reckziegel:
\newblock {\em Submanifolds with parallel second fundamental form studied via the Gau{\ss} map},
\newblock Annals of Global Analysis and Geometry {\bf 29} (2006), 51-93.




\bibitem{sebastian1} S.~Klein:
\newblock {\em Totally geodesic submanifolds of the complex quadric},
\newblock  Differential Geom. Appl. {\bf 26}  (2008), 79-96.

\bibitem{sebastian2}  S.~Klein:
\newblock {\em Totally geodesic submanifolds of the complex and the quaternionic 2-Grassmannians},
\newblock  Trans. Amer. Math. Soc. {\bf 361, No. 9}  (2009),  4927-4967. 

\bibitem{sebastian3} S.~Klein: 
\newblock {\em Reconstructing the geometric structure of a Riemannian symmetric space from its Satake diagram},
\newblock Geom. Dedicata {\bf 138}  (2009), 25-50.

\bibitem{sebastian4}  S.~Klein:
\newblock {\em Totally geodesic submanifolds in Riemannian symmetric spaces}
\newblock  Differential Geometry: Proceedings of the VIII International
Colloqium, World Scientific (2009), 136-145.

\bibitem{sebastian5}  S.~Klein:
\newblock {\em Totally geodesic submanifolds of the exceptional Riemannian symmetric spaces of rank 2},
\newblock  Osaka J. Math. {\bf 47, No 4}  (2010), 1077-1157.

\bibitem{N3}
H.~Naitoh:
\newblock {\em Isotropic Submanifolds with parallel
second fundamental form in $\rmP^m(c)$},
\newblock Osaka J. Math. {\bf 18, No 2}  (1981), 427-464.

\bibitem{N1}
H.~Naitoh:
\newblock {\em Symmetric submanifolds of compact symmetric spaces},
\newblock  Tsukuba J.Math. {\bf 10, No 4}  (1986), 215-242.

\bibitem{N2}
H.~Naitoh:
\newblock {\em Grassmann geometries on compact symmetric spaces of general type},
\newblock  J. Math. Soc. Japan {\bf 50}  (1998), 557-592.

\bibitem{NY}
K.~Nomizu, K.~Yano:
\newblock {\em On circles and spheres in Riemannian geometry},
\newblock  Math. Ann. {\bf 210}  (1974), 163-170.

\bibitem{Str}
W. Str{\"u}bing:
\newblock {\em Symmetric submanifolds of Riemannian manifolds},
\newblock Math. Ann. {\bf 245}  (1979), 37-44.

\bibitem{Ta}
M.~Takeuchi:
\newblock {\em Parallel submanifolds of space forms, Manifolds and Lie groups},
\newblock In honor of Yozo Matsushima (J. hano et al., eds), Birkh{\"a}user, Boston, 429-447, 1981.

\end{thebibliography}

\end{document}